\documentclass[letterpaper,11pt,oneside,reqno]{amsart}

\usepackage{amsfonts,amsmath, amssymb,amsthm,amscd}
\usepackage{bm}
\usepackage{lmodern}
\usepackage{mathrsfs}
\usepackage{yfonts}
\usepackage{verbatim}
\usepackage{hyperref}
\usepackage{graphicx}
\usepackage[utf8]{inputenc}
\usepackage{latexsym}
\usepackage{upgreek}
\usepackage{relsize}
\usepackage{xcolor}
\usepackage{mathdots}
\usepackage{mathtools}
\usepackage{dsfont}
\usepackage[margin=1in,footskip=.25in]{geometry}
\usepackage{tikz}
\usepackage{enumitem}

\numberwithin{equation}{section}

\renewcommand{\emph}[1]{\textsf{\textit{#1}}}

\usepackage[width=.9\textwidth]{caption}

\let\oldtocsection=\tocsection
\let\oldtocsubsection=\tocsubsection
\renewcommand{\tocsection}[2]{\hspace{0em}\oldtocsection{#1}{#2}}
\renewcommand{\tocsubsection}[2]{\hspace{2em}\oldtocsubsection{#1}{#2}}

\begin{document}
\fontdimen8\textfont3=0.5pt  


\def \Ai {{\rm Ai}}
\def \Pf {{\rm Pf}}
\def \sgn {{\rm sgn}}
\newcommand{\EE}{\ensuremath{\mathbb{E}}}
\newcommand{\Var}{\mathrm{Var}}
\newcommand{\qq}[1]{(q;q)_{#1}}
\newcommand{\PP}{\ensuremath{\mathbb{P}}}
\newcommand{\R}{\ensuremath{\mathbb{R}}}
\newcommand{\C}{\ensuremath{\mathbb{C}}}
\newcommand{\Z}{\ensuremath{\mathbb{Z}}}
\newcommand{\N}{\ensuremath{\mathbb{N}}}
\newcommand{\Q}{\ensuremath{\mathbb{Q}}}
\newcommand{\T}{\ensuremath{\mathbb{T}}}
\newcommand{\Y}{\ensuremath{\mathbb{Y}}}
\newcommand{\I}{\ensuremath{\mathbf{i}}}
\newcommand{\Real}{\ensuremath{\mathfrak{Re}}}
\newcommand{\Imag}{\ensuremath{\mathfrak{Im}}}
\newcommand{\subs}{\ensuremath{\mathbf{Subs}}}
\newcommand{\Sym}{\ensuremath{\mathsf{Sym}}}
\newcommand{\dist}{\textrm{dist}}
\newcommand{\Res}[1]{\underset{{#1}}{\mathbf{Res}}}
\newcommand{\Resfrac}[1]{\mathbf{Res}_{{#1}}}
\newcommand{\Sub}[1]{\underset{{#1}}{\mathbf{Sub}}}
\newcommand{\la}{\lambda}
\newcommand{\ta}{\theta}
\newcommand{\labold}{\boldsymbol{\uplambda}}
\def\note#1{\textup{\textsf{\color{blue}(#1)}}}
\def\noteg#1{\textup{\textsf{\color{magenta}(#1)}}}
\newcommand{\oldrho}{\rho}
\renewcommand{\rho}{\varrho}
\newcommand{\e}{\epsilon}
\newcommand{\eps}{\varepsilon}
\renewcommand{\le}{\leqslant}
\renewcommand{\leq}{\leqslant}
\renewcommand{\ge}{\geqslant}
\renewcommand{\geq}{\geqslant}
\newcommand{\ZZ}{\mathcal{Z}}

\newcommand{\wweight}{\varpi}
\newcommand{\omegaweight}{\omega}

\newcommand{\HH}{\mathcal{H}}
\newcommand{\PC}{\mathcal{P}}
\newcommand{\Cfive}{\mathcal{C}}
\newcommand{\Dfive}{\mathcal{D}}
\newcommand{\var}{\mathrm{var}}
\newcommand\tinyvert{\vcenter{\hbox{\scalebox{0.5}{$\boldsymbol{|}$}}}}
\newcommand\tinydiagup{\vcenter{\hbox{\scalebox{0.5}{$\boldsymbol{\diagup}$}}}}
\newcommand\nint{\mathop{\mathrlap{\,\, n}\int}}
\newcommand\murand{\hat{\mu}^{(n)}}
\newcommand{\Zpolymer}{z}
\newcommand{\zd}{\Zpolymer}
\newcommand{\hd}{h}
\newcommand{\X}{\mathcal{X}}
\newcommand{\SSS}{\mathcal{S}}
\newcommand{\Gammainv}{\mathrm{Gamma}^{-1}}


\newcommand{\diag}{\alpha_{\circ}}


\usetikzlibrary{patterns}
\usetikzlibrary{arrows}

\newcommand{\varnint}{\raisebox{-0.4cm}{\begin{tikzpicture}[scale=0.43]
			\draw[thick] (-0.4,-1) -- (0,-1) -- (0,1) -- (0.4,1);
			\draw (-0.3, -0.7) node{\footnotesize $n$}; 
	\end{tikzpicture}}\hspace{-0.1cm}}


\newtheorem{theorem}{Theorem}[section]
\newtheorem{conjecture}[theorem]{Conjecture}
\newtheorem{lemma}[theorem]{Lemma}
\newtheorem{proposition}[theorem]{Proposition}
\newtheorem{corollary}[theorem]{Corollary}

\newtheorem{theoremintro}{Theorem}
\renewcommand*{\thetheoremintro}{\Alph{theoremintro}}

\theoremstyle{definition}
\newtheorem{remark}[theorem]{Remark}

\theoremstyle{definition}
\newtheorem{example}[theorem]{Example}

\theoremstyle{definition}
\newtheorem{definition}[theorem]{Definition}

\theoremstyle{definition}
\newtheorem{definitions}[theorem]{Definitions}


\title{\large Stationary measures for the log-gamma polymer and KPZ equation in half-space}

\author[G. Barraquand]{Guillaume Barraquand}
\address{G. Barraquand,
Laboratoire de Physique de l’\'Ecole normale supérieure, ENS, Université PSL, CNRS, Sorbonne Université, Université Paris Cité, F-75005 Paris, France}
\email{guillaume.barraquand@ens.fr}
\author[I. Corwin]{Ivan Corwin}
\address{I. Corwin, Department of Mathematics, Columbia University, New York, NY 10027, USA}
\email{ivan.corwin@gmail.com}

\begin{abstract}	
We construct explicit one-parameter families of stationary measures for the Kardar-Parisi-Zhang equation in half-space with Neumann boundary conditions at the origin, as well as for the log-gamma polymer model in a half-space. The stationary measures are stochastic processes that depend on the boundary condition as well as a parameter related to the drift at infinity. They are expressed in terms of  exponential functionals of Brownian motions and gamma random walks. We conjecture that these constitute all extremal stationary measures for these models.
The log-gamma polymer result is proved through a symmetry argument related to half-space Whittaker processes which we expect may be applicable to other integrable models. The KPZ result comes as an intermediate disorder limit of the log-gamma polymer result and confirms the conjectural description of these stationary measures from \cite{barraquand2021steady}. To prove the intermediate disorder limit, we provide a general half-space polymer convergence framework that extends works of  \cite{wu2018intermediate,parekh2019positive,alberts2014intermediate}.
\end{abstract}

\maketitle

%


\section{Introduction}
\label{sec:introduction}
In a recent cycle of work  \cite{corwin2021stationary, bryc2021markov, barraquand2021steady}, the stationary measure for the open KPZ equation was determined explicitly. Remarkably, this stationary measure takes the form of an exponentially reweighted Brownian motion subject to a linear potential on the starting and ending heights -- a measure that was first studied in \cite{hariya2004limiting} and which is quite reminiscent of the two-level semi-discrete polymer model \cite{o2002representation, oconnell2001brownian}. The simple and elegant form of this stationary measure came as a surprise, especially since it was derived from a considerably more complicated set of formulas for ASEP. This raises the question of understanding the origin of this form of the stationary measure from a structural perspective.

It is with this in mind that we set out to study the stationary measure for the log-gamma polymer in a half-space. We work here in the half-space geometry, as opposed to a strip (as would be necessary to reach open KPZ through a limit), for two reasons. The first is that in the half-space, there is already work \cite{o2014geometric, barraquand2018half} elucidating some algebraic structure related to the log-gamma polymer, namely half-space Whittaker processes. This structure will prove key to our construction and description of the half-space log-gamma polymer stationary measures. The second is that we wish to prove the formulas conjectured in \cite{barraquand2021steady} for the half-space KPZ equation stationary measures. While, with considerable work, we expect that the approach used in \cite{barraquand2021steady} could be made rigorous, the approach we provide here is quite simple and of a different nature -- relying on half-space Whittaker processes instead of matrix product ansatz.

\subsection{Half-space KPZ equation}
\label{sec:introKPZ}
The half-space KPZ equation models stochastic interface growth in contact with a boundary \cite{kardar1986dynamic,Ito2018}). The equation, subject to Neumann boundary condition with parameter $u\in \R$ (KPZ$_u$ for short) is
\begin{align}
	\label{eq:KPZ}\tag{KPZ$_u$}
	\begin{split}
		\partial_T \HH_u(T,X)&=\frac{1}{2} \partial^2_{X}\HH_u(T,X) +\frac{1}{2} \left( \partial_X \HH_u(T,X)\right)^2 + \xi(T,X),\\ \partial_X \HH_u(T,X)\big\vert_{X=0} &= u,
	\end{split}
\end{align}
with space-time white noise $\xi$. It is defined for $T,X\in \R_{\geq 0}$ by the Hopf-Cole transform\footnote{References \cite{corwin2016open, parekh2017kpz} adopt the convention that a solution to \eqref{eq:KPZ} is the logarithm of a solution to \eqref{eq:SHEintro} with the arbitrary convention that $\mu=u$. Following \cite{corwin2021stationary}, we adopt the different convention that $\mu=u-1/2$. This new convention ensures that the boundary condition in \eqref{eq:KPZ} is satisfied by the expectation of $\HH(T,X)$. It also makes the phase diagram in Fig. \ref{fig:phasediagram} more symmetric.} $\HH_u(T,X):= \log \ZZ_{u-\frac{1}{2}}(T,X)$, where $\ZZ_{\mu}$ solves the half-line stochastic heat equation (SHE) with Robin boundary condition $\mu\in \R$,
\begin{align}
	\label{eq:SHEintro}
	\tag{SHE$_{\mu}$}
	\begin{split}
		\partial_T \ZZ_{\mu}(T,X) &= \frac{1}{2} \partial^2_{X}\ZZ_{\mu}(T,X) + \ZZ_{\mu}(T,X) \xi(T,X), \\ \partial_X\ZZ_{\mu}(T,X)\big\vert_{X=0} &=\mu \ZZ_{\mu}(T,0),
	\end{split}
\end{align}
with initial data given by  $\ZZ_{\mu}(0,X)=e^{ \HH(0,X)}$ for $X\in \R_{\geq 0}$, where $\HH(0,X)$ is initial data for \eqref{eq:KPZ}.
Further discussion and a proper definition of \eqref{eq:SHEintro} is deferred to Section \ref{sec:notanddefs}.

In this article we explicitly construct stationary measures for this model. We conjecture that these constitute the entire set of extremal stationary measures and give a phase diagram for how they should arise in the long-time limit for general initial data.
While the KPZ height function $\HH_u(T,X)$ does not have a stationary (in $T$) probability measure, its increment process $X\mapsto \HH_u(T,X)-\HH_u(T,0)$ should. We will say that a random function $\HH(X)$ is stationary for \eqref{eq:KPZ} if the solution to \eqref{eq:KPZ} with initial data $\HH(0,X)=\HH(X)$ satisfies that for all $T\in \R_{>0}$, $\HH(T,X)-\HH(T,0)$ is equal to $\HH(X)$ in law, as a process in $X\in \R_{\geq 0}$. The law of such an $\HH(X)$ will be called a stationary probability measure for \eqref{eq:KPZ}.

The spatial derivative $\partial_X \HH(T,X)$ solves the stochastic Burgers equation which is a continuum version of an interacting particle system, modeling stochastic mass transport. In that context, the boundary condition imposes a reservoir for the creation and destruction of mass. If $\HH(X)$ is stationary for the KPZ equation then $\partial_X \HH(X)$ (which is a generalized function) will be stationary for the stochastic Burgers equation and should describe the long-time density profile of the system.

The exponential $e^{\HH(T,X)}$ solves the SHE, as noted above. The SHE can be interpreted as a continuum version of the parabolic Anderson model which models the evolution of a large mass of particles subject to diffusion, as well as duplication and destruction (based on the sign and intensity of the white noise) \cite{carmonaMolchanov,10.1214/13-AAP944}. The solution to \eqref{eq:SHEintro} also gives the partition function for a continuum  directed random polymer model involving Brownian motion paths moving in a disordered environment and subject to attraction/repulsion (depending on whether $\mu$ is negative or positive) at the origin. Such models have been used to study the effect of boundaries on pinning polymer paths \cite{kardar1985scaling,10.1214/105051606000000015,caravenna2008pinning,Cometsbook, denardis2020delta, barraquand2021kardar}. In this context, the polymer endpoint probability density takes the form $e^{\HH(T,X)-\HH(T,0)}/\int_0^{\infty} e^{\HH(T,X)-\HH(T,0)}dX$. Assuming a finite  denominator, the KPZ equation stationary measures should yield the stationary polymer endpoint density subject to  pinning.

\subsubsection{Stationary measures} We turn to formulate our first main result -- the description of a one parameter family of stationary measures for \eqref{eq:KPZ} for each choice of boundary parameter $u\in \R$. We parameterize these by $v\in \R$ such that $v\leqslant \min\lbrace 0,u\rbrace$, and denote the corresponding random function by $\HH_{u,v}(X)$.  We will define these processes via exponential transforms of Brownian motions and an inverse gamma random variable.

\begin{definition}
	$X\sim \Gammainv(\theta)$ is an inverse-gamma random variable with parameter $\theta>0$ if its distribution is supported on $\mathbb R_{>0}$ with density against Lebesgue
	$$\frac{1}{\Gamma(\theta)} x^{-\theta-1}e^{-1/x}.$$
	We note the following moment formulas:
	\begin{equation}\label{eq:inversegammamoments}
		\EE[X] = \frac{1}{\theta-1},\quad \var(X) = \frac{1}{(\theta-1)^2(\theta-2)},\quad \EE[X^k] = \frac{1}{(\theta-1)\cdots (\theta-k)},
	\end{equation}
	where we assume $\theta>k\in \Z_{\geq 1}$ for the corresponding moment to be finite. We also note
	\begin{equation}\label{eq:loggamma}
		\mathbb E[\log X]= -\psi(\theta), \quad \var(\log X) = \psi'(\theta),
	\end{equation}
	where $\psi$ denotes the digamma function $\psi(z)=\partial_z \log\Gamma(z)$ and $\psi'$ the trigamma function.
\end{definition}

\begin{definition}
	For $u\in \R$ and $v\leqslant \min\lbrace 0,u\rbrace$, define the process $\HH_{u,v}(X)$ as
	\begin{equation}
		\HH_{u,v}(X) :=\log\left( \frac{1}{\wweight}\int_{0}^X dS e^{B_1(S)+B_2(X)-B_2(S)}+ e^{B_2(X)}\right),
		\label{eq:defhuv}
	\end{equation}
	where $B_1$ and $B_2$ are independent standard Brownian motions with drifts $-v$ and $v$ (resp.), and $\wweight\sim\Gammainv(u-v)$, independent from $B_1$ and $B_2$. When $u=v$,  we define $\frac{1}{\wweight}=0$.
\end{definition}
\begin{remark}\label{rem:specialcases}
	Let us mention a few important properties of the process $\HH_{u,v}$:
	\begin{itemize}%
		\item For general $u,v$,  spatial increments of $\HH_{u,v}$ are not independent or translation invariant.
		\item For $u=v\leq 0$, $\HH_{u,v}=B_2$ is Brownian motion with drift $u=v$ (follows immediately).
		\item For $u=-v\geq 0$, $\HH_{u,v}$ is Brownian motion with drift $u=-v$ (see Lemma \ref{lemma:Brownianity}).
		\item The parameter $v$ controls the drift at infinity, in the sense that $\lim_{x\to\infty} \frac{\HH_{u,v}(x)}{x} =-v$ a.s.
		\item The definition of the process $\HH_{u,v}$ in \eqref{eq:defhuv} makes perfect sense for $v>0$ (provided $u\geq v$) in which case the processes $\HH_{u,v}$ and $ \HH_{u,-v}$ have the same distribution (see Lemma \ref{lemma:symmetryHuv}). Thus, we may restrict the sign of $v$. It will be more convenient to choose $v\leqslant 0$ in order to be consistent with earlier literature where the process \eqref{eq:defhuv} has appeared \cite{barraquand2021steady, bryc2021markov2}.
	\end{itemize}
\end{remark}

Our first result is the description of the  half-space KPZ equation stationary measures.

\begin{theorem}
	For $u,v\in\R$ with $v\leqslant \min\lbrace 0,u\rbrace$,  $\HH_{u,v}$ is stationary for \eqref{eq:KPZ}.
	\label{theo:invariantKPZintro}
\end{theorem}

Before discussing the relation between this theorem and previous work, we record:

\begin{conjecture}[\cite{barraquand2021steady}]\label{conj:ergodictheory}
	$\{\HH_{u,v}\}_{v\leq \min\lbrace 0,u\rbrace}$ constitutes all extremal stationary measures for \eqref{eq:KPZ} and for any initial data $\HH(0,X)$ with drift $-v$ as $X\to \infty$ we have the following.
	\begin{enumerate}%
		\item For $u\geq 0$: If $v\leq 0$,  the process $\HH(T,\cdot)-\HH(T,0)$ converges in distribution as $T\to \infty$ to the process $\HH_{u,v}(\cdot)$; if $v\geq 0$,  it converges to the process $\HH_{u,0}(\cdot)$.
		\item For $u\leq 0$:  If $v\leq u$    the process $\HH(T,\cdot)-\HH(T,0)$ converges in distribution as $T\to \infty$ to the process $\HH_{u,v}$; if $v\geq u$,  it converges to $\HH_{u,u}$, that is a Brownian motion with drift $u$.
	\end{enumerate}
\end{conjecture}
Some additional conditions besides having drift $-v$ may be needed in this conjecture. The phase diagram that is conjectured is illustrated in Figure \ref{fig:phasediagram}.
In the limit where the drift $-v$ goes to $-\infty$, the sequence of initial datum $\HH(0,X)=-vX+\log(v)$ approximates of narrow wedge initial data where $e^{\HH(0,X)}=\delta(X)$. In this case the conjecture suggests that $\HH(T,\cdot)$ weakly converges  as $T\to \infty$ to a Brownian motion with drift $u$ when $u\leq 0$ and to $\HH_{u,0}$ when $u\geq 0$. Similarly, for flat initial data $\HH(0,\cdot)\equiv  0$ the process $\HH(T,\cdot)$ should weakly converge  as $T\to\infty$ to a Brownian motion with drift $u$ when $u\leq 0$ and to $\HH_{u,0}$ when $u\geq 0$.

\begin{figure}
	\begin{tikzpicture}[scale=0.65]
		\fill[blue, opacity=0.1] (0,0) -- (4,0) -- (4,4) -- (0,4) -- cycle;
		\fill[red, opacity=0.1] (0,0) -- (4,0) -- (4,-3.5) -- (-3.5,-3.5) -- cycle;
		\fill[green, opacity=0.1] (0,0) -- (0,4) -- (-3.5,4) -- (-3.5,-3.5) -- cycle;
		\draw[ultra thick, -stealth] (0,0) -- (4,0) node[anchor=north]{$u$};
		\draw[ultra thick, -stealth] (0,0) -- (0,4) node[anchor=east]{$v$};
		\draw[ultra thick] (0,0) -- (-3.5,-3.5);
		\draw[thick, dashed, gray] (-3.5,3.5) -- (3.5,-3.5);
		\draw[dashed, gray] (0,0) -- (-3,0) node[anchor=east]{$0$};
		\draw[dashed, gray] (0,0) -- (0,-3) node[anchor=north]{$0$};
		\draw (2,2) node{$\HH_{u,0}(X)$};
		\draw (1.1,-2) node{$\HH_{u,v}(X)$};
		\draw (-2,0.5) node{$B(X)+uX $};
	\end{tikzpicture}
	\caption{Fix $u,v\in \mathbb R$. Conjecture \ref{conj:ergodictheory} claims that the solution to \eqref{eq:KPZ} converges at large time to one of the spatial process $\HH_{u,v}$, $\HH_{u,0}$ or a Brownian motion $B(X)+uX$, according to where $(u,v)$ lies in this diagram (see also \cite[Fig. 2]{barraquand2021steady}). The same diagram also describes the large $L$ limits of the measures $\mathcal P^L_{u,v}$ for different values of $(u,v)$, obtained in \cite{hariya2004limiting}. Along the antidiagonal line $u+v=0$, $\lim_{L\to\infty} \mathcal P_{u,v}^L$ is always Brownian motion with drift $u=-v$.}
	\label{fig:phasediagram}
\end{figure}
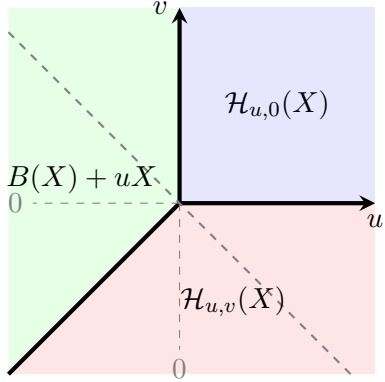

\subsubsection{Relation to previous work}
For the full-line KPZ equation ($X\in \R$) Brownian motion with arbitrary drift is stationary. This was known in the physics literature since \cite{forster1977large} and proved in \cite[Proposition B.2]{bertini1997stochastic} based on the convergence of the height function of the  weakly asymmetric simple exclusion process (WASEP) to the KPZ equation together with the fact that Bernoulli product measures are invariant under ASEP dynamics (see \cite{funaki2015kpz} for a different proof).
For the KPZ equation with periodic boundary conditions, the Brownian bridge measure is the unique stationary measure, see for instance \cite{hairer2018strong,GP18}.

Boundaries in the half-space KPZ equation or open KPZ equation on  $[0,L]$ generically break the Gaussian nature of the stationary measure. The (conjecturally unique) stationary measure for the open KPZ equation on $[0,1]$ with Neumann boundary conditions $\partial_X \HH(T,X)\vert_{X=0} =u$ and $\partial_X \HH(T,X)\vert_{X=1} =-v$ was recently constructed in \cite{corwin2021stationary} and characterized when $u+v\geq 0$ therein through Laplace transform formulas. That work built on descriptions of the invariant measures for open ASEP on a segment \cite{derrida1993exact, uchiyama2004asymmetric, bryc2017asymmetric} and the convergence of WASEP on a segment to the open KPZ equation \cite{corwin2016open,parekh2017kpz} (see also the review \cite{Corwin2022survey}). The Laplace transform formula in \cite{corwin2021stationary} was then inverted in \cite{bryc2021markov} and the physics article \cite{barraquand2021steady} (see also \cite{bryc2021markov2}), yielding the following simple description for the stationary measure.

For $L\in \R_{>0}$ and $u,v\in \mathbb R$, define a probability measure $\mathcal P^L_{u,v}$ on $C([0,L])$,  describing the law of the process $ X\mapsto W(X)+Y(X)$,
where $W$ is a standard Brownian motion on $[0,L]$ with variance $1/2$, and $Y$ is a process, independent from $W$, whose distribution $\mathbb P_Y$ is absolutely continuous with respect to the Brownian measure with variance $1/2$, denoted $\mathbb P_B$, with Radon-Nikodym derivative
\begin{equation}
	\frac{d\mathbb P_{Y}}{d\mathbb P_B}(B) = \frac{1}{\mathcal Z_{u,v}} e^{-2vB(L)}\left( \int_0^L e^{-2B(s)}ds\right)^{-u-v},
	\label{eq:RND}
\end{equation}
where $\mathcal Z_{u,v}$ is a normalization constant (computed exactly in \cite[Prop. 4.1]{donati2000striking} for $u,v>0$). 
The process $Y$ was first studied in \cite{hariya2004limiting}, motivated by the Matsumoto-Yor   identity \cite{matsumoto2001relationship, matsumoto2005exponential}  which involves similar reweighting of the Brownian measure by exponential functionals.
When $u+v>0$, the stationary process introduced in \cite{corwin2021stationary} is exactly distributed according to the probability measure $\mathcal P^1_{u,v}$. It was conjectured in \cite{barraquand2021steady} that $\mathcal P^L_{u,v}$ is a stationary measure for the KPZ equation on $[0,L]$ for any $L$, and any $u,v\in \mathbb R$.

The description of the large $L$ limit of $\mathcal P_{u,v}^L$ is one of the main results of \cite{hariya2004limiting}. A phase transition phenomenon emerges and the limit depends on whether $(u,v)$ belongs to three regions indicated in the phase diagram in  Figure \ref{fig:phasediagram}. Specifically, \cite{hariya2004limiting} proved that as $L\to \infty$,
\begin{enumerate}%
	\item For $u\geq 0,v\geq 0$ (Figure \ref{fig:phasediagram} blue region), $\mathcal P_{u,v}^L$ weakly converges to the distribution of $\HH_{u,0}$;
	\item For $u\leq 0,v\geq u$ (Figure \ref{fig:phasediagram} green region), $\mathcal P_{u,v}^L$ weakly converges to the distribution of $ \HH_{u,u}$, that is a standard Brownian motion with drift $u$;
	\item For $v\leq 0, u\geq v$ (Figure \ref{fig:phasediagram} red region) , $\mathcal P_{u,v}^L$ weakly converges to the distribution of $\HH_{u,v}$.
\end{enumerate}
It was also conjectured in \cite{barraquand2021steady} that the large $L$ limits of the measures $\mathcal P^L_{u,v}$ would be stationary for \ref{eq:KPZ}. Theorem \ref{theo:invariantKPZintro} confirms this conjecture, though we proceed in a rather different manner, working from the start with a half-space model -- the log-gamma polymer. It would be interesting to see if the approach suggested in \cite{barraquand2021steady} can be made rigorous -- first extending the results of \cite{corwin2021stationary} to identify the stationary measure of the open KPZ equation on the interval $[0,L]$ and then  verifying that the limit of those stationary measures as $L\to \infty$ are also stationary for \ref{eq:KPZ}. That approach, however, could only yield  Theorem \ref{theo:invariantKPZintro} subject to the restriction $u+v\geq 0$ (getting around this in the approach used in \cite{corwin2021stationary}  seems quite difficult). One advantage of the approach we develop here is that such a restriction is not present. Indeed, conjecturally we are able to access all of the half-space KPZ stationary measures.

Prior to this recent cycle of work, it was clear that Brownian motion with drift $u$ should be stationary  for \eqref{eq:KPZ}. This (which is a special case of Theorem \ref{theo:invariantKPZintro}) can be shown in the spirit of \cite[Proposition B.2]{bertini1997stochastic} by combining a special case (known from \cite{liggett1975ergodic}) of the half-space ASEP invariant measure which is of product form with the convergence result of \cite{corwin2016open}. Another approach to this Brownian case was described in  \cite{barraquand2020half} and relies on identifying a special half-space log-gamma polymer stationary measure in which the free energy increments are also of product form. This product measure is a special case of our one-parameter family of the log-gamma polymer stationary measures that we identify in Section \ref{sec:loggammaintro}.

Returning to $\HH_{u,v}$,  setting
$ B_1 = \beta_1 - \beta_2$ and $B_2=\beta_1+\beta_2$ (so $\beta_1$ and $\beta_2$ are independent Brownian motions with diffusion coefficient $1/2$ and drifts resp. $0$ and $v$) we see that
\begin{equation}
	\HH_{u,v}(x) =\beta_1(x)+\beta_2(x) +\log\left( 1+\frac{1}{\wweight}\int_{0}^x ds e^{- 2 \beta_2(s)} \right).
	\label{eq:geometricPitman}
\end{equation}
This formula for $\HH_{u,v}$ as a sum of $\beta_1$ and the geometric Pitman transform of $\beta_2$ appears in \cite{barraquand2021steady} where $\HH_{u,v}$ is matched to the main object of study in  \cite{hariya2004limiting} (the equivalence between \eqref{eq:geometricPitman} and \eqref{eq:defhuv} is also used in \cite{barraquand2021half} to study the distribution of \eqref{eq:KPZ} from $\HH_{u,v}$ initial data).

Conjecture \ref{conj:ergodictheory} is a continuum analog of the half-line  ASEP ergodic theorem \cite[Theorem 1.8]{liggett1975ergodic}, in the sense that the WASEP height function converges to a solution of \eqref{eq:KPZ} \cite{corwin2016open, parekh2017kpz}, and the phase diagram in \cite{liggett1975ergodic} becomes that of Conjecture \ref{conj:ergodictheory}. More precisely,  \cite[Theorem 1.7]{liggett1975ergodic} proves the existence (and some asymptotic properties as well, though not an explicit description) of invariant measures for half-space ASEP. These measures are denoted, therein, by $\mu(\lambda, \rho)$ where $\lambda$ is a boundary parameter (the density imposed by the reservoir at the boundary) and $\rho$ corresponds to the average density of particles at infinity under the measure $\mu(\lambda, \rho)$. The ASEP ergodic theorem \cite[Theorem 1.8]{liggett1975ergodic} assumes ASEP initial data that is  product Bernoulli with an asymptotic density at infinity (though probably such a theorem holds for more general initial data provided an asymptotic density at infinity). The approach used in \cite{liggett1975ergodic} relies heavily on the fact that half-space ASEP is approximated by ASEP on an interval (i.e., open ASEP), which itself is a finite-state space Markov process. It is unclear whether this approach works in the  setting of Conjecture \ref{conj:ergodictheory}.

\subsubsection{Ideas in the proof of Theorem \ref{theo:invariantKPZintro}}
\label{sec:KPZproofsketch}
We study a discrete analogue of  \eqref{eq:KPZ}, the half-space (or sometimes called octant) log-gamma discrete directed polymer model and characterize its stationary measures in Section \ref{sec:loggammaintro}. Having established this discrete result (Theorem \ref{theo:loggammaintro}), Theorem \ref{theo:invariantKPZintro} follows by way of a scaling limit (Theorem \ref{prop:KPZconvergence}).  The convergence of half-space discrete polymer partition functions to \eqref{eq:SHEintro} has been considered before in \cite{wu2018intermediate, parekh2019positive}. Those works either assume degenerate initial data (i.e. point to point polymers) or degenerate boundary conditions (asymptotically Dirichlet). Since we needed to deal with convergence for general initial data and Neumann boundary conditions, we provided a general set of results (Theorem \ref{thm:SHEsheetlimit}) for half-space SHE limits for intermediate disorder limits of half-space polymer partition functions. These are built around convergence of a six-parameter partition function that includes varying the starting and ending times and spatial locations, as well as the boundary parameters and the inverse temperature. From this master result, convergence for appropriate initial data follows as a corollary. We use heat kernel estimates from \cite{wu2018intermediate} and the overall technique from \cite{alberts2014intermediate} (see also \cite{CSZ}).

\subsection{Half-space log-gamma polymer}
\label{sec:loggammaintro}
The log-gamma polymer model is  an exactly solvable  model of directed polymer in the quadrant $\mathbb Z_{>0}^2$, with $\Gammainv$-distributed weights introduced in \cite{seppalainen2012scaling} (see also \cite{corwin2014tropical}). In this paper, we are interested in a half-space version of the model that was first studied in \cite{o2014geometric} (see also some subsequent works \cite{barraquand2018half,bisi2019point,BOCZ,barraquand2021identity}).

\subsubsection{Partition function discrete SHE}
The partition function for this model, denoted by $\zd(n,m)$ for $n,m\in \Z_{\geq 0}$ with $n\geq m$ is defined in Section \ref{sec:loggamma} via a discrete path integral through the $\Gammainv$ weights. It is immediate that it satisfies the following discrete version of the SHE subject to imposing suitable initial data for the distribution of the process $\left( \zd(n,0) \right)_{n\geq 0}$:
\begin{align}\label{eq:discreteSHE}\tag{dSHE$_{u,\alpha}$}
	\begin{split}
		\zd(n,m) &= \wweight_{n,m} (\zd(n-1,m)+\zd(n,m-1)) \text{ for }n>m\geq 1,\\
		\zd(n,n) &= \wweight_{n,n} \zd(n,n-1) \text{ for }n\geq 1.
	\end{split}
\end{align}
Above we will assume that the weights $\wweight_{n,m}$ are all independent, and distributed as
$$\begin{cases}
	\wweight_{n,m}\sim\Gammainv(2\alpha) &\mbox{ for }n>m\geq 1,\\
	\wweight_{n,n}\sim\Gammainv(\alpha+u)&\mbox{ for }n\geq 1,
\end{cases}$$
where we have a bulk parameter $\alpha>0$ and a boundary parameter $u>-\alpha$.
The free energy $\hd(n,m):=\log \zd(n,m)$ plays a similar role here as the KPZ equation $\HH(T,X)$ did earlier. A process $\big( h(k)\big)_{k\in \Z_{\geq 0}}$ is stationary for the half-space log-gamma polymer if the solution to \eqref{eq:discreteSHE} with initial data $\hd(\cdot,0)=\hd(\cdot)$ has the property that the distribution of $\big( \hd(m+k,m)-\hd(m,m)\big)_{k\in \Z_{\geq 0}}$ in the same for all $m\in \Z_{\geq 0}$ (and equal to that of the initial data).

\subsubsection{Half-space log-gamma polymer stationary measures}
A process $\left( r(k)\right)_{k\in \Z_{\geq 0}}$ is a $\Gammainv(\theta)$ multiplicative random walk if $\big(r(k+1)/r(k)\big)_{k\in \Z_{\geq 0}}$ are i.i.d. $\Gammainv(\theta)$.

\begin{definition}
	For $\alpha\in \R_{\geq 0}$ and $u,v\in \R$ with  $u,v>-\alpha$ and $v\leq \min\lbrace 0,u\rbrace$, define
	\begin{equation}
		\zd_{u,v}(k) := r_2(k)  + \frac{1}{\wweight}\sum_{\ell=1}^k r_1(\ell) \frac{r_2(k)}{r_2(\ell-1)},\qquad \hd_{u,v}(k):=\log \zd_{u,v}(k),\quad \textrm{for $k\in \Z_{\geq 0}$,}
		\label{eq:defzuv}
	\end{equation}
	where $r_1$ and $r_2$ are independent $\Gammainv(\alpha+v)$ and $\Gammainv(\alpha-v)$ (resp.) multiplicative random walks with $r_1(0)=r_2(0)=1$, and $\wweight\sim \Gammainv(u-v)$ is independent (again, when $u=v$, we define $\frac{1}{\wweight}=0$).
\end{definition}

\begin{remark}\label{rem:discreteprops}
	Let us mention a few properties of the process $\zd_{u,v}$ (equivalently $\hd_{u,v}$):
	\begin{itemize}%
		\item For general $u,v$, $\zd_{u,v}(k+1)/\zd_{u,v}(k)$ are not independent or translation invariant.
		\item For $u=v\leq 0$, $\zd_{u,u}=r_2$ is a  $\Gammainv(\alpha-u)$ multiplicative random walk.
		\item For $u=-v\geq 0$, $\zd_{u,-u}$ is a $\Gammainv(\alpha-u)$ multiplicative random walk (Lemma \ref{prop:specialcase}).
		\item The parameter $v$  controls the behaviour at infinity in the sense that $\zd_{u,v}(k+1)/\zd_{u,v}(k)\Rightarrow \Gammainv(\alpha+v)$ when $k$ goes to infinity (see Section \ref{sec:discprevwork}).
		\item The definition of  $\zd_{u,v}$ in \eqref{eq:defzuv} makes perfect sense for $v>0$ (provided $-\alpha< v\leq u$) in which case the processes $\zd_{u,v}$ and $\zd_{u,-v}$ have the same distribution (Lemma \ref{rem:symmetryvtominusv}).
	\end{itemize}
\end{remark}

Our second result is the  description of the half-space log-gamma stationary measures.
\begin{theorem} For any $\alpha\in \R_{>0}$ and $u,v\in \R$ such that $u,v>-\alpha$ and $v\leq \min\lbrace 0,u\rbrace$, the process $\hd_{u,v}$  is stationary for the half-space log-gamma polymer, i.e., for \eqref{eq:discreteSHE}.
	\label{theo:loggammaintro}
\end{theorem}

The following is a discrete version of Conjecture \ref{conj:ergodictheory}.

\begin{conjecture}\label{conj:ergodictheoryloggamma}
	For given $\alpha\in \R_{>0}$ and $u>\-\alpha$, $\big\lbrace \zd_{u,v}\big\rbrace_{v\in (-\alpha, \min(u,0)]}$ constitutes all extremal stationary measures for \eqref{eq:discreteSHE} and for any (potentially random) initial data $\zd(k,0)$ such that $\lim_{k\to\infty} \log\zd(k,0)/k = -d\in \R$ we have the following. Let $v=\psi^{-1}(d)-\alpha$ (where $\psi^{-1}(d)$ is the unique positive root of $\psi(z)=d$).
	\begin{enumerate}%
		\item For $u\geq 0$: If $v\leq 0$, the process $\zd(\cdot,m)/\zd(m,m)$ converges weakly  as $m\to \infty$ to $\zd_{u,v}$; if $v\geq 0$, it converges to $\zd_{u,0}$.
		\item For $u\leq 0$:  If $v\leq u$  the process $\zd(\cdot,m)/\zd(m,m)$ converges weakly as $m\to \infty$ to $\zd_{u,v}$; if $v\geq u$, it converges to $\zd_{u,u}$,  a $\Gammainv(\alpha-u)$ multiplicative random walk.
	\end{enumerate}
\end{conjecture}

\subsubsection{Relation to previous work}\label{sec:discprevwork}
Stationary measures of the log-gamma polymer in a quadrant $\mathbb Z_{\geq 0}^2$ (i.e. full-space) have been found in \cite{seppalainen2012scaling}, where the model was first introduced. There exists a one-parameter family of stationary measures which are multiplicative random walks with inverse gamma distributed increments.  It is shown in \cite{georgiou2013ratios} (see also \cite{janjigian2020uniqueness}) that these are the only stationary measures which are  ergodic with respect to translations in both directions of the lattice (so that they correspond to spatially ergodic stationary measures).  Regarding the half-space log-gamma polymer partition function, as in \eqref{eq:discreteSHE}, the fact that the $\Gammainv(\alpha-u)$ multiplicative random walk is stationary was already known \cite{barraquand2020half} following the same method as in \cite{seppalainen2012scaling} (see also Section \ref{sec:loggammaonerow} and \cite[Section 4]{barraquand2020half}). All of the other stationary solutions we exhibit are new.
Unlike the KPZ equation, the log-gamma polymer has not been studied in geometries with periodic or two-sided boundary conditions.

The stationary measure $\zd_{u,v}$ admits a discrete version of \eqref{eq:geometricPitman}. Let
$$p(k) = \prod_{i=1}^k \xi_i\quad\textrm{for $k\in \Z_{\geq 0}$}\qquad \textrm{ and }\qquad r(k)= \zeta_1\prod_{i=2}^k \frac{\zeta_i}{\xi_{i-1}} \quad \textrm{for $k\in \Z_{\geq 1}$}.$$
where $\zeta_i\sim \Gammainv(\alpha+v)$ are i.i.d. and $\xi_i\sim \Gammainv(\alpha-v)$ are i.i.d as well.
Denote by $\mathrm{Beta}'(a,b)$ the Beta prime distribution with parameters $(a,b)$ (i.e., distribution of $G_{a}/G_{b}$ for independent variables $G_a\sim\mathrm{Gamma}(a)$ and $G_b\sim\mathrm{Gamma}(b)$). Then $r(k)$ is a $\mathrm{Beta}'(\alpha-v,\alpha+v)$ multiplicative random walks with $r(1)\sim \Gammainv(\alpha+v)$ (note that increment ratios of $p(k)$ and $r(k)$ are not independent).
For $\wweight\sim\Gammainv(u-v)$ independent from $p$ and $r$,  define
$$
a(n) := 1+\frac{1}{\wweight} \sum_{k=1}^n r(k).
$$
Then the process $\zd_{u,v}$ from \eqref{eq:defzuv} has the same distribution as
$
\zd_{u,v}(k) = p(k)a(k).
$
This can be seen as a discrete analogue of \eqref{eq:geometricPitman}. The laws of processes very  similar to $r$ and $a$ have been recently  studied in \cite{arista2021matsumoto} for certain initial conditions. In particular, for $v>0$, \cite[Proposition 6.1]{arista2021matsumoto} shows that $a(n)$ converges almost-surely as $n\to\infty$ to a random variable distributed as
$
1+\frac{G_{u-v}}{G_{2v}},
$
where $G_{u-v} \sim\mathrm{Gamma}(u-v)$ and $G_{2v} \sim\mathrm{Gamma}(2v)$ are independent \cite[Theorem 6.2]{arista2021matsumoto}. This shows, in particular, that the process $\zd_{u,v}$ behaves at infinity as a $\Gammainv(\alpha-\vert v\vert)$ multiplicative random walk as noted above.

\subsubsection{Ideas in the proof of Theorem \ref{theo:loggammaintro}}
\label{sec:loggammaproofsketch}
There are two ideas that go into this proof.
The first is a remarkable symmetry enjoyed by the half-space (and full-space) log-gamma polymer model. In Definition \ref{def:loggammahalfspace} we recall the half-space log-gamma polymer model in terms of a sum over lattice paths of products of weights $\wweight_{i,j}$ which are inhomogeneous in a manner introduced in \cite{corwin2014tropical, o2014geometric}. Namely, their distribution depends on a boundary parameter $\alpha_{\circ}$ and a family of inhomogeneity parameters $\alpha_1, \alpha_2, \dots$ so that $\wweight_{i,j}\sim \Gammainv(\alpha_i+\alpha_j)$ for $i>j$ and $\wweight_{i,i}\sim \Gammainv(\alpha_i+\alpha_{\circ})$.
We demonstrate that the distribution of $\zd(n_1,m), \dots, \zd(n_k,m)$ (the partition function for this model) is invariant under permuting  the parameters $\alpha_1, \dots, \alpha_m$ (see Lemma \ref{lem:symmetry}). This symmetry stems from the relation (as in \cite{o2014geometric}) between the partition functions and Whittaker functions, and can be seen as a limit of a similar symmetry enjoyed by half-space Macdonald processes \cite{barraquand2018half}.

The second idea involves tuning the inhomogeneities to special values so as to route all polymer paths through a given vertex. Combining this with the above symmetry leads to the desired stationarity. The inhomogeneity tuning that we use is in the spirit of the tuning used in full-space models, going back to works of \cite{BaikRains2000,ferrarispohn2006} on the PNG model and TASEP and similar with other stationary models in the KPZ class \cite{borodin2015height, aggarwal2016phase, aggarwal2016current,imamura2017free, imamura2019fluctuations, imamura2019stationary, betea2020stationary, barraquand2020half}. However, the path routing and use of symmetry to construct the stationary measures seems novel to our work.

Let us explain the idea further. We demonstrate that stationary distributions of partition function ratios arise when the parameters $\alpha_{\circ}$ and $\alpha_1,\alpha_2, \dots$ are tuned so some of the inverse gamma weights become infinite. The simplest way to produce such a divergence is to set $\alpha_{\circ}=u, \alpha_1=-u$ as this forces the weight $\wweight_{1,1}$ to diverge. By the symmetry we can swap $\alpha_1$ and $\alpha_m$ without changing the law of $\zd(n_1,m), \dots, \zd(n_k,m)$. However, now $\wweight_{m,m}$ diverges and if we consider the ratios $\zd(n,m)/\zd(m,m)$, the only paths in $\zd(n,m)$ which contribute are those that go through $(m,m)$. This shows that the law of $\zd(n,m)/\zd(m,m)$ as a process in $n$ is independent of $m$ and hence that the $\Gammainv(\alpha-u)$ multiplicative random walk  is stationary. To get to the full family $\zd_{u,v}$ requires a two-level tuning where we impose that $\alpha_1+\alpha_2=0$ so the weight $\wweight_{2,1}$ diverges. Using the parameter symmetry, the location of this diverging weight can be freely moved to any point in the lattice and used to force that the only contributing paths are those that pass through that point. In particular, we permute $(\alpha_1,\alpha_2)$ with $(\alpha_{m-1},\alpha_m)$ which forces paths through $(m,m-1)$. This ultimately produces our full family of stationary measures. The details of this argument are given in Proposition \ref{prop:stationarity2}.

\subsubsection{Possible extensions of this method}
The path-routing argument just described  applies to some other discrete exactly solvable models. In Section \ref{sec:LPP}, we present an analogue of Theorem \ref{theo:loggammaintro} for geometric and exponential half-space last-passage percolation (LPP). The exponential case is a limit of the log-gamma case; the geometric case is not though yields to a similar proof. There are a number of other potential directions to consider beyond this.

Instead of considering partition functions associated to one single polymer path, one may also consider the partition function for $\ell$ non-intersecting paths, whose distribution can still be exactly computed in certain cases \cite{corwin2014tropical, borodin2014macdonald, o2014geometric, barraquand2018half}. By employing a similar method as described in Section \ref{sec:loggammaproofsketch}, but forcing several weights to diverge, our method may shed light on stationary measures for the half-space multi-layer discrete SHEs (in the spirit of \cite{o2016multi}) satisfied by the non-intersecting path partition functions (see  \cite{barraquand2022stationary} in the full-space case).

Besides polymer models, it may be possible to implement path routing arguments to derive stationary measures for more general models that fit into the family of higher-spin stochastic six-vertex model \cite{corwin2015stochastic, borodin2016higher}. In full-space this family is well studied. Only a few half-space versions have been constructed \cite{barraquand2018stochastic, barraquand2018half, barraquand2022random}, leaving this a fertile and  challenging direction.

\subsubsection*{Outline} In Section \ref{sec:loggamma} we consider the log-gamma polymer and prove Theorem \ref{theo:loggammaintro}. In Section \ref{sec:LPP}, we obtain similar results for half-space LPP. In Section \ref{sec:KPZ}, we prove Theorem \ref{theo:invariantKPZintro}, using a general intermediate disorder half-space polymer to SHE convergence  results that is stated and proved in Section \ref{sec:generalconvergence} (Theorem \ref{thm:SHEsheetlimit}). This is applied to the log-gamma polymer in Theorem \ref{prop:KPZconvergence} and the necessary matching of notation that goes into the proof of that is provided in Section \ref{sec:proofconvergence}.

\subsubsection*{Acknowledgments}
	We thank Shalin Parekh and Xuan Wu for helpful conversations, and we are grateful to the anonymous referees for their careful reading and valuable comments.   

	G.B. was partially supported by ANR grant ANR-21-CE40-0019. I.C was partially supported by the NSF through grants DMS:1937254, DMS:1811143, DMS:1664650, as well as through a Packard Fellowship in Science and Engineering, a Simons Fellowship, a Miller Visiting Professorship from the Miller Institute for Basic Research in Science, and a W.M. Keck Foundation Science and Engineering Grant. Both G.B. and I.C. also wish to acknowledge the NSF grant DMS:1928930 which supported their participation in a fall 2021 semester program at MSRI in Berkeley, California.

\section{Stationary half-space log-gamma polymers}
\label{sec:loggamma}
We consider the half-space log-gamma polymer model \cite{o2014geometric}, following the notation of \cite{barraquand2018half}.
\begin{definition}[Inhomogeneous log-gamma polymer]
	Let  $\diag, \alpha_1, \alpha_2, \dots $ be real parameters  such that $\alpha_i+\diag>0$ for all $i \geqslant 1$ and $\alpha_i+\alpha_j>0$ for all $i\neq j\geqslant 1$. Let  $\big(\wweight_{i,j}\big)_{i\geqslant j}$ be a family of independent random variables such that for $i>j, \wweight_{i,j}\sim \Gammainv(\alpha_i + \alpha_j)$ and $\wweight_{i,i}\sim \Gammainv(\diag + \alpha_i)$.
	The partition function of the half-space log-gamma polymer is defined as (see Figure~\ref{fig:halfspaceloggamma})
	$$\zd(n,m) = \sum_{\pi: (1,1) \to (n,m)} \prod_{(i,j)\in \pi} \wweight_{i,j}, $$
	where the sum is over up-right paths from $(1,1)$ to $(n,m)$ in the octant $ \lbrace (i,j)\in \Z_{>0}^2 : i\geqslant j \rbrace $.
	\label{def:loggammahalfspace}
\end{definition}
\begin{figure}
	\begin{center}
		\begin{tikzpicture}[scale=0.8]
			\draw[thick, gray]  (0,0) -- (5.5,5.5);
			\draw (-.5, -.2) node{{\footnotesize $(1,1)$}};
			\draw (9.6, 5.3) node{{\footnotesize $(n,m)$}};
			\draw[-stealth'] (9.5,2) node[anchor=west, align=center]{{\footnotesize $\wweight_{i,j} \sim \Gammainv(\alpha_i + \alpha_j)$} \\ \footnotesize  for $i>j$} to[bend left] (8,2);
			\draw[-stealth'] (2.5,4.5) node[anchor=south]{{\footnotesize $\wweight_{i,i} \sim \Gammainv(\diag + \alpha_i) $}} to[bend right] (3,3);
			\clip (-0.05,-0.05) -- (9.5,-0.05) -- (9.5,5.5) -- (5.49, 5.49) -- (-0.05,-0.05);
			\draw[dashed, gray] (0,0) grid (10,6);
			\draw[ultra thick] (0,0) -- (1,0) -- (2,0) -- (2,1) -- (2,2) -- (3,2) -- (4,2) -- (4,3) -- (5,3) -- (6,3) -- (6,4) -- (7,4) -- (8,4) -- (8,5) -- (9,5);
		\end{tikzpicture}
	\end{center}
	\caption{An admissible path in the half-space log-gamma polymer (Definition \ref{def:loggammahalfspace}).}
	\label{fig:halfspaceloggamma}
\end{figure}
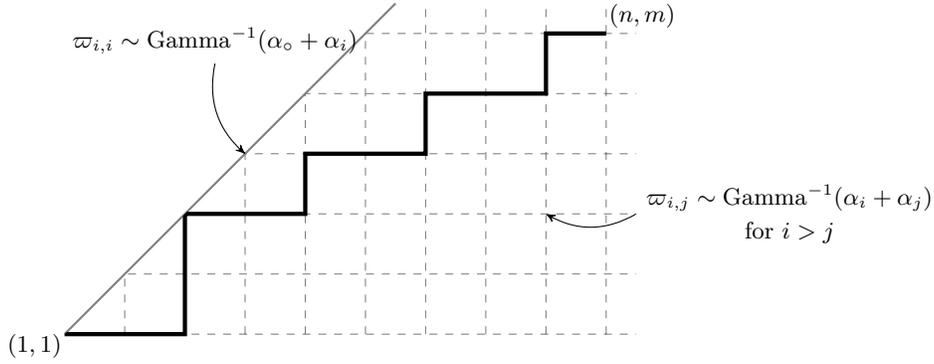

In this section we define two types of stationary half-space log-gamma partition functions (with respect to homogeneous bulk parameters). The first one, denoted $\zd^{\rm stat}_u(n,m)$, was introduced in \cite{barraquand2020half} and depends on a bulk parameter $\alpha$ and a single boundary parameter $u$. The second one (which is new and includes the first as a special case), denoted $\zd^{\rm stat}_{u,v}(n,m)$, depends on $\alpha, u$, and an additional parameter $v$ that controls the free energy drift at infinity. Both these partition functions are stationary in the sense that the law of $\zd^{\rm stat}_{u,v}(m+k,m)/\zd^{\rm stat}_{u,v}(m,m)$, as a process in $k$, does not depend on $m$ (at least for $m\geq 2$). This stationary process will be matched to $z_{u,v}$ defined in \eqref{eq:defzuv} and shown to satisfy \eqref{eq:discreteSHE}, thus proving Theorem \ref{theo:loggammaintro}. Our stationary partition functions arise from specific specializations of the parameters of the inhomogeneous polymer model, generalizing what happens in the case of the full-space stationary log-gamma polymer model \cite{seppalainen2012scaling}.

\subsection{One-row stationary measures}
\label{sec:loggammaonerow}
We recall the stationary $\zd^{\rm stat}_u(n,m)$ from \cite{barraquand2020half}.
\begin{definition}
	For parameters $\alpha\in\R_{>0}$ and $u\in (-\alpha,\alpha)$, define
	\begin{equation}
		\zd^{\rm stat}_u(n,m) := \frac{\zd(n,m)}{\wweight_{1,1}}
		\label{eq:defZustat}
	\end{equation}
	where $\zd(n,m)$ is the partition function from Definition \ref{def:loggammahalfspace} with $\diag=u, \alpha_1=-u$, and for all $i\geq 2$, $\alpha_i=\alpha$. Note that $\zd^{\rm stat}_u(n,m)$ does not depend on $\wweight_{1,1}$, so that it is defined even though $\diag+\alpha_1=0$. We define horizontal and vertical increments of the partition function as
	\begin{equation}
		U_{n,m} = \frac{\zd^{\rm stat}_u(n,m)}{\zd_u^{\rm stat}(n-1,m)}, \;\;V_{n,m} = \frac{\zd_u^{\rm stat}(n,m)}{\zd^{\rm stat}_u(n,m-1)}.
		\label{eq:defincrements}
	\end{equation}
\end{definition}

The partition function satisfies the recurrence, as in \eqref{eq:discreteSHE},
\begin{equation}
	\zd^{\rm stat}_u(n,m) = \wweight_{n,m}(\zd^{\rm stat}_u(n-1,m)+\zd^{\rm stat}_u(n,m-1)),
\end{equation}
where we fix $\zd^{\rm stat}_u(n,m)\equiv 0$ for  $(n,m)\notin\lbrace (i,j)\in \Z_{>0}^2 : i\geqslant j \rbrace $. Hence, we have
\begin{equation}
	U_{n,m} = \wweight_{n,m}\left( 1 +\frac{U_{n,m-1}}{V_{n-1,m}} \right),\;\;\;V_{n,m} =  \wweight_{n,m}\left( 1 +\frac{V_{n-1,m}}{U_{n,m-1}} \right).
	\label{eq:recurrenceforincrements}
\end{equation}
Stationary solutions to the recurrence equations \eqref{eq:recurrenceforincrements} were found in \cite[Lemma 3.2]{seppalainen2012scaling}.
\begin{lemma}
	Let $\alpha\in\R_{>0}$ and $u\in (-\alpha,\alpha)$. Assume that $U,V,w$ are independent random variables and let
	\begin{equation}
		U'=w\left(1+\frac{U}{V} \right), \;\; V'=w\left(1+\frac{V}{U} \right), \;\; w'=\left( \frac{1}{U} +\frac{1}{V}\right)^{-1}.
	\end{equation}
	If  $U\sim \mathrm{Gamma^{-1}}(\alpha+u)$, $V\sim \mathrm{Gamma^{-1}}(\alpha-u)$, $w\sim \mathrm{Gamma^{-1}}(2\alpha)$, then the triples $(U,V,w)$ and $(U', V', w')$ have the same distribution.
	\label{lem:seppalainenstationary}
\end{lemma}
Using this \cite[Proposition 4.5]{barraquand2020half} showed that $\zd^{\rm stat}_u$ is stationary in the following sense.
\begin{proposition}
	For $\alpha\in\R_{>0}$ and $u\in (-\alpha,\alpha)$, $\zd^{\rm stat}_{u}(m,m)$ satisfies:
	\begin{enumerate}%
		\item  Let $k\in \mathbb Z_{\geq 1}$ and consider points $\mathbf p_1=(n_1, m_1), \dots, \mathbf p_k=(n_k,m_k)$ along a down right-path in the octant (i.e., $n_i\geqslant m_i$, and the points are ordered such that $n_1\leqslant \dots \leqslant n_k$ and $m_1\geqslant \dots \geqslant m_k\geqslant 0$). Then, the joint distribution of
		$$
		\left( \frac{\zd^{\rm stat}_{u}\big((m,m)+\mathbf p_i\big)}{\zd^{\rm stat}_{u}\big((m,m)+\mathbf p_1\big)}\right)_{1\leq i\leq k}
		$$
		does not depend on $m\geq 1$.
		Moreover, the increment ratios $\zd^{\rm stat}_{u}(\mathbf p_{i+1})/\zd^{\rm stat}_{u}(\mathbf p_i)$ are independent, and distributed as
		$\zd^{\rm stat}_{u}(\mathbf p_{i+1})/\zd^{\rm stat}_{u}(\mathbf p_i) \sim   \mathrm{Gamma^{-1}}(\alpha-u) $ if $\mathbf p_{i+1}-\mathbf p_i=(1,0)$ and $\zd^{\rm stat}_{u}(\mathbf p_{i})/\zd^{\rm stat}_{u}(\mathbf p_{i+1}) \sim  \mathrm{Gamma^{-1}}(\alpha+u) $ if $\mathbf p_{i}-\mathbf p_{i+1}=(0,1)$.
		
		\item In particular, for $k\in \Z_{\geq 0}$, the process
		$
		k\mapsto \frac{\zd^{\rm stat}_{u}(m+k,m)}{\zd^{\rm stat}_{u}(m,m)}
		$
		has the same distribution for any $m\in \Z_{\geq 1}$, a $\Gammainv(\alpha-u)$ multiplicative random walk  starting from $1$.
	\end{enumerate}
	\label{prop:stationarity}
\end{proposition}

\subsection{Two-row stationary measures}
\label{sec:loggammatworow}
We now introduce a new class of stationary partition functions which includes the previously known example as a special case when $u=v$.
\begin{definition}\label{def:tworow}
	Fix any $\alpha\in \R_{>0}$ and $u,v\in \mathbb R$ with $u>-\alpha, v\in (-\alpha, \alpha)$ and $v<u$ (the case $u=v$ can be accessed subsequently by a limit, see the proof of of Theorem \ref{theo:loggammaintro} in Section \ref{sec:discretespatialprocess}).
	If $u+v>0$, then we  define
	\begin{equation}
		\zd^{\rm stat}_{u,v}(n,m) := \frac{\zd(n,m)}{\wweight_{2,1}}
		\label{eq:defZuvstat}
	\end{equation}
	where  $\zd(n,m)$ is the partition function from Definition \ref{def:loggammahalfspace} with $\diag=u, \alpha_1=v, \alpha_2=-v$, and for all $i\geq 3$, $\alpha_i=\alpha$. Note that $\zd^{\rm stat}_{u,v}(n,m)$ does not depend on $\wweight_{2,1}$, so that it is defined even though $\alpha_1+\alpha_2=0$. Any ratio of stationary partition functions of the form $\frac{\zd^{\rm stat}_{u,v}(n_1,m_1)}{\zd^{\rm stat}_{u,v}(n_2,m_2)}$ does not depend on $\wweight_{1,1}$, so these ratios are well-defined even when $u+v\leqslant 0$ (e.g., by simply setting $\wweight_{1,1}\equiv 1$).
\end{definition}

\begin{proposition}
	For $\alpha\in \R_{>0}$ and $u,v\in \mathbb R$ such that $u>-\alpha, v\in (-\alpha, \alpha)$ and $v<u$. The partition function $\zd^{\rm stat}_{u,v}(m,m)$ satisfies the following properties:
	\begin{enumerate}%
		\item Let $k\in \mathbb Z_{\geq 1}$ and consider points $\mathbf p_1=(n_1, m_1), \dots, \mathbf p_k=(n_k,m_k)$ along a down right-path in the octant (by that we mean that $n_i\geqslant m_i$, and the points are ordered such that $n_1\leqslant \dots \leqslant n_k$ and $m_1\geqslant \dots \geqslant m_k\geqslant 0$). Then, the joint distribution of
		$$\left( \frac{\zd^{\rm stat}_{u,v}((m,m)+\mathbf p_i)}{\zd^{\rm stat}_{u,v}((m,m)+\mathbf p_1)}\right)_{1\leq i\leq k}$$
		does not depend on $m$, as long as $m\geqslant 2$.
		\item 	In particular, for $k\in \Z_{\geq 0}$, the process
		$k\mapsto  \frac{\zd^{\rm stat}_{u,v}(m+k,m)}{\zd^{\rm stat}_{u,v}(m,m)}$
		has the same distribution for any $m\geqslant 2$.
	\end{enumerate}
	\label{prop:stationarity2}
\end{proposition}
A crucial step  in proving Proposition \ref{prop:stationarity2} is the following symmetry in the law of the partition functions for the   inhomogeneous log-gamma model from Definition \ref{def:loggammahalfspace}.
\begin{lemma} For $m\in \Z_{\geq 1}$ and $k\in \Z_{\geq 0}$, the law of
	\begin{equation}
		\left(\zd(m,m),\zd(m+1,m) \dots,\zd(m+k,m) \right)
		\label{eq:jointdistribution}
	\end{equation}
	is invariant with respect to permutations of the parameters $\alpha_1, \dots, \alpha_m$.
	\label{lem:symmetry}
\end{lemma}
\begin{proof}
	This type of symmetry is a known hallmark for probabilistic systems encoded in terms of specializations of symmetric functions. In our case, it can be seen that the joint distribution of $\left(\zd(m+k_{1},m), \dots,\zd(m+k_{\ell},m) \right)$ depends on the parameters $\boldsymbol \alpha=(\alpha_1, \dots, \alpha_m)$ only through the values of the Whittaker function $\psi_{\boldsymbol \alpha}(x)$. This function, however, is invariant by permutation of the $\alpha_i$. Since this result is not given explicitly in the works \cite{o2014geometric, barraquand2018half} on the integrability of the half-space log-gamma polymer, we produce a brief (not entirely self-contained) proof of it here.
	
	The symmetry property can be deduced from \cite{barraquand2018half} by using the more general framework of half-space Macdonald processes. Consider any down-right path in the octant formed by the points $\mathbf p_1=(m,m), \mathbf p_2=(m+1,m), \dots, \mathbf p_{k+1}=(m+k,m)$ first progressing towards the east, and then continuing southward with $\mathbf p_{k+2}=(m+k,m-1), \dots, \mathbf p_{k+m}=(m+k,1)$. The half-space Macdonald process for this path is the following probability measure on sequences $(\lambda^{\mathbf p_i})_{0\leq i\leq k+m}$ of integer partitions indexed by the points along the path:
	\begin{multline*}
		\mathbb P\big( (\lambda^{\mathbf p_i})_{0\leq i\leq k+m} \big)   \propto   \mathcal E_{\lambda^{\mathbf p_1}}(a_{\circ}) Q_{\lambda^{\mathbf p_2}/\lambda^{\mathbf p_1}}(a_{m+1}) \dots Q_{\lambda^{\mathbf p_{k+1}}/\lambda^{\mathbf p_{k}}}(a_{m+k}) \\ \times  P_{\lambda^{\mathbf p_{k+1}}/\lambda^{\mathbf p_{k+2}}}(a_{m}) \dots P_{\lambda^{\mathbf p_{k+m-1}}/\lambda^{\mathbf p_{k+m}}}(a_{2})P_{\lambda^{\mathbf p_{k+m}}}(a_1).
	\end{multline*}
	Here the $a_i$ are positive real numbers, $P_{\lambda/\mu}(x)$ and $Q_{\lambda/\mu}(x)$ are single-variable specializations of skew Macdonald polynomials and $\mathcal E_{\lambda}$ is another single-variable specialized polynomial defined in \cite[Equation (15)]{barraquand2018half}.
	Macdonald symmetric functions  were introduced in  \cite{macdonald1988new, macdonald1995symmetric}. They can be evaluated into any set of variables $x=\lbrace x_1, \dots, x_{m}\rbrace$ and are then symmetric polynomials in the $x_i$, with coefficients depending on two extra deformation parameters $q,t$.
	Since we will be interested only in the distribution of $\lambda^{\mathbf p_i}$ for $1\leq i\leq k+1$ (that is,  corresponding to points $\mathbf p_i$ on the $m$th row), we may sum the probability distribution over the other integer partitions $\lambda^{\mathbf p_i}$ for $i>k+1$. Use the branching rule for Macdonald polynomials $P$ (e.g., \cite[Equation (23)]{barraquand2018half}) and we obtain that
	$$
	\mathbb P\big( (\lambda^{\mathbf p_i})_{0\leq i\leq k+1} \big) \propto \mathcal E_{\lambda^{\mathbf p_1}}(a_{\circ}) Q_{\lambda^{\mathbf p_2}/\lambda^{\mathbf p_1}}(a_{m+1}) \dots Q_{\lambda^{\mathbf p_{k+1}}/\lambda^{\mathbf p_{k}}}(a_{m+k}) P_{\lambda^{\mathbf p_{k+1}}}(a_1, \dots, a_{m}).
	$$
	Thus, the law of $(\lambda^{\mathbf p_i})_{0\leq i\leq k+1} $ is symmetric with respect to permuting $a_1, \dots, a_m$.
	
	To conclude the proof, we set $t=0$, and write $a_i=q^{\alpha_i}, a_{\circ}=q^{\alpha_{\circ}}$. Then, \cite[Proposition 6.27]{barraquand2018half} shows that, jointly for all points $(n,m)$ along the path $\left( \mathbf p_i\right)_{1\leq i\leq k+m}$, we have
	$$ (1-q)^{n+m-1} q^{-\lambda_1^{(n,m)}} \xRightarrow[q\to 1]{} \zd(n,m),$$
	(i.e., weak convergence) where $\zd(n,m)$ is exactly the partition function from Definition \ref{def:loggammahalfspace}.
\end{proof}
\begin{remark}
	The joint distribution of the partition functions \eqref{eq:jointdistribution} can be computed explicitly using the geometric RSK correspondence. The joint density appears in \cite{barraquand2021identity} (which uses results from \cite{o2014geometric, bisi2019point}) -- namely, as the integrand of \cite[Equation (3.16)]{barraquand2021identity}. This can be factorized as \cite[Equation (3.17)]{barraquand2021identity} (which does not depend on $\boldsymbol{\alpha}$) times a Whittaker function indexed by $\boldsymbol{\alpha}$ (which is symmetric in the $\alpha_i$). This yields another proof of Lemma \ref{lem:symmetry}.
\end{remark}

\begin{proof}[Proof of Proposition \ref{prop:stationarity2}]
	Part (1) clearly implies part (2), thus it suffices to show that provided  $\alpha\in \R_{>0}$ and $u,v\in \mathbb R$ such that $u>-\alpha, v\in (-\alpha, \alpha)$ and $v<u$, we have
	\begin{equation}
		\left(\frac{\zd^{\rm stat}_{u,v}\big((m,m)+\mathbf p_i\big)}{\zd^{\rm stat}_{u,v}\big((m,m)+\mathbf p_1\big)}  \right)_{1\leq i\leq k} \overset{(d)}{=}
		\left(\frac{\zd^{\rm stat}_{u,v}\big((2,2)+\mathbf p_i\big)}{\zd^{\rm stat}_{u,v}\big((2,2)+\mathbf p_1\big)} \right)_{1\leq i\leq k}.
		\label{eq:toprove}
	\end{equation}
	As observed earlier in Definition \ref{def:tworow}, the ratios $\zd^{\rm stat}_{u,v}((m,m)+\mathbf p_i)/\zd^{\rm stat}_{u,v}((m,m)+\mathbf p_1)$ do not depend on $\wweight_{1,1}$. We will establish \eqref{eq:toprove} under the additional restriction that $u+v>0$ and then can analytically continue this to hold without that restriction (the analyticity of the density for both sides is straightforward).
	With this in mind, let us now assume that $u+v>0$ and for $\epsilon>0$, define the partition function $\zd^{\rm stat;\epsilon}_{u,v}(n,m)$ to equal the partition function $\zd(n,m)$ from  Definition \ref{def:loggammahalfspace} with $\diag=u, \alpha_1=v, \alpha_2=-v+\epsilon$, and for all $i\geq 3$, $\alpha_i=\alpha$. We have that
	\begin{equation}
		\left(\frac{\zd^{\rm stat;\epsilon}_{u,v}\big((m,m)+\mathbf p_i\big)}{\wweight_{2,1}} \right)_{1\leq i\leq k}\xRightarrow[\epsilon\to 0]{}  \left(\zd^{\rm stat}_{u,v}\big((m,m)+\mathbf p_i\big)\right)_{1\leq i\leq k}.
		\label{eq:limitepsilonto0}
	\end{equation}
	Similarly define $\widetilde{\zd}^{\rm stat;\epsilon}_{u,v}(n,m)$ to equal the partition function $\zd(n,m)$ from  Definition \ref{def:loggammahalfspace} with $\diag=u$, $\alpha_{m-1}=v, \alpha_m=-v+\epsilon$ and  for $1\leq i\leq m-2$, $\alpha_i=\alpha$.
	We claim that
	\begin{equation}
		\left(\zd^{\rm stat;\epsilon}_{u,v}((m,m)+\mathbf p_i)\right)_{1\leq i\leq k} \overset{(d)}{=}\left(\widetilde{\zd}^{\rm stat;\epsilon}_{u,v}((m,m)+\mathbf p_i)  \right)_{1\leq i\leq k}.
		\label{eq:identityfrompermutation}
	\end{equation}
	Lemma \ref{lem:symmetry} ensures that \eqref{eq:identityfrompermutation} is true when the $\mathbf p_i$ define a flat path (so that the points $(m,m)+\mathbf p_i$ are along the $m$th row). The result is true a fortiori for any down-right path above the $m$th row. Indeed, the partition functions on such a path can be written in terms of the partition functions on the $m$th row (whose distribution is symmetric in $\alpha_1, \dots, \alpha_m$) and the weights strictly above the $m$th row (whose distribution does not depend on $\alpha_1, \dots, \alpha_m$).

	\begin{figure}
		\begin{center}
			\begin{tikzpicture}[scale=0.58]
				\begin{scope}
					\draw[thick, gray]  (0,0) -- (6.5,6.5);
					\draw (-.5, -.3) node{{\footnotesize $(1,1)$}};
					\draw (9.6, 6.3) node{{\footnotesize $(m,m)+\mathbf p$}};
					\filldraw[red, draw=black] ([xshift=-0.2cm,yshift=-0.2cm]1,0) rectangle ++(0.4,0.4);
					\filldraw[green, draw=black] (0,0) circle(0.1);
					\filldraw[magenta, draw=black] (1,1) circle(0.1);
					\foreach \x in {2,3,...,9}
					\filldraw[blue, draw=black] (\x,0) circle(0.1);
					\foreach \x in {2,3,...,9}
					\filldraw[blue, draw=black] (\x,1) circle(0.1);
					\draw (9.7,5) node{\footnotesize $m$};
					\draw (10,4) node{\footnotesize $m-1$};
					\draw (9.7,1) node{\footnotesize $2$};
					\draw (9.7,0) node{\footnotesize $1$};
					\draw[ultra thick] (0,0) -- (1,0) -- (2,0) -- (2,1) -- (2,2) -- (3,2) -- (4,2) -- (4,3) -- (5,3) -- (7,3) -- (7,4)  -- (8,4) -- (8,5) -- (9,5)-- (9,6);
					\clip (0,0) -- (9.5,0) -- (9.5,6.5) -- (6.5, 6.5) -- (0,0);
					\draw[dashed, gray] (0,0) grid (10,6);
				\end{scope}
				\begin{scope}[xshift=11.7cm]
					\draw[thick, gray]  (0,0) -- (6.5,6.5);
					\draw (-.5, -.3) node{{\footnotesize $(1,1)$}};
					\draw (9.6, 6.3) node{{\footnotesize $(m,m)+\mathbf p$}};
					\filldraw[red, draw=black] ([xshift=-0.2cm,yshift=-0.2cm]5,4) rectangle ++(0.4,0.4);
					\filldraw[green, draw=black] (4,4) circle(0.1);
					\filldraw[magenta, draw=black] (5,5) circle(0.1);
					\foreach \x in {6,7,...,9}
					\filldraw[blue, draw=black] (\x,4) circle(0.1);
					\foreach \x in {6,7,...,9}
					\filldraw[blue, draw=black] (\x,5) circle(0.1);
					\clip (0,0) -- (9.5,0) -- (9.5,6.5) -- (6.5, 6.5) -- (0,0);
					\draw[dashed, gray] (0,0) grid (10,6);
					\draw[ultra thick] (0,0) -- (1,0) -- (2,0) -- (2,1) -- (2,2) -- (3,2) -- (4,2) -- (4,3) -- (5,3) -- (5,4)  -- (8,4) -- (8,5) -- (9,5)-- (9,6);
				\end{scope}
			\end{tikzpicture}
		\end{center}
		\caption{We illustrate \eqref{eq:epszero} with $k=1$, $m=6$, $\mathbf p=(4,1)$. Left: We depict one possible path in the partition function $\zd^{\rm stat;\epsilon}((m,m)+\mathbf p)$ and mark in green the point $(1,1)$ which is weighted by a $\Gammainv(u+v)$ random variable (this point does not matter in partition functions ratios though), in magenta the point $(2,2)$ which is weighted by a $\Gammainv(u-v+\epsilon)$ random variable, in blue the points weighted by a $\Gammainv(\alpha+v)$ or $\Gammainv(\alpha-v+\epsilon)$ random variable and with a red square the point $(2,1)$ which is weighted by a $\Gammainv(\epsilon)$ random variable that diverges as $\epsilon$ goes to zero. Right: We similarly depict one possible path in the partition function $\widetilde \zd^{\rm stat;\epsilon}((m,m)+\mathbf p)$ and used the same marked points convention as on the left. Because the weight on the red square diverges as $\epsilon \to 0$, $\widetilde \zd^{\rm stat;\epsilon}((m,m)+\mathbf p)$ can be approximated by the sum over the paths going through that point, so that the partition function factorizes.}
		\label{fig:factorization}
	\end{figure}
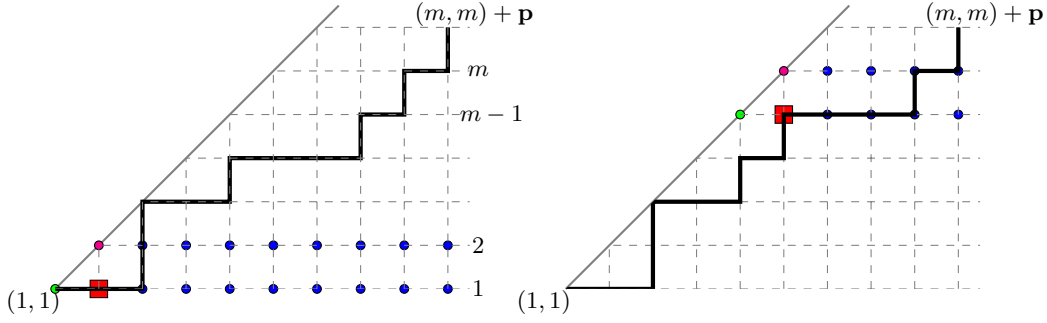
	
	Since the weight $\wweight_{m,m-1}$ diverges as $\epsilon\to 0$, we see that (see Figure \ref{fig:factorization})
	\begin{equation}\label{eq:epszero}
		\left( \frac{\widetilde{\zd}^{\rm stat;\epsilon}_{u,v}\big((m,m)+\mathbf p_i\big)}{\wweight_{m,m-1}}\right)_{1\leq i\leq k} =   \left(  \frac{\widetilde{\zd}^{\rm stat;\epsilon}_{u,v}(m,m-1)}{\wweight_{m,m-1}}\cdot \frac{\widetilde{\zd}^{\rm stat;\epsilon}_{u,v}\big((m,m-1)\to (m,m)+\mathbf p_i\big)}{\wweight_{m,m-1}} + \textrm{E}^{\epsilon}_i   \right)_{1\leq i\leq k}
	\end{equation}
	where the term $\textrm{E}^{\epsilon}_i$  accounts for the parts in the partition function $\widetilde{\zd}^{\rm stat;\epsilon}_{u,v}\big((m,m)+\mathbf p_i\big)$ that do not go through $(m,m-1)$ divided through by $\wweight_{m,m-1}$. Since $\wweight_{m,m-1}\sim \Gammainv(\eps)$, $\mathbb E[\frac{1}{\wweight_{m,m-1}}]=\epsilon$, and since $\wweight_{m,m-1}$ is independent from the weights arising in the sum over paths in  $\textrm{E}^{\epsilon}_i$, we can bound $\mathbb E[\vert \textrm{E}^{\epsilon}_i \vert ]$ by a constant times $\epsilon$, so that  $\textrm{E}^{\epsilon}_i$ converges to $0$ in $L^1$ as $\epsilon$ goes to zero.
	The notation $\widetilde{\zd}^{\rm stat;\epsilon}_{u,v}(\mathbf s\to \mathbf t)$ denotes the partition function of paths from $\mathbf s$ to $\mathbf t$ (in the octant) where the set of weights $(\wweight_{i,j})$ are the same as those used in the definition of  $\widetilde{\zd}^{\rm stat;\epsilon}_{u,v}$ above.
	Observe that by definition
	\begin{equation*} 
		\left(\frac{\widetilde{\zd}^{\rm stat;\epsilon}_{u,v}\big((m,m-1)\to (m,m)+\mathbf p_i\big)}{\wweight_{m,m-1}}\right)_{1\leq i\leq k}\overset{(d)}{=}\left(\frac{\zd^{\rm stat;\epsilon}_{u,v}\big((2,1)\to (2,2)+\mathbf p_i\big)}{\wweight_{2,1}}\right)_{1\leq i\leq k}.
	\end{equation*}
	By \eqref{eq:limitepsilonto0}, the right-hand side above weakly converge to the partition functions  \begin{equation} 
		\left(\frac{\zd^{\rm stat}_{u,v}((2,2)+\mathbf p_i)}{\wweight_{1,1}}\right)_{1\leq i\leq k}.
		\label{eq:defXi}
	\end{equation}
	Combining this with \eqref{eq:limitepsilonto0} yields 
	\begin{equation}
		\left( \frac{\widetilde{\zd}^{\rm stat;\epsilon}_{u,v}((m,m)+\mathbf p_i)}{\widetilde{\zd}^{\rm stat;\epsilon}_{u,v}((m,m)+\mathbf p_1)}\right)_{1\leq i\leq k} \xRightarrow[\epsilon\to 0]{} \left( \frac{\zd^{\rm stat}_{u,v}((2,2)+\mathbf p_i)}{\zd^{\rm stat}_{u,v}((2,2)+\mathbf p_1)}\right)_{1\leq i\leq k}.
		\label{eq:ratioepsilonto0}
	\end{equation}
	Indeed, denoting the random variables in the right-hand side of \eqref{eq:epszero} as $A^{\epsilon} \cdot  X_i^{\epsilon} + E_i^{\epsilon}$, we know that $X_i^{\epsilon}\Rightarrow X_i$ where the $X_i$ are the random variables in \eqref{eq:defXi}. Since the $E_i^{\epsilon}$ go to $0$ in $L^1$, the couple of variables $(E_i^{\epsilon}/A^{\epsilon}, E_1^{\epsilon}/A^{\epsilon})$ goes to zero in $L^1$ as well, so that 
	$(X_i^{\epsilon} + E_i^{\epsilon}/A^{\epsilon},X_1^{\epsilon} + E_1^{\epsilon}/A^{\epsilon} ) \Rightarrow (X_i,X_1)$ and we conclude that 
	$$ \frac{A^{\epsilon}X_i^{\epsilon} + E_i^{\epsilon}}{A^{\epsilon}X_1^{\epsilon} + E_1^{\epsilon}} = \frac{X_i^{\epsilon} + E_i^{\epsilon}/A^{\epsilon}}{X_1^{\epsilon} + E_1^{\epsilon}/A^{\epsilon}} \Rightarrow \frac{X_i}{X_1},$$
	which is exactly \eqref{eq:ratioepsilonto0}. 
	Finally, combining \eqref{eq:identityfrompermutation} and \eqref{eq:ratioepsilonto0} together, we obtain \eqref{eq:toprove}, which concludes the proof.
\end{proof}

\subsection{The stationary process $\zd_{u,v}$}
\label{sec:discretespatialprocess} We prove Theorem \ref{theo:loggammaintro} along with properties of $\zd_{u,v}$.

\begin{proof}[Proof of Theorem \ref{theo:loggammaintro}]
	Let us consider in more details the process
	$
	k\mapsto \frac{\zd^{\rm stat}_{u,v}(m+k,m)}{\zd^{\rm stat}_{u,v}(m,m)}
	$
	which is well-defined as long as $u>-\alpha, v\in (-\alpha, \alpha)$ and $v<u$ (without needing to assume that $u+v>0$). By Proposition \ref{prop:stationarity2} the law of this process does not depend on $m$. For $m=2$
	\begin{equation}
		\frac{\zd^{\rm stat}_{u,v}(2+k,2)}{\zd^{\rm stat}_{u,v}(2,2)} = \frac{1}{\wweight_{2,2}}\sum_{\ell=2}^{k+2} \,\,\prod_{i=3}^{\ell} \wweight_{i,1} \prod_{j=\ell}^{k+2} \wweight_{j,2},
		\label{eq:explicittworowinitialdata}
	\end{equation}
	where $\wweight_{2,2}\sim \Gammainv(u-v)$, $\wweight_{i,1}\sim \Gammainv(\alpha+v)$ for $i\geq 3$, and $\wweight_{i,2}\sim \Gammainv(\alpha-v)$ for $i\geq 3$. We see that in the term $\ell=2$, in \eqref{eq:explicittworowinitialdata}, the weight $\wweight_{2,2}$ simplifies, so that
	\begin{equation}\label{eq:relationtozuv}
		\left( \frac{\zd^{\rm stat}_{u,v}(2+k,2)}{\zd^{\rm stat}_{u,v}(2,2)}\right)_{k\in \Z_{\geq 0}} \overset{(d)}{=} \left( \zd_{u,v}(k) \right)_{k\in \Z_{\geq 0}}
	\end{equation}
	where the process  $\zd_{u,v}$ is defined in \eqref{eq:defzuv}.
	
	Now define, for all $n\geq m\geq 0$,
	$\zd(n,m)=\frac{\zd^{\rm stat}_{u,v}(n+2,m+2)}{\zd^{\rm stat}_{u,v}(2,2)}$ and observe that $\zd(n,m)$ satisfies the recurrence relation \eqref{eq:discreteSHE} with initial data $\zd(\cdot,0) \overset{(d)}{=}\zd_{u,v}(\cdot)$. Hence, Theorem \ref{theo:loggammaintro} follows immediately from Proposition \ref{prop:stationarity2}. Note that the statement of Proposition \ref{prop:stationarity2} assumes that $v<u$ so that  Theorem \ref{theo:loggammaintro} is established for any  $\alpha\in \R_{>0}$ and $u,v\in \R$ such that $u>-\alpha, v\in (-\alpha, \alpha)$ and $v<u$. In order to treat the case $v=u\leq 0$, when the process $\zd_{u,v}$ is a $\Gammainv(\alpha-u)$  multiplicative random walk, we can appeal to Proposition \ref{prop:stationarity} instead of Proposition \ref{prop:stationarity2} or argue by continuity from the case $v<u$. This is because when $v=u$, the weight $\wweight_{2,2}$ diverges. Thus, it is only paths that go through this point that contribute and this reduces the two-row partition function ratios to those of the one-row partition functions.
\end{proof}

We close by demonstrating two of the claims in Remark \ref{rem:discreteprops}.
%
\begin{lemma}
	When $u+v=0$, the process $\zd_{u,v}(k)$ has the law of a $\Gammainv(\alpha-u)$ multiplicative random walk.
	This implies  that
	$$
	\left\{\frac{\zd_{u,v}^{\rm stat}(n+1,m+1)}{\zd_{u,v}^{\rm stat}(2,2)}\right\}_{n,m\in \Z_{\geq 1}}\overset{(d)}{=}\left\{\zd_{u}^{\rm stat}(n,m)\right\}_{n,m\in \Z_{\geq 1}}.
	$$
	\label{prop:specialcase}
\end{lemma}
\begin{proof}
	We start with the proof that $\zd_{u,v}$ is a $\Gammainv(\alpha-u)$  multiplicative random walk. Recall that  $\zd_{u,v}(k) = \frac{\zd(2+k,2)}{\wweight_{1,1}\wweight_{1,2}\wweight_{2,2}}$ where $\zd(2+k,2)$ is as in Definition \ref{def:loggammahalfspace} with $\alpha_{\circ}=u, \alpha_1=v, \alpha_2=-v$ and $\alpha_i=\alpha$ for $i\geq 3$. Observe that the  increment ratios (that we define as in \eqref{eq:defincrements} replacing $\zd_u^{\rm stat}$ by $\zd(2+k,2)$) are such that
	$ V_{2,2}\sim \Gammainv(u-v)$ and $U_{i,1}\sim  \Gammainv(\alpha+v)$ for $i\geq 3$.
	Since the weights $\wweight_{1,1}$ is always canceled in the ratios, we may simply assume that $\wweight_{1,1}\equiv1$ (so that we do not need to assume that $u+v>0$).	Now, assuming that $u+v=0$, the vertical increment ratio  $ V_{2,2}$ becomes such that $ V_{2,2}\sim \Gammainv(2u)$ (because  $\overline{\omega}_{2,2}\sim \mathrm{Gamma}^{-1}(2u)$); the horizontal increment ratios along the first row are independent with $U_{i,1}\sim  \Gammainv(\alpha-u)$ (because $\wweight_{i,1}\sim \Gammainv(\alpha-u)$ for all  $i\geq 3$); and the weights along the second row are $\wweight_{i,2}\sim \Gammainv(\alpha+u)$, and also independent. This parametrization of weights allows to apply Lemma \ref{lem:seppalainenstationary}, which tell us that  horizontal increments on the second row are independent, and for all $i\geq 3$, $U_{i,2}\sim  \Gammainv(\alpha-u)$. This proves that $\zd_{u,v}$ is a $\Gammainv(\alpha-u)$ multiplicative random walk.
	In terms of stationary partition functions, the distribution of ratio increments $U_{i,2}$ that we have just determined means  that, when $u+v=0$,
	$\left\{\frac{\zd_{u,v}^{\rm stat}(n+1,2)}{\zd_{u,v}^{\rm stat}(2,2)}\right\}_{n\in \Z_{\geq 1}}\overset{(d)}{=}\left\{\zd_{u}^{\rm stat}(n,1)\right\}_{n\in \Z_{\geq 1}}$.
	The partition functions $\zd_{u,v}^{\rm stat}$ and $\zd_{u}^{\rm stat}$ satisfy the same recurrence relation and we have just matched their initial data, hence we obtain the claimed distributional equality to complete the lemma.
\end{proof}

\begin{lemma}
	The law of the process $\zd_{u,v}(k)$ remains the same if we change $v$ into $-v$.
	\label{rem:symmetryvtominusv}
\end{lemma}
\begin{proof}
	Recall the definition of $\zd_{u,v}$ from \eqref{eq:relationtozuv}. The statement then follows from the symmetry from Lemma \ref{lem:symmetry} since in the definition of $\zd^{\rm stat}_{u,v}$, the parameter attached to the first row is $\alpha_1=v$, and to the second row is $\alpha_2=-v$.
\end{proof}
This provides another route to prove Lemma \ref{prop:specialcase} since it suffices to treat the case where $u-v=0$ instead of assuming $u+v=0$. As noted at the end of the proof of  Theorem \ref{theo:loggammaintro}, when $u=v$, the factor $1/\wweight_{2,2}$ in \eqref{eq:explicittworowinitialdata} vanishes, so that only the term $\ell=2$ remains in the sum (see also the definition of $\zd_{u,v}$ in \eqref{eq:defzuv}) and the process is clearly a  $\Gammainv(\alpha-u)$ multiplicative random walk.

\section{Stationary half-space LPP}
\label{sec:LPP}

Geometric and exponential LPP in a half-space have been studied since early work of \cite{rains2000correlation, baik2001algebraic, baik2001asymptotics, sasamoto2004fluctuations} and more recently in \cite{baik2018pfaffian, betea2018free, bisi2019transition, betea2020stationary}. Our log-gamma results yield results for exponential LPP via a limit transition, and our method applies to the geometric case as well (though it does not occur as a direct limit of the polymer model).

\begin{definition}
	Let  $q_{\circ}, q_1, q_2, \dots $ be positive real parameters  such that $q_{\circ}q_i\in (0,1)$ for all $i \geqslant 1$ and $q_iq_j\in (0,1)$ for all $i\neq j\geqslant 1$. Let  $\big(g_{i,j}\big)_{i\geqslant j}$ be independent random variables with $g_{i,j}\sim \mathrm{Geo}(q_iq_j)$ for $i>j$  and $g_{i,i}\sim \mathrm{Geo}(q_{\circ}q_i)$.  $\mathrm{Geo}(q)$ denotes a geometrically distributed random variable with success parameter $p=1-q$ taking values in $\Z_{\geq 0}$.
	
	Let  $a_{\circ}, a_1, a_2, \dots $ be real parameters  such that $a_{\circ}+a_i>0$ for all $i \geqslant 1$ and $a_i+a_j>0 $ for all $i\neq j\geqslant 1$. Let  $\big(e_{i,j}\big)_{i\geqslant j}$ be a family of independent random variables such that for $i>j, e_{i,j}\sim \mathrm{Exp}(a_i+a_j)$ and $e_{i,i}\sim \mathrm{Exp}(a_{\circ}+a_i)$. The notation  $\mathrm{Exp}(a)$ denotes an exponentially distributed random variable with rate parameter $a$.

	Define the half-space geometric/exponential last passage times by
	$$G_{q_{\circ}, \mathbf q}(n,m) := \max_{\pi: (1,1) \to (n,m)} \sum_{(i,j)\in \pi} g_{i,j}, \qquad E_{a_{\circ}, \mathbf a}(n,m) := \max_{\pi: (1,1) \to (n,m)} \sum_{(i,j)\in \pi} e_{i,j},$$
	where $\max$ is over
	up-right paths from  $(1,1)$ and $(n,m)$ in the octant $ \lbrace (i,j)\in \Z_{>0}^2 : i\geqslant j \rbrace $.
	\label{def:LPPhalfspace}
\end{definition}
Similarly to Section \ref{sec:loggamma}, we define two types of stationary models\footnote{We could have defined  $G_{r,s}^{\rm stat}$ (resp. $ E_{u,v}^{\rm stat}$) without substracting the weight at vertex $(1,1)$, to be more consistent with the definition of $\zd_{u,v}^{\rm stat}$. However, such a choice would require to restrict to $rs<1$ (resp. $u+v>0$). To avoid introducing further notation to state Proposition \ref{prop:stationaryLPP} with the full range of parameters, it seemed more convenient to remove the weight at $(1,1)$. }, for each LPP model:
\begin{align*}
	G_r^{\rm stat}(n,m) &= G_{q_{\circ}, \mathbf q}(n,m)-g_{1,1} \text{ with }q_{\circ}=r, q_1=1/r, q_i=q \text{ for }i\geq 2, \\
	G_{r,s}^{\rm stat}(n,m)&= G_{q_{\circ}, \mathbf q}(n,m)-g_{1,1} - g_{2,1}\text{ with }q_{\circ}=r, q_1=s, q_2=1/s,  q_i=q \text{ for }i\geq 3,\\
	E_{u}^{\rm stat}(n,m) &= E_{a_{\circ}, \mathbf a}(n,m)-e_{1,1} \text{ with }a_{\circ}=u, a_1=-u, a_i=a \text{ for }i\geq 2, \\
	E_{u,v}^{\rm stat}(n,m)&= E_{a_{\circ}, \mathbf a}(n,m)-e_{1,1} - e_{2,1}\text{ with }a_{\circ}=u, a_1=v, a_2=-v,  a_i=a \text{ for }i\geq 3.
\end{align*}

\begin{proposition} Let $k\in \mathbb Z_{\geq 1}$ and consider points $\mathbf p_1=(n_1, m_1), \dots, \mathbf p_k=(n_k,m_k)$ along a down-right path in the octant.  Say that $M\in \left\lbrace G^{\rm stat}_{r}, E_{u}^{\rm stat}, G_{r,s}^{\rm stat},  E_{u,v}^{\rm stat}\right\rbrace$ is stationary if the law of
	$ \Big( M\big((m,m)+\mathbf p_i)\big)-M\big((m,m)+\mathbf p_1)\big)\Big)_{1\leq i\leq k}  $
	does not depend on $m$. Then, we have the following results:
	%
	\begin{enumerate}%
		\item[(1)] For $q\in (0,1)$ and $r>0$ such that $qr, q/r\in (0,1)$,  $G^{\rm stat}_{r}$ is stationary.
		\item[(2)]  For	$a>0$ and $u\in \mathbb R$ such that $\alpha\pm u>0$, $E^{\rm stat}_{u}$ is stationary.
		\item[(3)] For $q\in (0,1)$ and $r,s>0$ such that $qr, qs, q/s, r/s \in (0,1)$, $G^{\rm stat}_{r,s}$ is stationary.
		\item[(4)]  For $a>0$ and $u,v\in \mathbb R$ such that $a+u>0, a\pm v>0$ and $u-v>0$,  $ E^{\rm stat}_{u,v}$ is stationary.
	\end{enumerate}
	\label{prop:stationaryLPP}
\end{proposition}
\begin{remark}
	In Proposition \ref{prop:stationaryLPP}, the stationary processes from  (1) and (2) are special or limiting cases of those from (3) and (4). Previously, the only known stationary measures for  half-space LPP with  exponential  weights were those in (2) (see \cite[Lemma 2.1 and 2.2]{betea2020stationary}) where the increments of the stationary measure are exponential distributed.
\end{remark}
\begin{proof}
	The exponential cases, parts (2) and (4), are  limits of the results in Section \ref{sec:loggamma}. Indeed, if $\wweight$ is an inverse gamma random variable with parameter $\epsilon a$, then $\epsilon \log \wweight$ weakly converges to an exponential random variable with rate parameter $a$. This implies that, choosing the parameters in Definition as \ref{def:loggammahalfspace}  $\alpha_{\circ}=\epsilon a_{\circ}$, $\alpha_i=\epsilon a_i$, then  $\epsilon \log \zd(n,m)$ weakly converges to $E_{a_{\circ}, \mathbf a}(n,m)$, jointly for all $n,m$ in any finite domain.
	
	The results about the geometric model do not follow from Section \ref{sec:loggamma}. However, part (1) of Proposition \ref{prop:stationaryLPP} follows from the same argument as in \cite[Lemma 2.2]{betea2020stationary}, and part (3) of Proposition \ref{prop:stationaryLPP} follows from the same proof as in Proposition \ref{prop:stationarity2}, replacing $\zd(n,m)$ by $e^{G(n,m)}$ mutatis mutandis. The only part that cannot be straightforwardly adapted is the geometric LPP analogue of Lemma \ref{lem:symmetry}: for any integers  $m\in \Z_{\geq 1}$ and  $k\in \Z_{\geq 0}$,  the joint distribution of
	\begin{equation}
		\left(G_{q_{\circ}, \mathbf q}(m,m),G_{q_{\circ}, \mathbf q}(m+1,m), \dots,G_{q_{\circ}, \mathbf q}(m+k,m) \right)
		\label{eq:jointdistributiongeometric}
	\end{equation}
	is invariant with respect to permutations of $q_1, \dots, q_m$. The distribution of \eqref{eq:jointdistributiongeometric} is related to Pfaffian Schur processes in a way very similar as \eqref{eq:jointdistribution} is related to half-space Macdonald processes (see \cite[Proposition 3.10]{baik2018pfaffian}).
	Hence the symmetry follows from the same reason as in the proof of Lemma \ref{lem:symmetry}. Alternatively, an explicit formula for the distribution of \eqref{eq:jointdistributiongeometric} is given in \cite[Proposition 4.2]{baik2018pfaffian} and one may check that the expression is symmetric.
\end{proof}

\section{Proof of Theorem \ref{theo:invariantKPZintro}}
\label{sec:KPZ}

In this Section, we deduce results about invariant measures for the KPZ equation from the invariant measures of the log-gamma polymer, through a scaling limit. Namely, we prove Theorem \ref{theo:invariantKPZintro} as well as some of the claims in Remark \ref{rem:specialcases}.
This will all rely on the following theorem that proves convergence of the stationary half-space log-gamma partition functions to a specific solution of the stochastic heat equation \eqref{eq:SHEintro}. For each $n\in \Z_{\geq 1}$, we scale the bulk parameter $\alpha=\alpha^{(n)}$ in stationary log-gamma partition functions as  $\alpha^{(n)}=\frac{1}{2}+\sqrt{n}$ and define the scaled partition functions as follows: For all $T,X\in \R_{\geq 0}$ with $\frac{nT}{2},n^{1/2} X\in \Z_{\geq 0}$, we let
\begin{equation}
	e^{\HH^{(n)}_{u,v}(T,X)}:= (\sqrt{n})^{nT +n^{1/2} X} \frac{\Zpolymer_{u,v;\alpha^{(n)}}^{\rm stat}\left(\frac{n T}{2}+n^{1/2}X+2, \frac{n T}{2}+2 \right)}{\wweight_{1,1}\wweight_{2,2}},\label{eq:Znuv}
\end{equation}
where we recall that
$\Zpolymer_{u,v}^{\rm stat}$ was defined in \eqref{eq:defZuvstat}, and we have added a subscript $\alpha^{(n)}$ to make the dependence in $\alpha$ explicit. For all other values of $T,X\in \R_{\geq 0}$ we linearly interpolate $\exp\big(\HH^{(n)}_{u,v}(T,X)\big)$ and then define $\HH^{(n)}_{u,v}(T,X)$ by taking logarithms.

\begin{theorem}
	Fix any $u,v\in \mathbb R$ with $v\leq \min\lbrace u,0\rbrace$.
	Then, for any $T>0$, any $k\in \Z_{\geq 1}$ and  $0\leq X_1< \cdots< X_k$, we have the following convergence in finite dimensional distributions:
	\begin{equation}
		\Big(\HH^{(n)}_{u,v}(T,X_1), \ldots, \HH^{(n)}_{u,v}(T,X_k) \Big)\xRightarrow[n\to\infty]{} \Big(\HH_{u,v}(T,X_1), \ldots, \HH_{u,v}(T,X_k) \Big)
	\end{equation}
	where $\HH_{u,v}(T,X)$ solves \eqref{eq:KPZ} with initial data $\HH_{u,v}(0,X)=\HH_{u,v}(X)$, defined in  \eqref{eq:defhuv}.
	\label{prop:KPZconvergence}
\end{theorem}
Theorem \ref{prop:KPZconvergence} is proved in Section \ref{sec:proofconvergence}, using more general convergence results from Section \ref{sec:generalconvergence}.
The scaling \eqref{eq:Znuv} is such that at $T=0$, the rescaled partition functions converge to the initial data $\HH_{u,v}(X)$.  Indeed, we have the following convergence.
\begin{lemma}\label{prop:limitinitialcondition}
	For any $u,v\in \mathbb R$ with $v\leq \min \lbrace u,0\rbrace$, we have the weak convergence
	\begin{equation}
		\HH^{(n)}_{u,v}(0,X)\xRightarrow[n\to\infty]{} \HH_{u,v}(X),\label{eq:cvinitialuv}
	\end{equation}
	as processes in $X\in \R_{\geq 0}$, where  $\HH_{u,v}$ is defined in \eqref{eq:defhuv}. Additionally, there exists a constant $C=C(u,v)>0$ such that for all $n\in \Z_{\geq 1}$ and $X\in \R_{\geq 0}$,
	\begin{equation}\label{eq:L2tildeZ}
		\EE\left[\left(e^{\HH^{(n)}_{u,v}(0,X)}\right)^2\right] \leq Ce^{CX}.
	\end{equation}
	\label{lemma:limitinitialcondition}
\end{lemma}
\begin{proof}
	Using \eqref{eq:Znuv} and \eqref{eq:relationtozuv}, we have, jointly for all $X$ such that $n^{1/2}X\in \mathbb Z_{\geq 0}$, 
	\begin{equation}\label{eq:ztildeuv}
		e^{\HH^{(n)}_{u,v}(0,X)} \overset{(d)}{=} (\sqrt{n})^{n^{1/2} X} r_2(n^{1/2}X)  + \frac{1}{\wweight \sqrt{n}} \sum_{\ell=1}^{n^{1/2} X} \sqrt{n}^{\ell} r_1(\ell) \times   (\sqrt{n})^{n^{1/2} X+1-\ell} \frac{r_2(n^{1/2}X)}{r_2(\ell-1)},
	\end{equation}
	where we recall that $r_1$ and $r_2$ are independent $\Gammainv(\alpha^{(n)}+v)$ and $\Gammainv(\alpha^{(n)}-v)$ (resp.) multiplicative random walks with $r_1(0)=r_2(0)=1$, and $\wweight\sim \Gammainv(u-v)$ is independent. When $u=v$, we simply assume that the second term in \eqref{eq:ztildeuv} equals zero.  Since  $\alpha^{(n)}=\frac 1 2 + \sqrt{n}$, using \eqref{eq:loggamma} and the asymptotic expansions for the digamma and trigamma functions we see that if $\wweight^{(n)}\sim\Gammainv(\alpha^{(n)}+v)$,
	\begin{align*}
		\mathbb E[\log \wweight^{(n)}] &= -\psi(\alpha^{(n)}+v) =-\log(\sqrt{n}) -\frac{v}{\sqrt{n}} +O\left(\frac{1}{n}\right),\\
		\mathrm{Var}[\log \wweight^{(n)}] &= \psi'(\alpha^{(n)}+v) = \frac{1}{\sqrt{n}}-\frac{v}{n} +O\left( \frac{1}{n^{3/2}} \right).
	\end{align*}
	Donsker's theorem implies that for $B_1$ and $B_2$ independent standard Brownian motions with drift $-v$ and $v$ (resp.) we have the joint convergence as processes (in $0\leq Y\leq X$) of
	\begin{multline*}
		\left((\sqrt{n})^{n^{1/2} X} r_2(n^{1/2}X),\quad
		(\sqrt{n})^{n^{1/2} Y} r_1(n^{1/2}Y),\quad
		(\sqrt{n})^{n^{1/2}X-n^{1/2} Y+1}  \frac{r_2(n^{1/2}X)}{r_2(n^{1/2}Y-1)} \right) \\
		\xRightarrow[n\to\infty]{} \left(e^{B_2(X)}, e^{B_1(Y)}, e^{B_2(X)-B_2(Y)}\right).
	\end{multline*}
	As in \cite[Proposition 1.6]{auffinger2012universality}, this implies that $\HH_{u,v}^n(0,X)$ converges as a process in $X$ to
	$$
	\log \left(e^{B_2(X)} + \frac{1}{\wweight} \int_0^X e^{B_1(Y)+B_2(X)-B_2(Y)} dY \right)=\HH_{u,v}(X).
	$$
	Turning to \eqref{eq:L2tildeZ}, let $M_k(\pm v):=\mathbb E\left[\left(\wweight^{(n)}\right)^k\right]$ where $\wweight^{(n)}\sim\Gammainv(\alpha^{(n)}\pm v)$, and write
	\begin{multline*}
		\EE\left[\left(e^{\HH^{(n)}_{u,v}(0,X)}\right)^2\right]= (nM_2(-v))^{n^{1/2}X}     \\
		+ \frac{\EE[1/\wweight]}{\sqrt{n}} \sum_{\ell=1}^{n^{1/2}X} \frac{1}{\sqrt n M_1(-v)} (nM_1(v)M_1(-v))^{\ell}  (nM_2(-v))^{n^{1/2}X+1-\ell} \\
		+\frac{\EE[1/\wweight^2]}{n} \sum_{\ell,\ell'=1}^{n^{1/2}X}  (n M_2(v))^{\min\lbrace \ell, \ell'\rbrace}  (nM_1(v)M_1(-v))^{\vert \ell-\ell'\vert}(nM_2(-v))^{n^{1/2}X+1-\max\lbrace \ell,\ell'\rbrace}.
	\end{multline*}
	Note that using \eqref{eq:inversegammamoments}, we have $\sqrt n M_1(-v)=1+(v+1/2)n^{-1/2}+O(n^{-1})$,  $nM_2(\pm v) = 1+ (2\mp 2v) n^{-1/2} + O(n^{-1})$  and
	$nM_1(v)M_1(-v)= 1+ n^{-1/2} +O(n^{-1})$, along with the fact that $\EE[1/\wweight]$ and $\EE[1/\wweight^2]$ are constant as $n$ changes. Thus the expansion above is bounded by $Ce^{CX}$ (note that the summations are compensated exactly by the prefactors involving $\sqrt{n}$ and $n$). This proves \eqref{eq:L2tildeZ}.
\end{proof}

\begin{proof}[Proof of Theorem \ref{theo:invariantKPZintro}]
	Fix any $u,v\in\R$ with $v\leq \min\lbrace u,0\rbrace$ and let $\HH_{u,v}(T,Y)$ denote the solution to \eqref{eq:KPZ} with initial data  $\HH_{u,v}(X)$ from \eqref{eq:defhuv}. To prove the theorem we must show that for any $T\in \R_{\geq 0}$, the law of the process $\HH_{u,v}(T,\cdot)-\HH_{u,v}(T,0)$ is the same as that of $\HH_{u,v}(\cdot)$. Since both of these are continuous processes, it suffices to show that in terms  of finite-dimensional distributions, namely that for any $k\in \Z_{\geq 1}$ and $Y_{1},\ldots, Y_{k}\in \R_{\geq 0}$,
	\begin{equation}\label{eq:HHdisteq}
		\left\{\HH_{u,v}(T,Y_j)-\HH_{u,v}(T,0)\right\}_{j\in \{1,\ldots, k\}}\overset{(d)}{=}  \left\{\HH_{u,v}(Y_j)\right\}_{j\in \{1,\ldots, k\}}
	\end{equation}
	Now, observe that as a consequence of Proposition \ref{prop:stationarity2} (or Theorem \ref{theo:loggammaintro}),
	\begin{equation}
		\left\{\HH^{(n)}_{u,v}(T,Y_j)-\HH^{(n)}_{u,v}(T,0)\right\}_{j\in \{1,\ldots, k\}}\overset{(d)}{=}  \left\{\HH^{(n)}_{u,v}(0,Y_j)\right\}_{j\in \{1,\ldots, k\}}
	\end{equation}
	where $\HH^{(n)}_{u,v}(T,Y)$ is defined in \eqref{eq:Znuv}. By Theorem \ref{prop:KPZconvergence} and Lemma \ref{prop:limitinitialcondition}, the sides of the above equality in distribution converge in distribution to the corresponding sides of \eqref{eq:HHdisteq}.\end{proof}

The following two properties of $\HH_{u,v}$ (see Remark \ref{rem:specialcases}) respectively follow from the convergence in Lemma  \ref{lemma:limitinitialcondition} along with the results of  Lemma \ref{prop:specialcase} and  Lemma \ref{rem:symmetryvtominusv} (resp.).
\begin{lemma}
	\label{lemma:Brownianity}
	For $u=-v\geq 0$, $\HH_{u,v}$ is a Brownian motion with drift $u$.
\end{lemma}
\begin{lemma}\label{lemma:symmetryHuv}
	If $v>0$ and $u\geq v$, then  $\HH_{u,v}$ and $ \HH_{u,-v}$ have the same distribution.
\end{lemma}

\section{Convergence of half-space polymer partition functions to the half-space SHE}
\label{sec:generalconvergence}

We will formulate and prove a general result about intermediate disorder scaling limits for half-space polymer models. This generalizes previous works of \cite{wu2018intermediate} (which only studied point to point partition functions) \cite{parekh2019positive} (which dealt with Dirichlet type boundary conditions instead of Neumann type) and the full-space work of \cite{alberts2014intermediate}.

\subsection{Notation and definitions}\label{sec:notanddefs}

We use  the convention of uppercase letters for continuum variables and lower case letters for discrete variable. For two topological spaces $E$ and $F$ let $C(E,F)$ denote the space of continuous functions from $E$ to $F$.
As a special case, let $\Cfive$ denote the space of all real-valued continuous functions $(\mu,\beta,S,X,T,Y)\mapsto f_{\mu,\beta}(S,X;T,Y)$ defined in the domain $\Dfive\subset \R^6$ defined by $\mu,\beta,S,T\in \R$ with $S<T$ and $X,Y\in \R_{\geq 0}$. Equip $\Cfive$ with the topology of uniform convergence on compact sets (this means that the two time variables $S$ and $T$ will always maintain a positive distance for such compact sets). Let $(C_i)_{i\in \Z_{\geq 1}}$ denote a sequence of growing  compact sets whose union equals $D$ and let
$$d_{C_i}(f,g)= \sup_{(\mu,\beta,S,X,T,Y)\in C_i} |f_{\mu,\beta}(S,X;T,Y)-g_{\mu,\beta}(S,X;T,Y)|.$$
Then the  uniform convergence on compact sets topology is metrizable with the metric
\begin{equation}\label{eq:metric}
	d(f,g):= \sum_{i=1}^{\infty} 2^{-i} \max\lbrace  d_{C_i}(f,g),1\rbrace.
\end{equation}
We denote weak convergence of a sequence $f^{(n)}$ of random functions in $\Cfive$ to a limit (in distribution) $f$ by $f^{(n)}\Rightarrow f$ and we will refer below to this notion of convergence as ``convergence as a process in $\Cfive$''. We can similarly define process convergence and a metric for the uniform convergence on compact sets topology when we are dealing with fewer variables, e.g. when $\mu$ and $\beta$ are fixed, or when $(T,Y)$ are the only variables. We use dots in place of variables to denote the variables which are allowed to vary in our notion of process convergence, e.g. if $\mu$ is fixed we would write $f^{(n)}(\mu,\cdot,\cdot,\cdot,\cdot,\cdot)\Rightarrow f(\mu,\cdot,\cdot,\cdot,\cdot,\cdot)$ for the convergence of the process in the remaining five variables. We will also allow some of the variables to sit as subscripts, e.g. $\ZZ^{(n)}_{\mu,\beta}(S,X;T,Y)$.

\begin{definition}\label{def:SHEproper}
	For $u\in \R$, a random process $\HH_u(T,Y)\in C(\R_{\geq 0}\times \R_{\geq 0},\R)$ is said to be the Hopf-Cole solution to \eqref{eq:KPZ}, the half-line KPZ equation with Neumann boundary parameter $u$, with (potentially random) initial data $\HH(0,X)\in C(\R_{\geq 0},\R)$ if $\HH_u(T,Y)=\log \ZZ_{u-1/2,1}(T,Y)$, where $\ZZ_{\mu,\beta}(T,Y)\in C(\R_{\geq 0}\times \R_{\geq 0},\R_{> 0})$ solves \eqref{eq:SHEbeta}, the half-line stochastic heat equation with Robin boundary condition $\mu\in \R$ and noise strength $\beta\in \R$.
	\begin{align}
		\label{eq:SHEbeta}
		\tag{SHE$_{\mu,\beta}$}
		\begin{split}
			\partial_T \ZZ_{\mu,\beta}(T,X) &= \frac{1}{2} \partial^2_{X}\ZZ_{\mu,\beta}(T,X) + \ZZ_{\mu,\beta}(T,X) \beta \xi(T,X), \\
			\partial_X\ZZ_{\mu,\beta}(T,X)\big\vert_{X=0} &=\mu \ZZ_{\mu,\beta}(T,0),
		\end{split}
	\end{align}
	with initial data given by  $\ZZ_{\mu,\beta}(0,X)=e^{ \HH(0,X)}$ for $X\in \R_{\geq 0}$, where $\HH(0,X)$ is initial data for \eqref{eq:KPZ}.
	When $\beta=1$ we have that \eqref{eq:SHEbeta} matches with \eqref{eq:SHEintro}.
	Solving \eqref{eq:SHEbeta} means that $\ZZ_{\mu,\beta}$ satisfies the mild form of the SHE: For all $T\in \R_{>0}$ and $Y\in \R_{\geq 0}$, $\ZZ_{\mu,\beta}(T,\cdot)$ is adapted to the natural filtration generated by $\ZZ_{0}(0,\cdot)$ and $\xi$ on the time interval $[0,T]$, and almost-surely satisfies
	\begin{equation}
		\ZZ_{\mu,\beta}(T,Y) = \int_{\R_{\geq 0}} \PC_{\mu}(0,X;T,Y) \ZZ_{0}(X)dX   + \int_{0}^T \int_{\R_{\geq 0}} \PC_{\mu}(S,X;T,Y)\ZZ_{\mu,\beta}(S,X)\, \beta\, \xi(dS,dX)
		\label{eq:mild}
	\end{equation}
	where the stochastic integral is in It\^{o} sense and $\PC_\mu(S,X;T,Y)$ is the Robin heat kernel on the half-line with parameter $\mu$, i.e., the unique solution to
	\begin{align*}
		\partial_T \PC_\mu(S,X;T,Y) &= \tfrac{1}{2} \partial^2_{Y}\PC_\mu(S,X;T,Y),\\
		\partial_Y \PC_\mu(S,X;T,Y)\big\vert_{Y=0} &= \mu \PC_\mu(S,X;T,0), \;\; \forall S<T\in\R_{>0}, \forall X\in \R_{\geq 0},\\
		\lim_{T\to S}\PC_\mu(S,X;T,Y) &= \delta(X-Y) \text{ in a weak sense on }L^2(\mathbb R_{\geq 0}).
	\end{align*}
	For use below, for any $k\in \Z_{\geq 1}$ and $S,T\in \R$ with $S<T$ define
	$$
	\Delta_k(S,T) = \big\{\vec{R}=(R_1,\ldots,R_k)\in \R^k_{>0}:S< R_1<\cdots<R_k<T\big\}
	$$
	and for $\vec{R}\in \Delta_k(S,T)$ and $\vec{W}\in \R_{\geq 0}^k$ define the $k$-point transition probability
	\begin{equation}\label{eq:PCmulti}
		\PC_{\mu}(S,X;\vec{R},\vec{W};T,Y):= \prod_{j=1}^{k+1} \PC_{\mu}(R_{j-1},W_{j-1};R_j,W_j)
	\end{equation}
	with the convention that $R_0=S, R_{k+1}=T, W_0=X$ and $W_{k+1}=Y$. Finally, write
	$$
	\xi(d\vec{R},d\vec{W}):= \prod_{j=1}^{k} \xi(dR_j,dW_j).
	$$
\end{definition}

\begin{proposition}
	Suppose there exists an $a>0$ such that $\ZZ_{0}$ satisfies
	\begin{equation}\label{eq:initdatabdd}
		\sup_{X\in \R_{\geq 0}} e^{-aX} \EE[\ZZ_{0}(X)^2] <\infty.
	\end{equation}
	Then there exists a unique mild solution $\ZZ_{\mu,\beta}(T,Y)$ to the SHE  \eqref{eq:mild} in the class of adapted (to the initial data and white noise natural filtration) processes satisfying for all $T_0\in \R_{>0}$
	\begin{equation}\label{eq:SHEL2bdd}
		\sup_{Y\in \R_{\geq 0}, T\in [0,T_0]} e^{-aY} \EE[\ZZ_{\mu,\beta}(T,Y)^2] <\infty.
	\end{equation}
	Assuming $\ZZ_{0}(\cdot)$ is almost surely positive, so is $\ZZ_{\mu,\beta}(\cdot,\cdot)$. Finally, $\ZZ_{\mu,\beta}$ admits a {\it chaos series}
	\begin{equation}\label{eq:SHEchaos}
		\ZZ_{\mu,\beta}(T,Y) =\sum_{k=0}^{\infty} \beta^k \int\limits_{\Delta_k(0,T)}\int\limits_{\R_{\geq 0}^{k+1}} \ZZ_{0}(X_0)dX_0 \cdot \PC_{\mu}(0,X_0;\vec{R},\vec{W};T,Y)\cdot\xi(d\vec{R},d\vec{W}),
	\end{equation}
	where the term $k=0$ is the same as the first term in the right-hand side of \eqref{eq:mild}. 
\end{proposition}

\begin{proof}
	Except for the chaos series, \cite[Proposition 4.2]{parekh2017kpz} proves this, using a now classical argument that dates back to \cite{walsh1986introduction}. It is easily shown by the type of arguments in the proof of \cite[Proposition 4.2]{parekh2017kpz} that the chaos series satisfy the mild form of the SHE \eqref{eq:mild}  and \eqref{eq:SHEL2bdd}, provided \eqref{eq:initdatabdd}. By uniqueness, this implies \eqref{eq:SHEchaos}.
\end{proof}

We will also need to work with the four-parameter time-space SHE sheet $\ZZ_{\mu,\beta}(S,X;T,Y)$.

\begin{proposition}
	\label{prop:sheet} For any $\mu,\beta,S,T\in \R$ with $S<T$ and $X,Y\in \R_{\geq 0}$ (i.e. \break $(\mu,\beta, S,X,T,Y)\in \Dfive$), there exists a unique mild solution
	\begin{equation}
		\ZZ_{\mu,\beta}(S,X;T,Y) = \PC_{\mu}(S,X;T,Y)   + \int_{S}^T \int_{\R_{\geq 0}}\PC_{\mu}(R,W;T,Y)\ZZ_{\mu,\beta}(S,X;R,W)\beta \xi(dR,dW)
		\label{eq:milddelta}
	\end{equation}
	to the SHE with $\delta$ initial data started at time $S$ and location $X$ in the class of adapted (to the white noise natural filtration on the time interval $[S,T]$) processes satisfying for all $T_0\in\R_{>S}$
	\begin{equation}\label{eq:SHEL2bdddelta}
		\sup_{Y\in \R_{\geq 0}, T\in (S,T_0]} (T-S)\cdot \EE[\ZZ_{\mu,\beta}(S,X;T,Y)^2] <\infty.
	\end{equation}
	For all $(\mu,\beta, S,X,T,Y)\in \Dfive$, the $\ZZ_{\mu,\beta}(S,X;T,Y)$ can be defined on a common probability space as
	\begin{equation}\label{eq:SHEchaosdelta}
		\ZZ_{\mu,\beta}(S,X;T,Y) =\sum_{k=0}^{\infty}\beta^k \int\limits_{\Delta_k(S,T)}\int\limits_{\R_{\geq 0}^{k}}\PC_{\mu}(S,X;\vec{R},\vec{W};T,Y) \xi(d\vec{R},d\vec{W}).
	\end{equation}
	The family of solutions $\ZZ_{\mu,\beta}(S,X;T,Y)$ satisfies the following composition law:
	\begin{equation}\label{eq:comp}
		\ZZ_{\mu,\beta}(S,X;T,Y) = \int_{\R_{\geq 0}}\ZZ_{\mu,\beta}(S,X;R,W)\ZZ_{\mu,\beta}(R,W;T,Y)dW \qquad \textrm{for any }R\in (S,T).
	\end{equation}
	For any $\ZZ_{0}$ satisfying \eqref{eq:initdatabdd} the solution to the SHE in \eqref{eq:SHEchaos} satisfies the convolution formula:
	\begin{equation}\label{eq:conv}
		\ZZ_{\mu,\beta}(T,Y) = \int_{\R_{\geq 0}}\ZZ_{\mu,\beta}(0,X;T,Y)\ZZ_{0}(X)dX.
	\end{equation}
	Finally, the random function $(\mu,\beta,S,X,T,Y)\mapsto\ZZ_{\mu,\beta}(S,X;T,Y)$ almost surely takes values in $\Cfive$ and moreover the chaos series \eqref{eq:SHEchaosdelta} provides a monotone coupling in the $\mu$ variable whereby for any $\mu^\downarrow \leq \mu \leq \mu^\uparrow$, almost surely $\ZZ_{\mu^\downarrow,\beta}(\cdot,\cdot;\cdot,\cdot)\geq \ZZ_{\mu,\beta}(\cdot,\cdot;\cdot,\cdot)\geq \ZZ_{\mu^\uparrow,\beta}(\cdot,\cdot;\cdot,\cdot)$.
\end{proposition}
\begin{proof}
	The existence and uniqueness is given in \cite[Proposition 4.3]{parekh2017kpz}. As before, it is easy to show  (e.g. by the type of arguments in \cite{walsh1986introduction} or in the proof of \cite[Proposition 4.3]{parekh2017kpz}) that the chaos series \eqref{eq:SHEchaosdelta} satisfy the mild form of the SHE \eqref{eq:milddelta}  and \eqref{eq:SHEL2bdddelta}. By uniqueness, this implies \eqref{eq:SHEchaosdelta} and this gives the desired coupling of all solutions  $\ZZ_{\mu,\beta}(S,X;T,Y)$ on the common probability space supported by the white noise $\xi$. That $\ZZ_{\mu,\beta}(S,X;T,Y)$ is continuous in all six parameters follows easily from estimating moments as each variable is changed (which is achieved with estimates on the heat kernels  $\PC_{\mu}(S,X;T,Y)$) and then applying the Kolmogorov continuity criteria. The convolution law \eqref{eq:conv} and composition law \eqref{eq:comp} both follow easily from performing the desired integrations to the chaos series and matching them with the chaos series for the desired outcome, see for instance the proof of \cite[Proposition 2.9]{corwin2021kpz} for details. The monotonicity is a bit trickier to see since it does not hold term-by-term in the chaos series. Instead, for instance, it can be seen as a consequence of our convergence result, Theorem \ref{thm:SHEsheetlimit} in the case of deterministic $\X^{(n)}_r\equiv 1-\mu n^{-1/2}$  along with the monotonicity of the polymer partition function observed in \eqref{eq:mono}. It should be noted that this monotonicity result is used in the proof of Theorem \ref{thm:SHEsheetlimit}, though only in the case where the boundary parameters $\X^{(n)}_r$ are random (thus the logic is not circular).
\end{proof}

\subsection{Convergence of the discrete partition function to the SHE}\label{sec:discreteconv}
For $s\in \Z$ and $x\in \Z_{\geq 0}$ let $\PP_{R}^{s,x}$ denote the probability measure on simple symmetric random walks started at time $s$ and position $x$, and reflected at the origin. Equivalently, this measure is that of the absolute value of a simple symmetric random walk started at time $s$ and position $x$. For $t\in \Z$ with $t>s$, the probability this measure assigns to a path $\SSS=(\SSS_s,\ldots,\SSS_{t})$ between times $s$ and $t$ is given by
$\PP_{R}^{s,x}(\SSS)= \prod_{r=s}^{t-1} 2^{-\mathbf{1}_{\SSS_r>0}}$ provided that $\SSS_s=x$ and $(\SSS_s,\ldots,\SSS_{t})$ forms a valid random walk path that never goes negative (otherwise the right-hand side is zero). Note that the factor $r=t$ is intentionally left off. We define a discrete transition probability
$$
p(s,x;t,y):=\PP_{R}^{s,x}(\SSS_t=y)= \sum_{\SSS:\SSS_t=y}\PP_{R}^{s,x}(\SSS),
$$
noting that this is zero unless $s+x=t+y\textrm{ mod }2$. As such, we will always assume below that the sum of the time and space variables are even and restrict ourselves to this sublattice.

We will deform this by introducing boundary weights. Let $\boldsymbol{\X}=\big(\X(i)\big)_{i\in \Z}$ be a collection of i.i.d. non-negative random variables that we call the {\it boundary weights}. With these boundary weights we define a random measure (generally not a probability measure so the use of $\PP$ is a bit of an abuse of notation) on paths $\SSS= (\SSS_s,\ldots, \SSS_t)$ between times $s$ and $t$ by setting
$$
\PP_{\boldsymbol{\X}}^{s,x}(\SSS) := \left(\prod_{s\leq i<t: \SSS_i=0} \X(i)\right)\cdot   \PP_{R}^{s,x}(\SSS).
$$
Define the associated random transition kernel (usually not a probability mass function) as
\begin{equation}\label{eq:pxsxty}
	p_{\boldsymbol{\X}}(s,x;t,y) := \PP_{\boldsymbol{\X}}^{s,x}(\SSS_t=y) = \sum_{\SSS:\SSS_t=y} \PP_{\boldsymbol{\X}}^{s,x}(\SSS).
\end{equation}
In the special case where $\X(i)\equiv \gamma\geq 0$ is deterministic, we write $p_{\gamma}(s,x;t,y)$.

We will define the modified polymer partition function by introducing {\it bulk weights}. For an inverse temperature $\beta>0$, let $\boldsymbol{\omega} = \big(\omega(s,x)\big)_{s\in \Z, x\in \Z_{> 0}}$ be a collection of random variables such that $1+\beta \omega(s,x)$ is non-negative
and  $\omega(s,x)$ are i.i.d. for $s\in \Z_{\geq 0}$ and $x\in \Z_{>0}$ and $\omega(s,0)\equiv 0$. For $s,t\in \Z$ with $s<t$ and $x,y\in \Z_{\geq 0}$, the {\it modified polymer partition function} $z_{\boldsymbol{\X},\boldsymbol{\omega}}(s,x;t,y)$ is a function of the boundary and bulk weights defined by
\begin{equation}\label{eq:zprod}
	z_{\boldsymbol{\X},\boldsymbol{\omega},\beta}(s,x;t,y) = \EE_{\boldsymbol{\X}}^{s,x}\left[\prod_{r=s}^{t-1} \big(1+ \beta \omega(r,\SSS_r)\big) \mathbf{1}_{\SSS_t=y}\right] = \sum_{\SSS:\SSS_t=y} \prod_{r=s}^{t-1} \big(1+\beta \omega(r,\SSS_r)\big) \PP_{\boldsymbol{\X}}^{s,x}(\SSS).
\end{equation}

By expanding the products, we can rewrite this as a discrete analog of the chaos series
\begin{equation}\label{eq:zseries}
	z_{\boldsymbol{\X},\boldsymbol{\omega},\beta}(s,x;t,y)  = \sum_{k=0}^{t-s} \beta^k \sum_{\vec{r}\in D_{k}(s,t)}\sum_{\vec{w}\in \Z_{\geq 0}^k}
	p_{\boldsymbol{\X}}(s,x;\vec{r},\vec{w};t,y) \omega(\vec{r},\vec{w})
\end{equation}
where
\begin{equation}\label{eq:Dk}
	D_{k}(s,t) := \{\vec{r}=(r_1,\ldots ,r_k)\in \Z^k: s\leq r_1<\cdots <r_k<t\}
\end{equation}
and for $\vec{r}\in D_{k}(s,t)$ and $\vec{w}=(w_1,\ldots, w_k)$,
\begin{equation}\label{eq:zseries2}
	p_{\boldsymbol{\X}}(s,x;\vec{r},\vec{w};t,y):= \prod_{j=1}^{k+1} p_{\boldsymbol{\X}}(r_{j-1},w_{j-1};r_{j},w_{j}), \qquad  \omega(\vec{r},\vec{w}):= \prod_{j=1}^{k} \omega(r_j,w_j),
\end{equation}
with the convention that $r_0=s, r_{k+1}=t, w_0=x$ and $w_{k+1}=y$.
In order for \break $p_{\boldsymbol{\X}}(s,x;\vec{r},\vec{w};t,y)\neq 0$, we must have the same even parity for all $r_j+w_j$ as $j$ varies, an event we denote by $\vec{r}\leftrightarrow \vec{w}$.

Another consequence of \eqref{eq:zseries}  is the discrete version of the mild formulation of the SHE:
\begin{equation}\label{eq:zdiscmild}
	z_{\boldsymbol{\X},\boldsymbol{\omega},\beta}(s,x;t,y)  =p_{\boldsymbol{\X}}(s,x;t,y) + \sum_{r=s}^{t-1} \sum_{w\in \Z_{\geq 0}} p_{\boldsymbol{\X}}(r,w;t,y)  \beta \omega(r,w)z_{\boldsymbol{\X}, \boldsymbol{\omega},\beta}(s,x;r,w).
\end{equation}
Clearly from \eqref{eq:zprod}, $z_{\boldsymbol{\X},\boldsymbol{\omega},\beta}(s,x;t,y)$ satisfies a composition law whereby for any $r\in [s,t]$,
$$
z_{\boldsymbol{\X},\boldsymbol{\omega},\beta}(s,x;t,y) = \sum_{w\in \Z_{\geq 0}} z_{\boldsymbol{\X},\boldsymbol{\omega},\beta}(s,x;r,w)z_{\boldsymbol{\X},\boldsymbol{\omega},\beta}(r,w;t,y)
$$
Furthermore, $z_{\boldsymbol{\X},\boldsymbol{\omega}}(s,x;t,y)$ allows us to construct partition functions  with initial data.

\begin{figure}
	\centering
	\scalebox{0.54}{\includegraphics{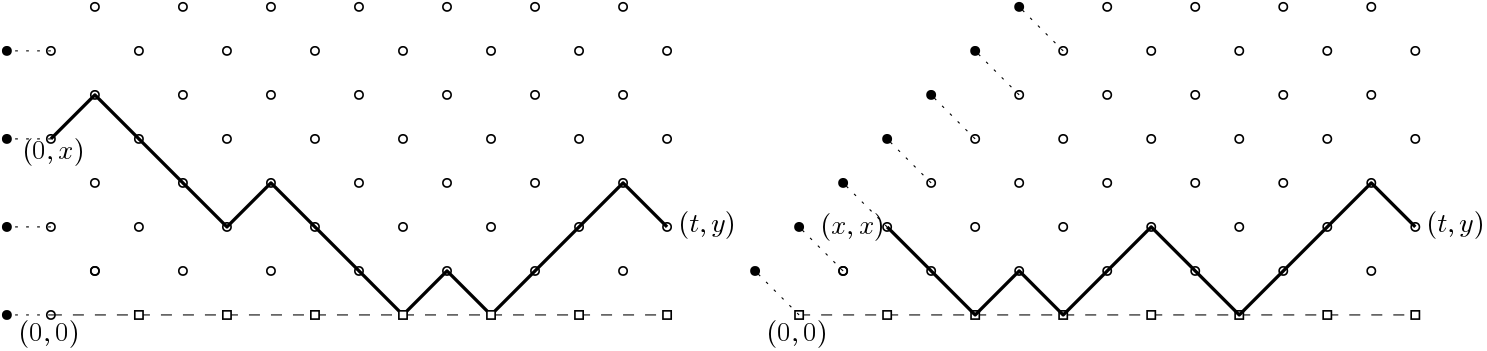}}
	\caption{Left: The partition function from \eqref{eq:ztinyvert} is depicted. The $\bullet$ at locations $(-1,x)$ represent the initial data, while the $\circ$ represent the bulk weights and the $\Box$ represent the boundary weights. Paths can start at any location $(0,x)$ for $x\in 2\Z_{\geq 0}$ and their weight is the product of the weights along the path (except for at the endpoint $(t,y)$) as well as at $(-1,x)$. The partition function is the sum over all such paths and over all such starting points. Right: The partition function  from \eqref{eq:ztinydiag} is depicted. Now the initial data is represented by the $\bullet$ at locations $(x-1,x+1)$. Paths can start at any location $(x,x)$ for $x\in \Z_{\geq 0}$ and  their weight is the product of the weights along the path (except for at the endpoint $(t,y)$) as well as at $(x-1,x+1)$.}\label{fig:Zinitialdata}
\end{figure}

We will focus on two types of polymers with initial data (see Figure \ref{fig:Zinitialdata}). In the first case, we will specify initial data at time $0$. Since we have assumed that the parity of the time plus space coordinate is even, we need only have initial data specified at even sites. In particular, assume that $\big(z^{\tinyvert }(x)\big)_{x\in 2\Z_{\geq 0}}$ is a given set of potentially random non-negative real numbers (the vertical line symbolizes that this initial data is specified along a ray in this direction in the time-space plane). Then, for $t\in \Z_{\geq  0}$ define
\begin{equation}\label{eq:ztinyvert}
	z^{\tinyvert}_{\boldsymbol{\X},\boldsymbol{\omega},\beta,z^{\tinyvert}}(t,y) = \sum_{x\in 2\Z_{\geq 0}} z^{\tinyvert}(x) z_{\boldsymbol{\X},\boldsymbol{\omega},\beta}(0,x;t,y).
\end{equation}
This is an infinite sum of non-negative numbers and hence converges (potentially to infinity). We will later specify conditions under which this will be almost surely finite.

In the second case, we will specify initial data  along the time-space diagonal line $s=x$. We will see in Section \ref{sec:proofconvergence} why it is useful to consider such type of initial data, to study the large scale limit of models such as the one defined in Section \ref{sec:loggamma}. In this case, we assume that $\big(z^{\tinydiagup }(x)\big)_{x\in \Z_{\geq 0}}$ is a given set of potentially random non-negative real numbers (the diagonal line symbolizes that this initial data is specified along a ray in this direction in the time-space plane). Then, for $t,y\in \Z_{\geq 0}$ with $t\geq y$ define
\begin{equation}\label{eq:ztinydiag}
	z^{\tinydiagup}_{\boldsymbol{\X},\boldsymbol{\omega},\beta,z^{\tinydiagup}}(t,y) = \sum_{x\in \Z_{\geq 0}} z^{\tinydiagup}(x) z_{\boldsymbol{\X},\boldsymbol{\omega},\beta}(x,x;t,y).
\end{equation}
This is an infinite sum of non-negative numbers and hence converges (potentially to infinity). We will later specify conditions under which this will be almost surely finite.

The modified partition function enjoys various forms of monotonicity with respect to changing the weights. In particular, if we introduce (on the same probability space on which $\boldsymbol{\X}$ and $\boldsymbol{\omega}$ are defined) two other sets of boundary weights $\boldsymbol{\X}^{\downarrow} = (X^{\downarrow}_i)_{i\in \Z}$ and $\boldsymbol{\X}^{\uparrow} = (X^{\uparrow}_i)_{i\in \Z}$ coupled such that ,
\begin{equation}\label{eq:mono}
	\textrm{for all } i\in\Z,\quad \X^{\downarrow}_i \leq \X_i\leq \X^{\uparrow}_i, \quad\textrm{then} \quad z_{\boldsymbol{\X}^{\downarrow},\boldsymbol{\omega},\beta}\leq z_{\boldsymbol{\X},\boldsymbol{\omega},\beta}\leq z_{\boldsymbol{\X}^{\uparrow},\boldsymbol{\omega},\beta}
\end{equation}
as functions in the four parameters.

\begin{theorem}\label{thm:SHEsheetlimit}
	For $\mu\in \R$ and  $n\in \Z_{\geq 1}$ define deterministic boundary weights $\boldsymbol{\X}^{(n),\mu}=\big(\X^{(n),\mu}(r)\big)_{r\in \Z}$ with $\X^{(n),\mu}(r)\equiv 1-n^{-1/2} \mu $ and assume that  we are given bulk weights $\boldsymbol{\omega}^{(n)} = \big(\omega^{(n)}(s,x)\big)_{s\in \Z, x\in \Z_{\geq 0}}$ such that for all $s\in \Z$ and $x\in \Z_{>0}$,
	\begin{equation}\label{eq:thmassumption2}
		\EE\big[\omega^{(n)}(s,x)\big]=0, \;\;\lim_{n\to\infty}\EE\big[\big(\omega^{(n)}(s,x)\big)^2\big]=1, \textrm{ and } \sup_{n\in \Z_{\geq 1}}\EE\big[\big(\omega^{(n)}(s,x)\big)^8\big]<\infty.
	\end{equation}
	Finally, let $\beta^{(n)}=n^{-1/4} \beta/\sqrt{2}$ for $\beta\in \R$.
	
	\begin{enumerate}%
		\item[(1)] {\bf Sheet convergence.} For each $n\in \Z_{\geq 1}$ and all $\mu,\beta,S,T\in \R$ with $S<T$ and $X,Y\in \R_{\geq 0}$ (i.e. $(\mu,\beta, S,X,T,Y)\in \Dfive$) such that $nS,nT\in \Z$, $n^{1/2} X , n^{1/2} Y\in \Z_{\geq 0}$ and $nS+n^{1/2} X$ and $nT+n^{1/2} Y$ are both even, define
		\begin{equation}\label{eq:ZZnmuSXTY}
			\ZZ^{(n)}_{\mu, \beta}(S,X;T,Y) := \frac{n^{1/2}}{2} z_{\boldsymbol{\X}^{(n),\mu},\mathbf{\boldsymbol{\omega}^{(n)},\beta^{(n)}}}(nS,n^{1/2} X;nT,n^{1/2} Y) \cdot 2^{\mathbf{1}_{Y=0}}.
		\end{equation}
		For all other values of $S,X,T$ and $Y$, define $\ZZ^{(n)}_{\mu,\beta}(S,X;T,Y)$ by linear interpolation from the above defined values. Then, $\ZZ^{(n)}_{\cdot,\cdot}(\cdot,\cdot;\cdot,\cdot)$ converges as a process in $\Cfive$ to $\ZZ_{\cdot,\cdot}(\cdot,\cdot;\cdot,\cdot)$, from \eqref{eq:SHEchaosdelta}.
		
		\item[(2)] {\bf Convergence of partition functions with initial data.}
		Fix any $\mu,\beta\in \R$. For each $n\in \Z_{\geq 1}$, given $\big(z^{\tinyvert (n)}(x)\big)_{x\in 2\Z_{\geq 0}}$, define $\ZZ^{(n)}(X) = z^{\tinyvert (n)}(n^{1/2} X)$ for $X\in \R$ with $n^{1/2} X\in 2\Z_{\geq 0}$ and extend to other $X$ by linear interpolation. Assume that there exists a (potentially random) function $\ZZ\in C(\R_{\geq 0},\R)$ such that $\ZZ^{(n)}(\cdot)\Rightarrow \ZZ(\cdot)$ as a process, and that for some $a>0$
		\begin{equation}\label{eq:expgrowthinitial}
			\sup_{n\in \Z_{\geq 1}}\sup_{X\in \R_{\geq 0}} e^{-aX} \EE[\ZZ^{(n)}(X)^2] <\infty.
		\end{equation}
		For all $n\in \Z_{\geq 1}$, and $T,Y\in \R_{\geq 0}$  such that $nT\in \Z$, $n^{1/2} Y\in \Z_{\geq 0}$ and $nT+n^{1/2} Y$ is even, let
		\begin{equation}\label{eq:ZZnvert}
			\ZZ^{(n)}_{\mu,\beta}(T,Y) =z^{\tinyvert}_{\boldsymbol{\X}^{(n),\mu},\boldsymbol{\omega}^{(n)},\beta^{(n)},z^{\tinyvert (n)}}(nT,n^{1/2} Y) \cdot 2^{\mathbf{1}_{Y=0}}
		\end{equation}
		where $z^{\tinyvert}_{\boldsymbol{\X}^{(n)},\boldsymbol{\omega}^{(n)},\beta^{(n)},z^{\tinyvert (n)}}$ is defined in \eqref{eq:ztinyvert}.
		For all other values of $T$ and $Y$, define $\ZZ^{(n)}_{\mu,\beta}(T,Y)$ by linear interpolation. Then, $\ZZ^{(n)}_{\mu,\beta}(\cdot,\cdot)\Rightarrow \ZZ_{\mu,\beta}(\cdot,\cdot)$ in the sense of finite dimensional distributions, where $\ZZ_{\mu,\beta}$ is defined in \eqref{eq:SHEchaos} with initial data $\ZZ_0$ specified by the above assumption. Moreover, for any $k\in\Z_{\geq 1}$ and sequence $\big((T^{(n)}_j,Y^{(n)}_j)\big)_{j\in \{1,\ldots, k\}}$ which converges as $n\to \infty$ to $\big((T_j,Y_j)\big)_{j\in \{1,\ldots, k\}}$, the sequence
		$$\big(\ZZ^{(n)}_{\mu,\beta}(T^{(n)}_j,X^{(n)}_j)\big)_{j\in \{1,\ldots, k\}}\Rightarrow \big(\ZZ_{\mu,\beta}(T_j,X_j)\big)_{j\in \{1,\ldots, k\}}.$$
		
		Similarly, given $\big(z^{\tinydiagup (n)}(x)\big)_{x\in \Z_{\geq 0}}$, define $\ZZ^{(n)}(X) = z^{\tinydiagup (n)}(n^{1/2} X)$ for $X\in \R$ with $Xn^{1/2}\in \Z_{\geq 0}$ and extend to other $X$ by linear interpolation. Under the exact same assumptions as above and with $z^{\tinydiagup}$ from \eqref{eq:ztinydiag} in place of $z^{\tinyvert}$, i.e. setting
		\begin{equation}\label{eq:ZZndiag}
			\ZZ^{(n)}_{\mu,\beta}(T,Y) =z^{\tinydiagup}_{\boldsymbol{\X}^{(n),\mu},\boldsymbol{\omega}^{(n)},\beta^{(n)},z^{\tinydiagup (n)}}(nT,n^{1/2} Y) \cdot 2^{\mathbf{1}_{Y=0}}
		\end{equation}
		we have that $\ZZ^{(n)}_{\mu,\beta}(\cdot,\cdot)\Rightarrow 2 \ZZ_{\mu,\beta}(\cdot,\cdot)$ in the sense of finite dimensional distributions.
		\item[(3)] \textbf{Random boundary weights.}
		Fix any $\mu\in \R$. If we replace the assumption of deterministic boundary weights  $\X^{(n),\mu}(r)\equiv 1-n^{-1/2} \mu$ by i.i.d. random boundary weights satisfying
		\begin{align}
			\begin{split}\label{eq:thmassumption1}
				&\lim_{n\to \infty} \EE\big[n^{1/2}\big(1-\X^{(n),\mu}(r)\big)\big]= \mu \quad \textrm{and}\\
				&\exists \epsilon\in (0,1]\textrm{ and }K>0 \textrm{ such that }\forall n\in \Z_{\geq 1}, \var(\X^{(n),\mu}) \leq K n^{-\epsilon},
			\end{split}
		\end{align}
		then the convergence of the sheet in part (1) and the convergence of the partition functions with initial data in part (2) continues to hold in the sense of finite dimensional distributions (with $(\beta,S,X,T,Y)$ varying in part (1) and $(T,Y)$ varying in part (2)).
	\end{enumerate}
\end{theorem}

With additional work, the finite dimensional distribution convergence in parts (2) and (3) could possibly be improved to process-level convergence. As we do not use this, we have not pursued it.

\subsection{Proof of Theorem \ref{thm:SHEsheetlimit}}
We address the three parts in sequence.

\medskip
\noindent {\bf Part (1).} There are two parts to the sheet convergence. First we prove convergence in the sense of finite-dimensional distributions and then we extend it to process-level convergence by demonstrating tightness. To do this, we first recall two key lemmas, both of which can be found in \cite{alberts2014intermediate} in the full-space setting (though the case of $\R_{\geq 0}$ follows verbatim).

Fix $S<T$ such that  $s=nS, t=nT \in \mathbb Z$. Define disjoint (recall $D_{k}(s,t)$ from \eqref{eq:Dk})
\begin{multline*}
	R^{(n)}_k := \left\{ R^{(n)}(\vec{r},\vec{w})=\big[n^{-1} \vec{r},n^{-1} (\vec{r}+\vec{1})\big)\right. \\ \left. \times \big[n^{-1/2} \vec{w},n^{-1/2} (\vec{w}+\vec{2})\big):\vec{r}\in D_{k}(s,t), \vec{w}\in \Z_{\geq 0}^{k}, \vec{r}\leftrightarrow \vec{w}\right\}
\end{multline*}
where the notation $\big[n^{-1} \vec{r},n^{-1} (\vec{r}+\vec{1})\big) = \big[n^{-1}r_1,n^{-1} (r_1+1)\big)\times \cdots\times\big[n^{-1}r_k,n^{-1} (r_k+1)\big)$, similarly for $\big[n^{-1/2} \vec{w},n^{-1/2} (\vec{w}+\vec{2})\big)$. We recall that $\vec{r}\leftrightarrow \vec{w}$ means that all $r_j+w_j$ are even integers, so that the rectangles are disjoint. We will want to work with functions in
$$
L^2(S,T,k):=L^2\big([S,T]^k\times \R_{\geq 0}^k)
$$
For a function $g\in L^2(\Delta_k(S,T)\times\mathbb R_{\geq 0}^k)$, we denote by $\bar{g}^{(n)}\in L^2(S,T,k)$ the function which is constant on each rectangle $R\in R^{(n)}_k$, taking value there in equal to the average value of $g$ on the rectangle, i.e.
$\bar{g}^{(n)}|_{R} = |R|^{-1} \int_R g$, and taking value $0$ outside of the rectangles. Note that $|R|=2^k n^{-3k/2}$ for each rectangle in $R\in R^{(n)}_k$.
Define the discrete $U$-statistic associated with a function $g$ and random variables $\boldsymbol{\omega}$ by
$$
S^{(n)}_k(g;\boldsymbol{\omega}):= 2^{k/2} \sum_{\vec{r}\in E_k(s,t)} \sum_{\vec{w}\in \Z_{\geq 0}^k} \bar{g}^{(n)}(n^{-1}\vec{r},n^{-1/2} \vec{w}) \omega(\vec{r},\vec{w}) \mathbf{1}_{\vec{r}\leftrightarrow \vec{w}}
$$
where $E_{k}(s,t) := \{\vec{r}=(r_1,\ldots ,r_k)\in \Z^k: r_i\neq r_j \textrm{ for } i\neq j\}$. Notice that the set $E_{k}(s,t)$ and $D_{k}(s,t)$ differ in that $E$ does not restrict the order of the $\vec{r}$ coordinates, while $D$ does.
\begin{lemma}\label{lem:akq41}
	For a given $\boldsymbol{\omega}$, the map $g\mapsto S^{(n)}_k(g,\boldsymbol{\omega})$ is linear.
	For any $g\in L^2(S,T,k)$, provided that $\EE\big[\omega(r,w)\big]=0$ and $\var\big(\omega(r,w)\big)\leq \sigma^2$ for all $r\in \Z$ and $w\in \Z_{\geq 0}$,
	$$
	\EE[\big(S^{(n)}_k(g)\big)^2] \leq \sigma^{2k} n^{3k/2} ||g||^2_{L^2(S,T,k)}.
	$$
\end{lemma}
\begin{proof}
	This follows immediately from the proof of \cite[Lemma 4.1]{alberts2014intermediate}.
\end{proof}

For a sequence of functions $G=(g_0,g_1,\ldots)\in \bigoplus_{k=0}^{\infty} L^2(S,T,k)$, define the chaos series
$$
I(G) := \sum_{k=0}^{\infty} \int\limits_{[S,T]^k} \int\limits_{\R_{\geq 0}^k} g_k(\vec{R},\vec{W})\xi^{\otimes k} (d\vec{R},d\vec{W})
$$
which is convergent provided that $\sum_{k=0}^{\infty} ||g_k||^2_{L^2(S,T,k)}<\infty$. We remark that for functions $g_k$ that vanish outside of $\Delta_k(S,T)\times \mathbb R_{\geq 0}^k$, the series $I(G)$ has the same form as the chaos series expansion \eqref{eq:SHEchaosdelta}. 
\begin{lemma}\label{lem:ustatsconverge}
	Assume that for each $n\in \Z_{\geq 1}$, $\boldsymbol{\omega}^{(n)}= \big(\omega^{(n)}(r,w)\big)_{r\in \Z, w\in \Z_{\geq 0}}$ is a
	collection of random variables such that for $s\in \Z_{\geq 0}$ and $x\in \Z_{>0}$,  $\omega(s,0)\equiv 0$ and  $\omega(s,x)$ are i.i.d. with $\EE\big[\omega^{(n)}(r,w)\big]=0$ and $\lim_{n\to \infty} \sigma^{(n)}=1$ where $(\sigma^{(n)})^2 := \var\big(\omega^{(n)}(r,w)\big)$. Then for any $G=(g_0,g_1,\ldots)\in \bigoplus_{k=0}^{\infty} L^2(S,T,k)$, provided that
	\begin{equation}\label{eq:vanishing}
		\lim_{K\to \infty} \limsup_{n\to \infty} \sum_{k=K}^{\infty} (\sigma^{(n)})^{2k} ||g_k||^2_{L^2(S,T,k)} =0
	\end{equation}
	it follows that
	$
	I^{(n)}(G) := \sum_{k=0}^{\infty} n^{-3k/4} S^{(n)}_k(g_k;\boldsymbol{\omega}^{(n)})
	$
	converges in distribution as $n\to \infty$  to $I(G)$. The same convergence holds jointly for $G_1,\ldots, G_m\in \bigoplus_{k=0}^{\infty} L^2(S,T,k)$, provided all verify \eqref{eq:vanishing}.
\end{lemma}
\begin{proof}
	This is a simple adaptation of \cite[Lemma 4.6]{alberts2014intermediate} (in the spirit of \cite[Lemma 4.4]{wu2018intermediate}).
\end{proof}

Let us apply this to \eqref{eq:zseries}. Define, as a rescaled version of \eqref{eq:zseries2}
\begin{equation}\label{eq:Pcdef}
	\PC^{(n)}_{\mu}(S,X;\vec{R},\vec{W};T,Y):=   \left(\tfrac{n^{1/2}}{2}\right)^{k+1} p_{\boldsymbol{\X}^{(n);\mu}}(nS,n^{1/2}X;n\vec{R},n^{1/2}\vec{W};nT,n^{1/2}Y)\mathbf{1}_{\vec{R}\in \Delta^{(n)}_{k}(S,T)}
\end{equation}
where $\vec{R}\in \Delta^{(n)}_k(S,T)$ if and only if $n\vec{R}\in D_{k}(nS,nT)$ and where the right-hand side is given by the function $p_{\boldsymbol{\X}^{(n)}}(s,x;\vec{r},\vec{w};t,y)$ for $\vec{r}\in D_{k}(s,t)$ and $\vec{w}\in \Z_{\geq 0}^k$ with $\vec{r}\leftrightarrow \vec{w}$, and extended from those points to be piece-wise constant on each rectangle $R^{(n)}(\vec{r},\vec{w})\in R^{(n)}_k$. We will denote the function $(\vec{R},\vec{W})\mapsto \PC^{(n)}_{\mu}(S,X;\vec{R},\vec{W};T,Y)$ by $\PC^{(n)}_{\mu}(S,X;\cdot,\cdot;T,Y)$ and extend it from $\Delta_k^{(n)}(S,T)\times \mathbb R_{\geq 0}^k$ to  $E_k^{(n)}(S,T)\times \mathbb R_{\geq 0}^k$ (where $\vec{R}\in E^{(n)}_k(S,T)$ if and only if $n\vec{R} \in E_k(nS,nT)$), so that it takes value $0$ outside $\Delta_k^{(n)}(S,T)\times \mathbb R_{\geq 0}^k$.  With this definition observe that since
\begin{multline}\label{eq:SSnexpansion}
	S^{(n)}_k(\PC^{(n)}_{\mu}(S,X;\cdot,\cdot;T,Y);\boldsymbol{\omega}^{(n)}) \\ = 2^{k/2} \sum_{\vec{R}\in E^{(n)}_k(S,T)}\sum_{\vec{W}\in n^{-1/2}\Z_{\geq 0}^k} \PC^{(n)}_{\mu}(S,X;\vec{R},\vec{W};T,Y)\omega^{(n)}(\vec{R},\vec{W})\mathbf{1}_{\vec{R}\leftrightarrow\vec{W}},
\end{multline}
we can rewrite 
\begin{equation}\label{eq:ZZndef}
	\ZZ^{(n)}_{\mu,\beta}(S,X;T,Y) = \sum_{k=0}^{n(T-S)}n^{-3k/4} S^{(n)}_k( \beta^k \PC^{(n)}_{\mu}(S,X;\cdot,\cdot;T,Y)\cdot 2^{\mathbf{1}_{Y=0}};\boldsymbol{\omega}^{(n)})  .
\end{equation}
Now we will show  that as $n\to\infty$, the right-hand side converges to the chaos series for $\ZZ_{\mu,\beta}(S,X;T,Y)$.

For this, we need a lemma that compares $\PC^{(n)}_{\mu}(S,X;\cdot,\cdot;T,Y)$ and $\PC_{\mu}(S,X;\cdot,\cdot;T,Y)$ as well as provides decay bounds on $\PC^{(n)}_{\mu}(S,X;\cdot,\cdot;T,Y)$ uniformly in $n$.
\begin{lemma}\label{lem:xuan62}
	For each $\mu\in \R$ and interval $I\subset \R$, there exists a constant $C>0$ such that for all $k,n\in \Z_{\geq 1}$, $S,T\in I$ with $S<T$, and $X,Y\in \R_{\geq 0}$
	\begin{align}
		&\PC^{(n)}_{\mu}(S,X;T,Y) \leq C|T-S|^{-1/2} e^{-\frac{|X-Y|^2}{C |T-S|}}\label{eq:xuan0}\\
		&||\PC^{(n)}_{\mu}(S,X;\cdot,\cdot;T,Y)\cdot 2^{\mathbf{1}_{Y=0}}||_{L^2(S,T,k)}^2 \leq |T-S|^{\frac{k}{2}-1} e^{-\frac{|X-Y|^2}{C \max(|T-S|,2k/n)}} \frac{C^k}{\Gamma((k+1)/2)}\label{eq:xuan1}\\
		&
		\lim_{n\to \infty} ||\PC^{(n)}_{\mu}(S,X;\cdot,\cdot;T,Y)\cdot 2^{\mathbf{1}_{Y=0}}-\PC_{\mu}(S,X;\cdot,\cdot;T,Y)||_{L^2(S,T,k)}=0\label{eq:xuan2}
	\end{align}
\end{lemma}
\begin{proof}
	This result appears as \cite[Lemma 6.2]{wu2018intermediate}. There is a small mistake there. When $Y=0$, the factor of $2$ that comes from the term $2^{\mathbf{1}_{Y=0}}$ is forgotten. Its need is clear from the proof (the factor of $2$ arises in \cite[Lemma A.1]{wu2018intermediate} when a random walk ends at the origin).
\end{proof}

Using \eqref{eq:xuan1} and \eqref{eq:xuan2} from Lemma \ref{lem:xuan62} and the triangle inequality,  we  also obtain that (of course, this can be shown directly too) 
\begin{equation} ||\PC_{\mu}(S,X;\cdot,\cdot;T,Y)||_{L^2(S,T,k)}^2 \leq |T-S|^{\frac{k}{2}-1} e^{-\frac{|X-Y|^2}{C \max(|T-S|,2k/n)}} \frac{C^k}{\Gamma((k+1)/2)}.
\end{equation}
Thus we can apply Lemma \ref{lem:ustatsconverge} (whose condition \eqref{eq:vanishing} is verified using the result of Lemma \ref{lem:akq41}) to show that
\begin{equation}\label{eq:convergedist}
	\widetilde{\ZZ}^{(n)}_{\mu,\beta}(S,X;T,Y):=\sum_{k=0}^{\infty}n^{-3k/4} S^{(n)}_k( \beta^k \PC_{\mu}(S,X;\cdot,\cdot;T,Y);\boldsymbol{\omega}^{(n)}) \Rightarrow \ZZ_{\mu,\beta}(S,X;T,Y)
\end{equation}
in distribution where we note that $\widetilde{\ZZ}^{(n)}_{\mu,\beta}$ is defined in terms of $\PC_{\mu}$ instead of $ \PC_{\mu}^{(n)}$.

The final step to finite-dimensional convergence is to show that  $\ZZ^{(n)}_{\mu,\beta}(S,X;T,Y) -\widetilde{\ZZ}^{(n)}_{\mu,\beta}(S,X;T,Y)$ goes to zero in distribution. We show this convergence in $L^2$. Write
$$
\ZZ^{(n)}_{\mu,\beta}(S,X;T,Y) -\widetilde{\ZZ}^{(n)}_{\mu,\beta}(S,X;T,Y) = I^{(n)}+ II^{(n)}
$$
where, for $\delta\PC^{(n)}_{\mu}(S,X;\cdot,\cdot;T,Y):=\PC^{(n)}_{\mu}(S,X;\cdot,\cdot;T,Y)\cdot 2^{\mathbf{1}_{Y=0}}-\PC_{\mu}(S,X;\cdot,\cdot;T,Y)$,
\begin{align*}
	I^{(n)} &= \sum_{k=0}^{n(T-S)}n^{-3k/4} S^{(n)}_k( \beta^k \cdot \delta\PC^{(n)}_{\mu}(S,X;\cdot,\cdot;T,Y);\boldsymbol{\omega}^{(n)}),\\
	II^{(n)} &= -\sum_{k=n(T-S)+1}^{\infty} n^{-3k/4} S^{(n)}_k( \beta^k\cdot \PC_{\mu}(S,X;\cdot,\cdot;T,Y);\boldsymbol{\omega}^{(n)}).
\end{align*}
$\EE[(II^{(n)})^2]\to 0$ follows from the triangle inequality, followed by Lemma \ref{lem:akq41} and the bound \eqref{eq:xuan1} for $\PC_{\mu}(S,X;\cdot,\cdot;T,Y)$. To address $I^{(n)}$, by the triangle inequality and Lemma \ref{lem:akq41},
$$
\EE[(I^{(n)})^2]^{1/2}\leq \sum_{k=0}^{n(T-S)}||(\sigma\beta)^k\cdot \delta\PC^{(n)}_{\mu}(S,X;\cdot,\cdot;T,Y)||_{L^2(S,T,k)}
$$
provided $\sigma^2 = \sup_{n\in \Z_{\geq 1}}\var\big(\omega(r,w)\big)$. By \eqref{eq:xuan2}, each term in this sum goes to zero, and by \eqref{eq:xuan1} and the analogous bound for $\PC_{\mu}(S,X;\cdot,\cdot;T,Y)$, there is a single term-wise dominating sequence that is valid for all $n\in \Z_{\geq 1}$ and which is summable. Thus, the dominated convergence theorem implies that the sum converges to zero, hence $\EE[(I^{(n)})^2]\to 0$ as wished.

This proves the  convergence of $\ZZ^{(n)}_{\mu,\beta}(S,X;T,Y)$ to $\ZZ_{\mu,\beta}(S,X;T,Y)$ for some \break  $(\mu,\beta,S,X;T,Y)\in\Dfive$. Finite distributional convergence follows from Lemma \ref{lem:ustatsconverge} and the argument above.

Now we turn to showing the tightness. As is often the case, to prove the tightness it suffices from the Kolmogorov continuity criterion, Prokhorov's theorem and the Arzela-Ascoli theorem that we demonstrate suitable moment bounds on $\ZZ^{(n)}_{\mu,\beta}(S,X;T,Y)$ and $\ZZ^{(n)}_{\mu,\beta}(S,X;T,Y)-\ZZ^{(n)}_{\mu',\beta'}(S',X';T',Y')$ as $(\mu,\beta,S,X,T,Y)$ and $(\mu',\beta',S',X',T',Y')$ vary in compact subsets of $\Dfive$. Of course, by the triangle inequality it suffices to consider varying only one parameter at a time. When the time or space variables are varied, the type of moment bounds that we need have previously been shown in related contexts in \cite[Appendix A]{alberts2014intermediate}, \cite[Propostion 5.9]{parekh2019positive} or the end of Section 6 in \cite{wu2018intermediate} by appealing to the discrete version of the mild form of the SHE in \eqref{eq:zdiscmild}. In proving those bounds, the final assumption in \eqref{eq:thmassumption2} of having a bulk environment with finite 8$^{th}$ moment is used.  The argument for varying $\beta$ works similarly. We will not reproduce such calculations here and rather address the previously unaddressed case of varying the $\mu$ parameter. The following lemma provides a weak though sufficient moment bound. 
\begin{lemma}
	For any compact interval $I\in \R$, there exists a constant $C$ such that for all $\mu,\nu\in I$, 
	\begin{equation}\label{eq:ZZnsquare}
		\EE\Big[\big(\ZZ^{(n)}_{\mu,\beta}(S,X;T,Y)-\ZZ^{(n)}_{\nu,\beta}(S,X;T,Y)\big)^2\Big]\leq C|\mu-\nu|^2.
	\end{equation}
	\label{lem:mumomentbound}
\end{lemma}
The proof of Lemma \ref{lem:mumomentbound} is given below. Let us see first how we can use this to bound higher moments. For $r>p>2$ there exists a unique $\theta\in (0,1)$ so that $1/p = (1-\theta)/2 + \theta/r$. By the H\"older inequality,
\begin{multline}
	\EE\Big[\left\vert\ZZ^{(n)}_{\mu,\beta}(S,X;T,Y)-\ZZ^{(n)}_{\nu,\beta}(S,X;T,Y)\right\vert^p\Big]  \\ \leq
	\EE\Big[\big(\ZZ^{(n)}_{\mu,\beta}(S,X;T,Y)-\ZZ^{(n)}_{\nu,\beta}(S,X;T,Y)\big)^2\Big]^{(1-\theta)p/2} \\ \times \EE\Big[\left\vert\ZZ^{(n)}_{\mu,\beta}(S,X;T,Y)-\ZZ^{(n)}_{\nu,\beta}(S,X;T,Y)\right\vert^r\Big]^{\theta p/r}.
	\label{eq:holderboundmu}
\end{multline}
The second term on the right-hand side of \eqref{eq:holderboundmu} can be bounded by a constant (which depends on $p,r$). Indeed, we bound the $L^r$ norm by the $L^2$ norm, write the difference $\ZZ^{(n)}_{\mu,\beta}(S,X;T,Y)-\ZZ^{(n)}_{\nu,\beta}(S,X;T,Y)$ using the expansions \eqref{eq:SSnexpansion} and \eqref{eq:ZZndef}. Then we may estimate the $L^2$ norm using the triangle inequality and the bound \eqref{eq:xuan1}, which guarantees that the sum over $k$, $\vec R$ and $\vec W$ is finite.

The first term in \eqref{eq:holderboundmu} can be bounded via \eqref{eq:ZZnsquare} by $C|\mu-\nu|^{(1-\theta)p}$. When $r$ tends to infinity, $\theta$ drops to a limit $1-2/p$ and the exponent $(1-\theta)p$ approach $2$ from below. On this account we see that for any $p>2$ and any $\epsilon\in(0,1)$ there is a constant $C$ so that for all $\mu,\nu\in I$
\begin{equation}\label{eq:ZZmomentp}
	\EE\Big[\big(\ZZ^{(n)}_{\mu,\beta}(S,X;T,Y)-\ZZ^{(n)}_{\nu,\beta}(S,X;T,Y)\big)^p\Big] \leq C|\mu-\nu|^{2-\epsilon}.
\end{equation}
This moment bound, combined with similar moment bounds for varying other coordinates in $\ZZ^{(n)}_{\mu,\beta}(S,X;T,Y)$ is sufficient to prove tightness. This completes the proof of Part (1), modulo Lemma \ref{lem:mumomentbound}. 

\begin{proof}[proof of lemma \ref{lem:mumomentbound}] Recall the definition of  $\ZZ^{(n)}_{\mu,\beta}(S,X;T,Y)$ in  \eqref{eq:ZZnmuSXTY}   in terms of the partition function defined in \eqref{eq:zprod}. We need to bound, for $s=nS, t=nT, x=n^{1/2}X$ and $y=n^{1/2}Y$, uniformly in $n$, the quantity 
	\begin{equation}
		\label{eq:expectationto bound}
		\mathbb E\left[   \left( \frac{n^{1/2}}{2} \mathbb E_R^{s,x}\!\left[\! \left(\!\left(\!1-\mu n^{-1/2}\right)^{\ell_0(\mathcal S)} \!\!\!\!\!- \left(\!1-\nu n^{-1/2}\right)^{\ell_0(\mathcal S)} \! \right)\prod_{r=s}^{t-1}(1+\beta w(r, \mathcal S_r)) \;  \mathds{1}_{\mathcal S_t=y}  \right]    \right)^2  \right],
	\end{equation}
	where $\mathbb E_R^{s,x}$ denotes the expectation with respect to a reflected random walk $\mathcal S$ starting at $\mathcal S_s=x$, and $\ell_0(\mathcal S)$ is the local time at $0$ (between times $s$ and $t$). We need to show that \eqref{eq:expectationto bound} is bounded by $C(\mu-\nu)^2$.  First, we may use that for any function $f$ of the walk $\mathcal S$, 
	\begin{multline*}n^{1/2} \mathbb E_R^{s,x} \left[ f(\mathcal S) \mathds{1}_{\mathcal S_t=y}\right] =\\   \frac{n^{1/2}}{2}\mathbb P_R^{s,x} \left( \mathcal S_t=y\right)\mathbb E_R^{s,x} \left[  f(\mathcal S) \Big\vert \mathcal S_t=y\right] = \mathcal P^{(n)}_{0}(S,X; T,Y)\mathbb E_R^{s,x} \left[  f(\mathcal S) \Big\vert \mathcal S_t=y\right].
	\end{multline*}
	Hence, using Jensen's inequality, \eqref{eq:expectationto bound} is bounded by the product of $\mathcal P^{(n)}_{0}(S,X; T,Y)^2$ (which may be bounded by a constant using \eqref{eq:xuan0}) times 
	\begin{equation}
		\label{eq:newexpectationto bound}
		\mathbb E\left[ \mathbb E_R^{s,x} \left[ \left(\left(1-\mu n^{-1/2}\right)^{\ell_0(\mathcal S)} - \left(1-\nu n^{-1/2}\right)^{\ell_0(\mathcal S)}  \prod_{r=s}^{t-1}(1+\beta w(r, \mathcal S_r))\right)^2\bigg\vert \mathcal S_t=y \right]\right].
	\end{equation}
	Using that $\mathbb E [w(r, \mathcal S_r)]=0$ and $\mathbb E [w(r, \mathcal S_r)^2]=1$, this expectation is exactly 
	\begin{equation}
		\label{eq:newnewexpectationto bound}
		(1+\beta^2)^{n(T-S)}  \mathbb E_R^{s,x} \left[ \left(\left(1-\mu n^{-1/2}\right)^{\ell_0(\mathcal S)} - \left(1-\nu n^{-1/2}\right)^{\ell_0(\mathcal S)}  \right)^2  \Big\vert \mathcal S_t=y \right].
	\end{equation}
	The prefactor $(1+\beta^2)^{n(T-S)}$ is easily bounded uniformly in $n$, for $S,T$ varying in a compact set. Further, we know that the distribution of  $n^{-1/2}\ell_0(\mathcal S)$ converges as $n$ goes to infinity to a Brownian local time. Let $p_n(\ell)$ denote the distribution of the random variable $n^{-1/2}\ell_0(\mathcal S)$ under the measure $\mathbb P_R^{s,x} ( \;\cdot\; \vert \mathcal S_t=x)$. The expectation in \eqref{eq:newnewexpectationto bound} can be rewritten as 
	\begin{equation} 
		\label{eq:in terms of local time}\int_0^{+\infty}  \left( \left(1-\mu n^{-1/2}\right)^{n^{1/2}\ell} -  \left(1-\nu n^{-1/2}\right)^{n^{1/2}\ell} \right)^2 dp_n(\ell) .
	\end{equation}
	Using Cauchy's mean value theorem, 
	\begin{equation} 
		\label{eq:localtimeexponentialmoment}
		\frac{	\eqref{eq:in terms of local time}}{(\mu-\nu)^2}\leq \int_0^{+\infty}  \left( \ell e^{-\ell \min\lbrace \mu, \nu\rbrace} \right)^2dp_n(\ell).
	\end{equation}
	It can be shown (see e.g. \cite[Lemma A.4]{wu2018intermediate}) that the distribution of rescaled random walk local times has Gaussian decay, uniformly in $n$, i.e. $dp_n(\ell)\leq e^{-c\ell^2}d\ell$. This implies that \eqref{eq:localtimeexponentialmoment} is bounded by a constant, so that the proof is complete. 
\end{proof}

\medskip

\noindent{\bf Part (2).} We turn now to show convergence of partition functions with initial data. Let us focus first on the case where $\ZZ^{(n)}(X) = z^{\tinyvert (n)}(Xn^{1/2})$ and $\ZZ^{(n)}_{\mu,\beta}(T,Y)$ is given by \eqref{eq:ZZnvert}. Based on \eqref{eq:ztinyvert} and \eqref{eq:ZZnmuSXTY} we can rewrite $\eqref{eq:ZZnvert}$ as
\begin{equation}\label{eq:Znnn}
	\ZZ^{(n)}_{\mu,\beta}(T,Y) = 2n^{-1/2}\!\!\!\!\!\!\!\! \sum_{X\in 2n^{-1/2} \Z_{\geq 0}}\!\!\!\!\!\!  \ZZ^{(n)}(X) \ZZ^{(n)}_{\mu,\beta}(0,X;T,Y) = \varnint dX \ZZ^{(n)}(X) \ZZ^{(n)}_{\mu,\beta}(0,X;T,Y)
\end{equation}
where the integral symbol $\varnint dX$ is shorthand for $2n^{-1/2} \sum_{X\in 2n^{-1/2} \Z_{\geq 0}}$. Our target is to compare $\ZZ^{(n)}_{\mu,\beta}(T,Y)$ to $\ZZ_{\mu, \beta}(T,Y)$ which itself admits a similar representation as
$$
\ZZ_{\mu,\beta}(T,Y)=\int dX \ZZ(X) \ZZ_{\mu,\beta}(0,X;T,Y)
$$
where the integral is always taken over $\R_{\geq 0}$. We achieve this by using the Skorohod representation theorem to show uniform convergence of the integrand in \eqref{eq:Znnn} to that above, along with dominating control over the tails of the integral.

The Skorohod representation theorem embeds into a common probability space the sequences $(\ZZ^{(n)}(\cdot))_{n\in \Z_{\geq 1}}$ and
$(\ZZ^{(n)}_{\cdot,\cdot}(\cdot,\cdot;\cdot,\cdot)_{n\in \Z_{\geq 1}}$ along with their limits $\ZZ(\cdot)$ and $\ZZ_{\cdot,\cdot}(\cdot,\cdot;\cdot,\cdot)$ in such a way that the convergence is upgraded to almost sure convergence with respect to the metric \eqref{eq:metric} which metrizes the uniform convergence on compact subsets topology. In other words, under this coupling we have that  $d(\ZZ^{(n)}(\cdot),\ZZ(\cdot))$ and $d(\ZZ^{(n)}_{\cdot,\cdot}(\cdot,\cdot;\cdot,\cdot),\ZZ_{\cdot,\cdot}(\cdot,\cdot;\cdot,\cdot))$ both converge to zero almost surely.

We will show that for any fixed $T,Y$, we have the convergence in probability
\begin{equation}
	\lim_{n\to \infty} |\ZZ^{(n)}_{\mu,\beta}(T,Y)-\ZZ_{\mu,\beta}(T,Y)|=0.\label{eq:limittoshow}
\end{equation}
This implies convergence in probability of finite dimensional distributions (we do not pursue tightness), hence the weak convergence $$\big(\ZZ^{(n)}_{\mu,\beta}(T^{(n)}_j,X^{(n)}_j)\big)_{j\in \{1,\ldots, k\}}\Rightarrow \big(\ZZ_{\mu,\beta}(T_j,X_j)\big)_{j\in \{1,\ldots, k\}}$$ 
in the statement of the theorem.

To show \eqref{eq:limittoshow}, we use the triangle inequality to write
$$
|\ZZ^{(n)}_{\mu,\beta}(T,Y)-\ZZ_{\mu,\beta}(T,Y)|\leq I^{(n)}+II^{(n)}+III^{(n)}
$$
where we have
\begin{align*}
	I^{(n)}&=   \varnint dX |\ZZ^{(n)}(X)-\ZZ(X)| \cdot |\ZZ^{(n)}_{\mu,\beta}(0,X;T,Y)|,\\
	II^{(n)}&= \varnint dX |\ZZ(X)| \cdot |\ZZ^{(n)}_{\mu,\beta}(0,X;T,Y)-\ZZ_{\mu,\beta}(0,X;T,Y)|,\\
	III^{(n)}&= \left|\varnint dX \ZZ(X)\ZZ_{\mu,\beta}(0,X;T,Y)-\int dX \ZZ(X)\ZZ_{\mu,\beta}(0,X;T,Y)\right|.
\end{align*}
It suffices to show that $\lim_{n\to \infty} I^{(n)}=0$ in probability and similar statements for $II$ and $III$. Let us start with  $I^{(n)}$.
For $M\in \R_{>0}$ write $I^{(n)}_{\leq M}$ and $I^{(n)}_{> M}$ for the above integral restricted to $X\leq M$ or $X>M$.
For any $M$,
$$
I^{(n)}_{\leq M} \leq   M\sup_{X\leq M}|\ZZ^{(n)}(X)-\ZZ(X)| \sup_{X\leq M} |\ZZ^{(n)}_{\mu,\beta}(0,X;T,Y)|.
$$
By the almost sure sup-norm convergence on compacts (by the Skorohod representation theorem) it follows that almost surely $\lim_{n\to \infty} \sup_{X\leq M}|\ZZ^{(n)}(X)-\ZZ(X)|=0$ and likewise \break $\lim_{n\to \infty} \sup_{X\leq M} |\ZZ^{(n)}_{\mu,\beta}(0,X;T,Y)|<\infty$. Thus,  almost surely $\lim_{n\to \infty}I^{(n)}_{\leq M}=0$. This holds on the same full-measure set for all $M'<M$ and hence the event that $\lim_{n\to \infty} I^{(n)}_{\leq M}=0$ for all $M\in \R_{>0}$ holds almost surely, and in particular this limit holds in probability. This outcome is not the same as what we wish to show, that $\lim_{n\to \infty}I^{(n)}=0$ in probability. This is because we have not yet controlled the potential contribution from the tails $I^{(n)}_{> M}$. We do this now.

Fix $\epsilon>0$. We first claim that for any $\eta>0$, we may find some $M$ such that $\mathbb P(I^{(n)}_{>M}>\epsilon)\leq \eta/2$, uniformly over $n$. This follows from Markov inequality and the fact that $\EE[I^{(n)}_{>M}]$ can be made arbitrarily small for $M$ large enough, uniformly in $n$. Indeed, by Cauchy-Schwarz
$$
\EE[I^{(n)}_{>M}] \leq \varnint dX \mathbf{1}_{X>M}  \EE[|\ZZ^{(n)}(X)-\ZZ(X)|^2]^{1/2} \cdot \EE[|\ZZ^{(n)}_{\mu,\beta}(0,X;T,Y)|^2]^{1/2}.
$$
The term  $\EE[|\ZZ^{(n)}(X)-\ZZ(X)|^2]^{1/2}$ is controlled by the assumption \eqref{eq:expgrowthinitial} and \break $\EE[|\ZZ^{(n)}_{\mu,\beta}(0,X;T,Y)|^2]^{1/2}$ via the following string of inequalities:
\begin{multline}\label{eq:string}
	\EE[|\ZZ^{(n)}_{\mu,\beta}(0,X;T,Y)|^2]^{1/2} \leq \sum_{k=0}^{nT} (\sigma \beta)^k ||\PC^{(n)}_{\mu}(0,X;\cdot,\cdot;T,Y)\cdot 2^{\mathbf{1}_{Y=0}}||_{L^2(0,T,k)}\\
	\leq \frac{1}{T}\sum_{k=0}^{nT} \frac{(T^{1/2}C\sigma \beta)^k}{\Gamma((k+1)/2)}  e^{-\frac{(X-Y)^2}{C \max(T,2k/n)}}
	\leq \frac{1}{T}\sum_{k=0}^{\infty} \frac{(T^{1/2}C\sigma \beta)^k}{\Gamma((k+1)/2)}  e^{-\frac{(X-Y)^2}{2CT}}\leq C' e^{-\frac{(X-Y)^2}{2CT}}.
\end{multline}
The first inequality follows by combining the $L^2$ triangle inequality with Lemma \ref{lem:akq41}, the second inequality from the bound \eqref{eq:xuan1}, the third  inequality by noting that for $k\in [0,nT]$, $\max(T,2k/n)\leq 2T$ and then extending the summation to infinity, and the final inequality by bounding the summation by a constant $C'$ that depends on $T$.
Thus, we have
\begin{equation*}
	\EE[I^{(n)}_{>M}] \leq \varnint dX \mathbf{1}_{X>M}  \EE[|\ZZ^{(n)}(X)-\ZZ(X)|^2]^{1/2} \cdot C' e^{-\frac{(X-Y)^2}{2CT}} \\ \leq  \varnint dX \mathbf{1}_{X>M} C'' e^{aX/2}e^{-\frac{(X-Y)^2}{2CT}}
\end{equation*}
where we have used assumptions \eqref{eq:expgrowthinitial} and \eqref{eq:initdatabdd} for the final inequality (and $a$ is chosen to satisfy that assumption). Owing to the Gaussian decay, it is easy to see that for $M$ large enough, the final (discrete) integral can be made arbitrarily small, uniformly in $n$. 

Thus, for any $\eta>0$, we have found some $M$ such that $\mathbb P(I^{(n)}_{>M}>\epsilon)\leq \eta/2$, uniformly over $n$. Then, using the almost sure convergence of $I^{(n)}_{\leq M}$, we may find some $n_0$ such that for $n\geq n_0$, $\mathbb P(I^{(n)}_{\leq M}>\epsilon)\leq \eta/2 $. Hence, for $n\geq n_0$, we have that $\mathbb P(I^{(n)}>\epsilon)\leq \eta$. Since this holds for any $\eta$, it  shows that for all $\epsilon>0$, $\mathbb P(I^{(n)}>\epsilon)$ goes to zero as $n$ goes to infinity, i.e. $\lim_{n\to\infty} I^{(n)}=0$ in probability.

One can similarly show that $\lim_{n\to\infty} II^{(n)}=0$ in probability. Indeed, $\lim_{n\to\infty} II^{(n)}_{<M}=0$ almost surely using the fact that  $\lim_{n\to \infty} \sup_{X\leq M}|\ZZ^{(n)}_{\mu,\beta}(0,X;T,Y)-\ZZ_{\mu,\beta}(0,X;T,Y)|=0$ and  $\lim_{n\to \infty} \sup_{X\leq M} |\ZZ(X)|<\infty$. We may control  $II^{(n)}_{>M}$ uniformly in $n$ through  Markov inequality and the same argument as above,  using the the bound \eqref{eq:xuan2} and the assumption \eqref{eq:initdatabdd}. 

Finally, we have that $\lim_{n\to \infty} III^{(n)}=0$ almost surely, and in particular in probability,  since Riemann sums of continuous integrable functions converge to their integral. 

%
%
%
%
%
%
%


This completes the proof of part (2) in the case of $\ZZ^{(n)}(X) = z^{\tinyvert (n)}(Xn^{1/2})$ and $\ZZ^{(n)}(T,Y)$ given by \eqref{eq:ZZnvert}.

Finally, let us briefly address the final paragraph in part (2). Now we are given $\big(z^{\tinydiagup }(x)\big)_{x\in \Z_{\geq 0}}$ and define $\ZZ^{(n)}(X) = z^{\tinydiagup }(n^{1/2} X)$ for $X\in \R$ with $Xn^{1/2}\in \Z_{\geq 0}$ and extend to other $X$ by linear interpolation. We claim that $\ZZ^{(n)}_{\mu,\beta}(\cdot,\cdot)\Rightarrow 2 \ZZ_{\mu,\beta}(\cdot,\cdot)$ (recall \eqref{eq:ZZndiag}). To see this, we rewrite
\begin{align}\label{eq:Znnndiag}
	\ZZ^{(n)}_{\mu,\beta}(T,Y) &= 2n^{-1/2} \sum_{X\in n^{-1/2} \Z_{\geq 0}}\!\!\!\!\!\!  \ZZ^{(n)}(X) \ZZ^{(n)}_{\mu,\beta}(n^{-1/2} X,X;T,Y)  \\
	\nonumber &= 2 \varnint dX \ZZ^{(n)}(X) \ZZ^{(n)}_{\mu,\beta}(n^{-1/2} X,X;T,Y)
\end{align}
where, since the summations are over $ X\in n^{-1/2} \Z_{\geq 0}$ now (see Fig. \ref{fig:Zinitialdata} right),  the integral symbol $\varnint dX$ is redefined to be a  shorthand for $n^{-1/2}  \sum_{X\in n^{-1/2} \Z_{\geq 0}}$. The argument to match this to
$2\ZZ_{\mu,\beta}(T,Y)= 2\int dX \ZZ(X) \ZZ_{\mu,\beta}(0,X;T,Y)$ now proceeds quite similarly to the proof above, so we do not repeat it. The presence of the $n^{-1/2}X$ factor in place of $0$ in the first slot of $ \ZZ^{(n)}_{\mu,\beta}$ has no effect in the limit because the convergence of $\ZZ^{(n)}_{\cdot,\cdot}(\cdot,\cdot;\cdot,\cdot)$ to $\ZZ_{\cdot,\cdot}(\cdot,\cdot;\cdot,\cdot)$ is in $\Cfive$. The completes the proof of part (2) of the theorem.

\medskip

\noindent{\bf Part (3).}
Let us start with the sheet convergence. The idea is to compare the partition function with random boundaries to that with deterministic boundaries. Let us introduce some useful shorthand notation. Since the $\left( \X^{(n), \mu}(r)\right)_{r\in \mathbb Z}$ are now i.i.d., we introduce $\murand(r)$, defined by
$$
\X^{(n), \mu}(r) = 1 - n^{-1/2} \murand(r)
$$
and write $\mu^{(n)}= \EE[\murand(r)]$ where $\mu^{(n)}\to \mu$. We let $\ZZ^{(n)}:=\ZZ^{(n)}_{\mu^{(n)},\beta}(S,X;T,Y)$ denote the partition function defined in \eqref{eq:ZZnmuSXTY} with deterministic boundary with parameters $\mu^{(n)}$, and let $\hat{\ZZ}^{(n)}:=\ZZ^{(n)}_{\murand,\beta}(S,X;T,Y)$  denote the partition function with i.i.d. boundary with parameters $\murand(r)$. Similarly, let $\PC^{(n)}(\cdot;\cdot):=\PC^{(n)}_{\mu^{(n)}}(S,X;\cdot,\cdot;T,Y)2^{\mathbf{1}_{Y=0}}$ denote the transition kernel with deterministic boundary with parameter $\mu^{(n)}$ defined in \eqref{eq:Pcdef}, and let $\hat{\PC}^{(n)}(\cdot;\cdot):=\PC^{(n)}_{\murand}(S,X;\cdot,\cdot;T,Y)\cdot 2^{\mathbf{1}_{Y=0}}$ denote the transition kernel with random boundary with parameters $\murand(r)$.

We will now prove that 
\begin{equation}\label{eq:varzero}
	\var(\hat{\ZZ}^{(n)}-\ZZ^{(n)}) \to 0 \quad \textrm{as } n\to\infty.
\end{equation}
Since this holds for any choices of $(\beta,S,X,T,Y)$, this implies the claimed finite dimensional convergence of $\ZZ^{(n)}_{\murand,\beta}(S,X;T,Y)$ to the same limit as that of $\ZZ^{(n)}_{\mu^{(n)},\beta}(S,X;T,Y)$, namely to $\ZZ_{\mu,\beta}(S,X;T,Y)$.

To prove \eqref{eq:varzero} first note that $\EE[\hat{\ZZ}^{(n)}]=\EE[\ZZ^{(n)}]$ (as the bulk weights are mean zero and $\EE[\murand] =\mu^{(n)}$). Thus $\var(\hat{\ZZ}^{(n)}-\ZZ^{(n)}) = \EE\big[\big(\hat{\ZZ}^{(n)}-\ZZ^{(n)}\big)^2\big]$. As the boundary and bulk weights are independent, we can write the overall expectation $\EE[\cdot] = \EE^{\rm bdry}[\EE^{\rm bulk}[\cdot]]$ where $\EE^{\rm bdry}$ is the expectation over the boundary weights and $\EE^{\rm bulk}$ over the bulk weights. Recalling \eqref{eq:ZZndef}, we see that
\begin{align}
	\nonumber \EE\!\!\left[\!\big(\hat{\ZZ}^{(n)}-\ZZ^{(n)}\big)^2\right]^{1/2} \!\!\!\! &\leq \!\!\! \sum_{k=0}^{n(T-S)} \!\! \!\!\!\EE^{\rm bdry}\!\!\left[\EE^{\rm bulk}\!\left[\left(n^{-3k/4} S^{(n)}_k( \beta^k (\hat\PC^{(n)}-\PC^{(n)})(\cdot,\cdot);\boldsymbol{\omega}^{(n)})\right)^2\right]\right]^{1/2}\\
	&\leq \sum_{k=0}^{n(T-S)}(\sigma \beta)^{k} \left( \EE^{\rm bdry}\left[ ||(\hat\PC^{(n)}-\PC^{(n)})(\cdot,\cdot)||^2_{L^2(S,T,k)}\right]\right)^{1/2}\label{eq:ZZhat2}\\
	\nonumber& =\sum_{k=0}^{n(T-S)}(\sigma \beta)^{k} \left( \int_{[S,T]^k}\!\!\!d\vec{R}\int_{\R_{\geq 0}^k}\!\!\!d\vec{W}\, \var^{\rm bdry}\left(\hat\PC^{(n)}(\vec{R},\vec{W})\right)\right)^{1/2}
\end{align}
The first inequality above uses the triangle inequality and the definition of $\hat\ZZ^{(n)}$ and $\ZZ^{(n)}$; the second inequality uses Lemma \ref{lem:akq41} to bound the $\EE^{\rm bulk}$ expectation; the final inequality writes the $L^2(S,T,k)$ norm explicitly as integrals, then interchanges the expectation with the integrals (by Tonelli's theorem since the integrand is non-negative), and finally rewrites the resulting expectation as a variance (since $\EE[\murand]=\mu^{(n)}$ and hence $\EE[\hat\PC^{(n)}]=\PC^{(n)}$).

In order to bound $\var^{\rm bdry}\left(\hat\PC^{(n)}(\vec{R},\vec{W})\right)$ we make use of the following result.
\begin{lemma}\label{lem:xuan53}
	For any finite interval $I\subset \R$, $\epsilon\in (0,1]$ and $C\geq 1$, if $\var\big(\murand(r)\big) \leq C n^{1-\epsilon}$ for all $n\in \Z_{\geq 1}$ then there exists a sequence $\{c^{(n)}\}_{n\in \Z_{\geq 1}}$ of non-negative reals such that $c^{(n)}\to 0$ as $n\to \infty$ and such that for any $k\in \Z_{\geq 0}$, $S,T\in I$ with $S<T$, $X,Y\in \R_{\geq 0}$ and $\vec{R},\vec{W}\in \R_{\geq 0}^k$ we have
	\begin{equation}\label{eq:Xuan53}
		\var^{\rm bdry}\left(\PC^{(n)}_{\murand}(S,X;\vec{R},\vec{W};T,Y)\right) \leq \big((1+c^{(n)})^{k+1}-1\big) \left(\PC^{(n)}_{\mu^{(n)}}(S,X;\vec{R},\vec{W};T,Y)\right)^2.
	\end{equation}
	
\end{lemma}
\begin{proof} We start with the case $k=0$. We may estimate the second moment of $\PC^{(n)}_{\murand}(S,X;T,Y)$, using results on the coincidence local time of couples of random walks. This is done in \cite[Lemma 5.3]{wu2018intermediate}, which,  rewritten in our notation states that
	\begin{equation}\label{eq:Xuan53k0}
		\var^{\rm bdry}\left(\PC^{(n)}_{\murand}(S,X;T,Y)\right) \leq c^{(n)} \left(\PC^{(n)}_{\tilde\mu^{(n)}}(S,X;T,Y)\right)^2
	\end{equation}
	where $\tilde\mu^{(n)} = \min\{0,2\mu^{(n)}-\mu^{(n)} n^{-1/2}\}$. Notice that $\tilde\mu^{(n)}  \geq \mu^{(n)}$. From the definition of $\PC^{(n)}_{\mu}(S,X;T,Y)$
	in terms of $p_{\boldsymbol{\X}}(s,x;t,y)$ from \eqref{eq:pxsxty} it is clear that for $\mu\geq \nu$, $\PC^{(n)}_{\mu}(\cdot,\cdot;\cdot,\cdot)\leq \PC^{(n)}_{\nu}(\cdot,\cdot;\cdot,\cdot)$.
	Thus $\PC^{(n)}_{\tilde\mu^{(n)}}(S,X;T,Y)\leq \PC^{(n)}_{\mu^{(n)}}(S,X;T,Y)$ which, combined with \eqref{eq:Xuan53k0} implies the $k=0$ case of \eqref{eq:Xuan53}. To prove the general result we use the following: For $U_1,\ldots, U_{k+1}$ independent  with finite variance,
	\begin{align}\label{eq:uprods}
		\var\left(\prod_{i=1}^{k+1} U_i\right)=
		\sum_{\substack{S\subset \{1,\ldots, k+1\}\\S\neq \emptyset }} \prod_{i\in S} \var(U_i) \prod_{i\notin S} (\EE[U_i])^2.
	\end{align}
	For $i\in \{1,\ldots, k+1\}$ assign $U_i = \PC^{(n)}_{\murand}(R_{i-1},W_{i-1};R_i,W_i)$ (recalling the convention $R_0=S, R_{k+1}=T, W_0=X$ and $W_{k+1}=Y$) and note that these are independent and that $\EE[U_i] = \PC^{(n)}_{\mu^{(n)}}(R_{i-1},W_{i-1};R_i,W_i)$. Applying the $k=0$ case of \eqref{eq:Xuan53} implies that $\var(U_i) \leq c^{(n)} (\EE[U_i])^2$. Combined with \eqref{eq:uprods} this implies \eqref{eq:Xuan53} for general $k$.
\end{proof}

Applying Lemma \ref{lem:xuan53} to the final line of \eqref{eq:ZZhat2} (with the variance condition verified by appealing to the second line of \eqref{eq:thmassumption1}) we arrive at
\begin{multline*}
	\EE\big[\big(\hat{\ZZ}^{(n)}-\ZZ^{(n)}\big)^2\big]^{1/2} \\ \leq \sum_{k=0}^{n(T-S)}(\sigma \beta)^{k} \big((1+c^{(n)})^{k+1}-1\big)^{1/2}
	||\PC^{(n)}_{\mu^{(n)}}(S,X;\cdot,\cdot;T,Y)\cdot 2^{\mathbf{1}_{Y=0}}||_{L^2(S,T,k)}.
\end{multline*}
Using \eqref{eq:xuan1}, along with the fact that  $\max(T-S,2k/n) \leq 2(T-S)$ for $k\in [0,n(T-S)]$, the above inequality yields
\begin{equation}\label{eq:ZZdecay}
	\EE\big[\big(\hat{\ZZ}^{(n)}-\ZZ^{(n)}\big)^2\big]^{1/2} \leq \sum_{k=0}^{n(T-S)}\frac{C^{k} \big((1+c^{(n)})^{k+1}-1\big)^{1/2}}{\sqrt{\Gamma((k+1)/2)}}
	(T-S)^{\frac{k}{4}-\frac{1}{2}}e^{-\frac{(X-Y)^2}{4C(T-S)}},
\end{equation}
for some suitably large constant $C>0$. This series is dominated for all $n$ by an infinite summable series (where $c^{(n)}$ is replaced by $\sup_n c^{(n)}<\infty$). As $n\to \infty$ each term goes to zero (since $(1+c^{(n)})^{k+1}-1\to 0$ due to $c^{(n)}\to 0$) and thus by the dominated convergence theorem, the entire summation goes to zero as well. This implies \eqref{eq:varzero} as desired to prove the random boundary version of the sheet convergence in finite dimensional distributions.

Turning to the random boundary extension of part (2), it suffices to show that
\begin{equation}\label{eq:varinitz}
	\var\left(\varnint dX \ZZ^{(n)}(X) \ZZ^{(n)}_{\murand,\beta}(0,X;T,Y) -\varnint dX \ZZ^{(n)}(X) \ZZ^{(n)}_{\mu^{(n)},\beta}(0,X;T,Y) \right)^{1/2}\to 0
\end{equation}
as  $n\to\infty$.
The argument inside of the variance on the left-hand side above is centered and hence the variance is the same as the expected value of the squared argument. Thus,
\begin{align*}
	\mathrm{LHS}\eqref{eq:varinitz} &\leq \varnint dX \EE\left[\big(\ZZ^{(n)}(X)\big)^2\left(\ZZ^{(n)}_{\murand,\beta}(0,X;T,Y)-\ZZ^{(n)}_{\mu^{(n)},\beta}(0,X;T,Y)\right)^2\right]^{1/2}\\
	&\leq \varnint dX \EE\!\!\left[\big(\ZZ^{(n)}(X)\big)^2\right]^{1/2} \!\!\!\EE\left[\!\left(\ZZ^{(n)}_{\murand,\beta}(0,X;T,Y)-\ZZ^{(n)}_{\mu^{(n)},\beta}(0,X;T,Y)\right)^2\right]^{1/2}\\
	&\leq C^{(n)}\varnint dX \EE\left[\big(\ZZ^{(n)}(X)\big)^2\right]^{1/2} e^{-\frac{(X-Y)^2}{2C(T-S)}}
\end{align*}
where $C^{(n)}>0$ goes to zero as $n\to \infty$. The first inequality above is from the triangle inequality, the second from the independence of the initial data from the sheet partition functions, and the final from substitution \eqref{eq:ZZdecay} along with the considerations in the paragraph after that equation. Observe that the summation in the final line of the above displayed equation is finite owing to the assumption \eqref{eq:expgrowthinitial} on the initial data. Since $C^{(n)}\to 0$ as $n\to \infty$, this implies \eqref{eq:varinitz} as desired. The extension of the result in the final paragraph in part (2) to random boundary weights proceeds in the same manner, and hence we do not repeat it here.

\section{Proof of Theorem \ref{prop:KPZconvergence}}
\label{sec:proofconvergence}
Recall from \eqref{eq:Znuv} that
$$e^{\HH^{(n)}_{u,v}(T,Y)}:= \sqrt{n}^{nT +n^{1/2} Y} \frac{\Zpolymer_{u,v;\alpha^{(n)}}^{\rm stat}\left(\frac{n T}{2}+n^{1/2}Y+2, \frac{n T}{2}+2 \right)}{\wweight_{1,1}\wweight_{2,2}}$$
where $\alpha^{(n)} =\frac{1}{2}+\sqrt{n}$ and $u,v\in \R$ with $u>v$ are fixed.

Since Definition \ref{def:SHEproper} defines the solution to \eqref{eq:KPZ} in terms of \ref{eq:SHEbeta}, to prove Theorem \ref{prop:KPZconvergence} we must show that $e^{\HH^{(n)}_{u,v}(T,Y)}$ converges in finite dimensional distributions in the $Y$ variable to $\ZZ_{\mu,1}(T,Y)$  where $\mu=u-1/2$ and $\beta=1$ with initial data $\ZZ_{\mu,1}(0,X) = e^{\HH_{u,v}(X)}$.

To do this, we will apply Theorem \ref{thm:SHEsheetlimit}, hence the task below is to show that our log-gamma polymer can be fit into the framework in Section \ref{sec:discreteconv}. This requires a bit of massaging since the weights of the log-gamma polymer are not of the form $1+\beta\omega$ and since the paths there are up-right as opposed to simple random walk paths. After this massaging, we arrive at a matching to the type of partition function denoted by $z^{\tinydiagup}_{\boldsymbol{\X},\boldsymbol{\omega},\beta,z^{\tinydiagup}}(t,y)$ from \eqref{eq:ztinydiag}. Note that the matching of the initial data along with necessary moment growth bounds was already addressed in Lemma \ref{prop:limitinitialcondition}.

Let us start by explaining how to fit the log-gamma polymer into the framework in Section \ref{sec:discreteconv}.
Let $\overline{\wweight} :=\EE[\wweight_{4,3}]$ denote the mean of a generic log-gamma polymer bulk weight, represented here by $\wweight_{4,3}$. Recall that $\wweight_{i,i}\sim \Gammainv(\alpha+u)$ for $i\in \Z_{\geq 3}$ and $\wweight_{i,j}\sim \Gammainv(2\alpha)$ for $i,j\in\Z_{\geq 3}$ with $i>j$. Thus, $\overline{\wweight} = (2\alpha-1)^{-1}$. Define
\begin{equation}\label{eq:tildezuvyunscaled}
	\widetilde{\Zpolymer}_{u,v;\alpha}(t,y) := \left(\frac{1}{2\overline\wweight}\right)^{2t+y} \frac{\Zpolymer^{\rm stat}_{u,v}(t+y+2,t+2)}{\wweight_{1,1}\wweight_{2,2}}
\end{equation}
and observe that by letting $\widetilde{z}_{u,v;\alpha}(x):=\widetilde{\Zpolymer}_{u,v;\alpha}(0,x+1)$ we have that for $t\in \Z_{\geq 1}$,
\begin{equation}\label{eq:tildezvar}
	\widetilde{\Zpolymer}_{u,v;\alpha}(t,y) = \sum_{x=0}^{t+y-1}  \widetilde{\Zpolymer}_{u,v;\alpha}(x) \widetilde{\Zpolymer}_{u;\alpha}(x+3,3;t+y+2,t+2).
\end{equation}
Here, for $a,b,a',b'\in \Z_{\geq 3}$ with $a>b$, $a'>b'$, $a'\geq a$ and $b'\geq b$ (i.e., $(a,b)$ and $(a',b')$ can be connected via an up-right path)
$$
\widetilde{\Zpolymer}_{u;\alpha}(a,b;a',b') :=  \left(\frac{1}{2\overline{\wweight}}\right)^{a'-a+b'-b+1} \Zpolymer_{u;\alpha}(a,b;a',b').
$$
where $\Zpolymer_{u;\alpha}(a,b;a',b')$ denotes the log-gamma partition function starting at $(a,b)$ and ending at $(a',b')$. The choices of $(a,b)$ and $(a',b')$ above are such that the weights encountered by this point to point partition function are all either i.i.d. bulk weights with the same law as $\wweight_{4,3}$ or i.i.d. boundary weights with the same law as $\wweight_{3,3}$. Since these weights only depend on $u$ and $\alpha$, we have only included these subscripts in  $\Zpolymer_{u;\alpha}(a,b;a',b')$. Note that in \eqref{eq:tildezvar}, $ \widetilde{\Zpolymer}_{u,v;\alpha}(x)$ and $\widetilde{\Zpolymer}_{u;\alpha}(x+3,3;t+y+2,t+2)$ are independent of each other since $ \widetilde{\Zpolymer}_{u,v;\alpha}(x)$ only involves weights in the first two rows of the log-gamma polymer whereas $\widetilde{\Zpolymer}_{u;\alpha}(x+3,3;t+y+2,t+2)$ only involves weights in the third and higher rows.

Observe that we can rewrite
$$
\tilde\Zpolymer_{u;\alpha}(a,b;a',b') = \frac{1}{2} \cdot \left(\frac{\wweight_{a'b'}}{\overline{\wweight} 2^{\mathbf{1}_{a'=b'}}}\right)\cdot  2^{\mathbf{1}_{a'=b'}} \!\!\!\!\!\!\!\!\!\!\!\sum_{\pi:(a,b)\to(a',b')} 2^{-|\{(i,j)\in \pi^*:i\neq j\}|}\!\!\!\!\!\!\! \prod_{(i,j)\in \pi^*, i\neq j} \frac{\wweight_{i,j}}{\overline{\wweight}} \!\!\!\prod_{(i,i)\in \pi^*} \frac{\wweight_{i,i}}{2\overline{\wweight}}
$$
where the sum is over all up-right lattice paths and $\pi^*$ represents the set of all lattice points in $\pi$ except the terminal one at $(a',b')$.

Now comes the key distributional identity to make our matching with the notation of  Section \ref{sec:discreteconv}.
As a process in $x\in \Z_{\geq 0}$, $t\in\Z_{\geq 1}$ and $y\in \Z_{\geq 0}$,
$$
\Zpolymer_{u;\alpha}(x+3,3;t+y+2,t+2) \overset{(d)}{=}  \frac{1}{2} \cdot \left(\frac{\omega(2t+y-2,y)}{\overline{\wweight} 2^{\mathbf{1}_{y=0}}}\right)\cdot  2^{\mathbf{1}_{y=0}}  z_{\boldsymbol{\X},\boldsymbol{\omega},\beta}(x,x;2t+y-2,y),
$$
provided that $\boldsymbol{\X},\boldsymbol{\omega}$ and $\beta$ satisfy that the $\X(r)$ are all i.i.d. and likewise the $\omega(r,w)$ are all i.i.d. with $\beta\in \R_{>0}$ chosen so that
\begin{gather*}
	1+\beta \omega(r,w)  \overset{(d)}{=} \frac{\wweight_{4,3}}{\bar{\wweight}}\sim (2\alpha-1) \Gammainv(2\alpha),\\
	\X(r)  \overset{(d)}{=} \frac{\wweight_{3,3}}{2\overline{\wweight}}\sim \frac{2\alpha-1}{2} \Gammainv(\alpha+u) .
\end{gather*}
This matching relies on the fact that the factor $2^{-|\{(i,j)\in \pi^*:i\neq j\}|}$ exactly matches the weighting of paths that arises in defining the reflected random walk measure. The weight corresponding to the endpoint has been left out above to make the match with our earlier notation as well.

The above matching along with \eqref{eq:tildezvar} means that, as a process in $t\in \Z_{\geq 1}$ and $y\in \Z_{\geq 0}$,
\begin{multline}\label{eq:matchingZs}
	\widetilde{\Zpolymer}_{u,v;\alpha}(t,y)   \\ \overset{(d)}{=} \frac{2^{\mathbf{1}_{Y=0}}}{2} \cdot \left(\mathbf{1}_{y=0}\X(2t+y-2) + \mathbf{1}_{y>0}  \big(1+\beta \omega(2t+y-2,y)\big) \right)\cdot  z^{\tinydiagup}_{\boldsymbol{\X},\boldsymbol{\omega},\beta,z^{\tinydiagup}}(2t+y-2,y)
\end{multline}
where the initial data $z^{\tinydiagup}(\cdot)=\widetilde{z}_{u,v;\alpha}(\cdot)$ is independent of the boundary and bulk weights $\boldsymbol{\X},\boldsymbol{\omega}$.

Returning to $e^{\HH^{(n)}_{u,v}(T,Y)}$, observe that for $T,Y$ such that $nT/2,n^{1/2}Y\in \Z_{\geq 0}$ we have
$$
e^{\HH^{(n)}_{u,v}(T,Y)} = \widetilde{\Zpolymer}_{u,v;\alpha^{(n)}}(nT/2,n^{1/2}y)
$$
where $\alpha^{(n)}=n^{1/2} + 1/2$. For other values of $T,Y$, we defined $e^{\HH^{(n)}_{u,v}(T,Y)}$ by interpolating. It is easy to see that under this scaling of $\alpha^{(n)}$, the term
$$\left(\mathbf{1}_{y=0}\X(2t+y-2) + \mathbf{1}_{y>0}  \big(1+\beta \omega(2t+y-2,y)\big) \right)$$ in \eqref{eq:matchingZs} converges to $1$ in probability as $n\to \infty$. Therefore,  in light of \eqref{eq:matchingZs}, if we can show
\begin{equation}\label{eq:lefttodo}
	\frac{2^{\mathbf{1}_{Y=0}}}{2} z^{\tinydiagup}_{\boldsymbol{\X},\boldsymbol{\omega},\beta,z^{\tinydiagup}}(nT+n^{1/2}Y-2,n^{1/2}Y)\Rightarrow \ZZ_{\mu,1}(T,Y)
\end{equation}
(where $\ZZ_{\mu,1}(T,Y)$ solves \eqref{eq:SHEbeta} with $\beta=1$ and initial data $\ZZ_{\mu,1}(0,X) = e^{\HH_{u,v}(X)}$)
in the sense of  finite dimensional distribution, then it will immediately follow that \break  $e^{\HH^{(n)}_{u,v}(T,Y)}\Rightarrow \ZZ_{\mu,1}(T,Y)$ as well, thus proving the theorem.

In order to prove \eqref{eq:lefttodo}, we will use the convergence of partition function with initial data and random boundary weights that is provided to us by part (3) of Theorem \ref{thm:SHEsheetlimit}. In order to do so, we must verify a number of conditions on the convergence of the initial data, as well as properties of the bulk weights, inverse temperature, and boundary weights.

For the initial data, we must show that $\ZZ^{(n)}(X):= z^{\tinydiagup(n)}(n^{1/2} X)$ (where the $(n)$ signifies the dependence on the scaling with $n$) converges as a process to $e^{\HH_{u,v}(X)}$ and satisfies \eqref{eq:expgrowthinitial}, i.e., that for some $a>0$, $\sup_{n\in \Z_{\geq 1}}\sup_{X\in \R_{\geq 0}} e^{-aX} \EE[\ZZ^{(n)}(X)^2] <\infty$. Both of these properties of the initial data were shown earlier in Lemma \ref{prop:limitinitialcondition}.

For the bulk  weights we must verify \eqref{eq:thmassumption2}. Above, we explained that $1+\beta \omega(r,w)\sim (2\alpha-1) \Gammainv(2\alpha)$ and with $\alpha=\alpha^{(n)}$ this implies  $$1+\beta^{(n)} \omega^{(n)}(r,w)\sim 2n^{1/2}  \Gammainv(2n^{1/2} +1).$$ Now, letting $\beta^{(n)}= n^{-1/4}/\sqrt{2}$ (which leads to taking $\beta=1$ in \eqref{eq:SHEbeta}) it is easily checked via \eqref{eq:inversegammamoments} that
$\EE[1+\beta  \omega^{(n)}(r,w)] = 1$ and hence $\EE[\omega^{(n)}(r,w)]=0$. Similarly we see that $\var\big(\beta^{(n)}\omega^{(n)}(r,w)\big) = 4n \var\big( \Gammainv(2n^{1/2} +1)\big)$ from which it follows that  $\var\big(\omega^{(n)}(r,w)\big) = \frac{2n^{1/2}}{2n^{1/2}-1}$ which converges to $1$ as needed. To establish the 8th moment bound, we compute the $N$th moment of $\omega^{(n)}(r,w)$ using \eqref{eq:inversegammamoments} as 
\begin{align*} 
	\EE\left[\omega^{(n)}(r,w)^N\right] &= \EE\left[ \left((\sqrt{2}n^{1/4}) (2n^{1/2}\Gammainv(2n^{1/2} +1)-1) \right)^N\right]\\ &=  (2\sqrt{n})^{N/2} \sum_{k=0}^N \binom{N}{k}(-1)^{N-k} \frac{(2\sqrt{n})^k}{2\sqrt{n}(2\sqrt{n}-1)\dots (2\sqrt{n}-k+1)}\\
\end{align*}
It may be shown that for any fixed $N$, this quantity behaves as $C_N+O(1/\sqrt{n})$
as $n$ goes to infinity, where $C_N=(N-1)!!$ is the $N$th moment of a standard Gaussian. To justify it, one may express the moments in terms of the cumulants, and use the fact that the cumulant generating function admits a simple expression in terms of Gamma functions. Details of this argument may be found in \cite[Corollary 2.5]{krishnan2018tracy}, where it is proved in particular that the moments of $\omega^{(n)}(r,w)$ are uniformly bounded \cite[Eq. (18)]{krishnan2018tracy}. Thus, assumption \eqref{eq:thmassumption2} holds holds. 


For the boundary weights, we must verify \eqref{eq:thmassumption1}. Observe that from \eqref{eq:inversegammamoments},
$$\EE[\X^{(n),\mu}(r)] = \frac{n^{1/2}}{n^{1/2} + \mu} = 1-n^{-1/2}\mu + O(n^{-1}),$$ thus implying that
$\lim_{n\to \infty} \EE\big[n^{1/2}\big(1-\X^{(n),\mu}(r)\big)\big]= \mu$. Similarly, observe that  by \eqref{eq:inversegammamoments},
$$\var\big(\X^{(n),\mu}(r)\big) = \frac{n}{(n^{1/2} +\mu)^2(n^{1/2} +\mu-1)} = O(n^{-1/2})$$ and hence there exists and $\epsilon\in (0,1]$ and $K>0$ such that for all $n\in \Z_{\geq 1}$, $\var\big(\X^{(n),\mu}\big) \leq K n^{-\epsilon}$. Thus, we have verified all of the conditions necessary to apply  part (3) of Theorem \ref{thm:SHEsheetlimit}. This implies \eqref{eq:lefttodo} and completes our proof.

\renewcommand{\emph}[1]{\textit{#1}}
\bibliography{mainbiblio}
\bibliographystyle{goodbibtexstyle} 
\end{document}